\documentclass[11pt,a4paper]{amsart}		% amsart class che mi piace di piu'

\usepackage[utf8]{inputenc}					% for unicode support

\usepackage{lmodern}						% font
\usepackage[english]{babel}					% babel

\usepackage{amsmath,amsfonts,amssymb,amsthm}
\usepackage{tikz-cd}

\usepackage[left=3cm,right=3cm,top=3cm,bottom=3cm]{geometry}

\usepackage{mathtools}						% uncomment to tag only the referred equations (post-production)
\DeclareMathAlphabet{\mathbbold}{U}{bbold}{m}{n}	% nuovo alfabeto per i numeri in bold.

\theoremstyle{plain}
\newtheorem{thm-intro}{Theorem}
\newtheorem{cor-intro}[thm-intro]{Corollary}

\newtheorem*{theorem*}{Theorem}
\newtheorem{theorem}{Theorem}[section]
\newtheorem{prop}[theorem]{Proposition}

\newtheorem{lemma}[theorem]{Lemma}
\theoremstyle{definition}
\newtheorem{definition}[theorem]{Definition}
\theoremstyle{remark}
\newtheorem{rmk}{Remark}[section]
\newtheorem{example}[rmk]{Example}

\DeclareMathOperator{\spn}{\mathrm{span}}
\DeclareMathOperator{\rank}{\mathrm{rank}}
\DeclareMathOperator{\tr}{\mathrm{Tr}}
		
\DeclareMathOperator{\supp}{supp}

\DeclareMathOperator{\Hess}{Hess}
\DeclareMathOperator{\Ric}{Ric}
\DeclareMathOperator{\re}{Re}
\DeclareMathOperator{\im}{Im}

\newcommand{\N}{\mathbb{N}}					% real numbers
\newcommand{\R}{\mathbb{R}}					% real numbers
\newcommand{\distr}{\mathcal{D}}			% the distribution
\renewcommand{\epsilon}{\varepsilon}		% varepsilon
\newcommand{\dive}{\mathrm{div}}			% divergence

\renewcommand{\star}{\mathbf{H}}

\newcommand{\vol}{\mathrm{vol}}
\newcommand{\pot}{V}
\newcommand{\op}{H}
\newcommand{\Veff}{V_{\mathrm{eff}}}
\newcommand{\loc}{\mathrm{loc}}
\newcommand{\comp}{\mathrm{comp}}
\newcommand{\dom}{D}

\usepackage{hyperref}
\usepackage{xcolor}
\hypersetup{
    colorlinks,
    linkcolor={red!50!black},
    citecolor={blue!50!black},
    urlcolor={blue!80!black}
}

\title{Quantum confinement on non-complete Riemannian manifolds}

\author[Dario Prandi]{Dario Prandi$^\flat$}
\address{$^\flat$ CNRS, Laboratoire des Signaux et Systèmes, CentraleSup\'elec, Gif-sur-Yvette}
\email{\href{mailto:dario.prandi@l2s.centralesupelec.fr}{dario.prandi@l2s.centralesupelec.fr}}

\author[Luca Rizzi]{Luca Rizzi$^\sharp$}
\address{$^\sharp$ Univ. Grenoble Alpes, IF, F-38000 Grenoble, France \newline 
CNRS, IF, F-38000 Grenoble, France (current institution) \newline 
Inria, team GECO \& CMAP, \'Ecole Polytechnique, CNRS, Universit\'e Paris-Saclay, Palaiseau, France (past institution)}
\email{\href{mailto:luca.rizzi@univ-grenoble-alpes.fr}{luca.rizzi@univ-grenoble-alpes.fr}}

\author[Marcello Seri]{Marcello Seri$^\dagger$}
\address{$^\dagger$ Department of Mathematics and Statistics, University of Reading, Reading, UK}
\email{\href{mailto:m.seri@ucl.ac.uk}{m.seri@ucl.ac.uk}}

 
\begin{document}

\begin{abstract}
We consider the quantum completeness problem, i.e.\ the problem of confining quantum particles, on a non-complete Riemannian manifold $M$ equipped with a smooth measure $\omega$, possibly degenerate or singular near the metric boundary of $M$, and in presence of a real-valued potential $V\in L^2_\loc(M)$.
The main merit of this paper is the identification of an intrinsic quantity, the effective potential $\Veff$, which allows to formulate simple criteria for quantum confinement. Let $\delta$ be the distance from the possibly non-compact metric boundary of $M$. A simplified version of the main result guarantees quantum completeness if $V\ge -c\delta^2$ far from the metric boundary and 
\[
	\Veff+V\ge \frac3{4\delta^2}-\frac{\kappa}{\delta}, \qquad \text{close to the metric boundary}.
\]
These criteria allow us to: (i) obtain quantum confinement results for measures with degeneracies or singularities near the metric boundary of $M$; (ii) generalize the Kalf-Walter-Schmincke-Simon Theorem for strongly singular potentials to the Riemannian setting for any dimension of the singularity; (iii) give the first, to our knowledge, curvature-based criteria for self-adjointness of the Laplace-Beltrami operator; (iv) prove, under mild regularity assumptions, that the Laplace-Beltrami operator in almost-Riemannian geometry is essentially self-adjoint, partially settling a conjecture formulated in \cite{BL-LaplaceBeltrami}.
\end{abstract}

\maketitle

\setcounter{tocdepth}{1}
\tableofcontents % Quando abbiamo finito lo mettiamo cosi', per ora mi piace avere un'overview della struttura

\section{Introduction}

Let $(M,g)$ be a smooth Riemannian manifold of dimension $n\geq 1$, equipped with a smooth measure $\omega$. That is, $\omega$ is defined by a smooth, positive density, not necessarily the Riemannian one. Given a real-valued potential $V\in L^2_\loc(M)$, the evolution of a quantum particle is described by a wave function $\psi\in L^2(M)$, obeying the Schr\"odinger equation:
\begin{equation}\label{eq:schr}
	i\partial_t \psi = H \psi,
\end{equation}
where $H$ is the operator on $L^2(M)$ defined by,
\begin{equation}
	H = -\Delta_\omega + V, \qquad \dom(H) = C^\infty_c(M).
\end{equation}
Here, $\Delta_\omega = \dive_\omega \circ \nabla$ is the weighted Laplace-Beltrami on functions, computed with respect to the measure $\omega$. When $\omega = \vol_g$ is the Riemannian volume, then $\Delta_\omega = \Delta$ is the classical Laplace-Beltrami operator.

The operator $H$ is symmetric and densely defined on $L^2(M)$. The problem of finding its self-adjoint extensions has a long and venerable history, dating back to Weyl at the beginning of the 20th century. From the mathematical viewpoint, by Stone Theorem, any self-adjoint extension of $H$ generates a strongly continuous unitary semi-group on $L^2(M)$, which produces solutions of \eqref{eq:schr}, starting from a given initial condition $\psi_0 \in L^2(M)$. When multiple self-adjoint extensions are available, such an evolution is no longer unique. Concretely, when $M \subset \R^n$ is a bounded region of the Euclidean space, different self-adjoint extensions correspond to different boundary conditions. For example, one can have repulsion or reflection, up to a complex phase, at $\partial M$, leading to different physical evolutions.

On the other hand, when $H$ is essentially self-adjoint, that is, it admits a unique self-adjoint extension, there is no need to fix any boundary condition, nor to precisely describe the domain of the extension. The physical interpretation of this fact is that quantum particles, evolving according to \eqref{eq:schr}, are naturally confined to $M$. For this reason, the essential self-adjointness of $H$ is referred to as \emph{quantum completeness} or \emph{quantum confinement}.

For geodesically complete Riemannian manifolds, there is a well developed theory, giving sufficient conditions on the potential $V$ to ensure quantum completeness. In particular, when $V \geq 0$, then $H$ is essentially self-adjoint. We refer to the excellent \cite{BMS-manifolds}, which contains almost all results on the essential self-adjointness of Schr\"odinger-type operators on vector bundles over complete Riemannian manifolds.

Less understood is the case of non-complete Riemannian manifolds, that is, when geodesics (representing trajectories of classical particles) can escape any compact set in finite time. For bounded domains in $\R^n$, this problem has been thoroughly discussed in \cite{Nenciu2008}, giving refined conditions on the potential for the essential self-adjointness of $H = - \Delta + V$, where $\Delta$ is Euclidean Laplacian. The recent work \cite{MilatovicTruc} contains also quantum completeness results for Schr\"odinger type operators on vector bundles over open subsets of Riemannian manifolds, under strong assumptions on the potential at the metric boundary. Related results, for a magnetic Laplacians and no external potential, can be found in \cite{NenciuDisk} (for the Euclidean unit disk), and in \cite{DeVerdiere2009a} (for bounded domains in $\R^n$ and some Riemannian structures). Finally, we mention \cite{Masamune2005}, where conditions for quantum completeness of the Laplace-Beltrami operator on a non-complete Riemannian manifold are given in terms of the capacity of the metric boundary.

We stress that, in all the above cases, the explosion of the potential $V$ or the magnetic field close to the metric boundary plays an essential role. An interesting fact is that \emph{even in absence of external potential or magnetic fields}, the Laplace-Beltrami operator on a non-complete Riemannian manifold can be essentially self-adjoint, leading to purely geometric confinement. Let us discuss a simple example, the \emph{Grushin metric},
\begin{equation}
g = dx\otimes dx + \frac{1}{x^2} dy \otimes dy, \qquad \text{on }M= \R^2 \setminus \{x=0\}.
\end{equation}
This metric is not geodesically complete, as almost all geodesics starting from $M$ cross the singular region $\mathcal{Z} = \{x=0\}$ in finite time. The only exception is given by the negligible set of geodesics pointing directly away from $\mathcal{Z}$ with initial speed $\mathrm{sgn}(x) \partial_x$. Observe that the Riemannian measure $\vol_g = \frac{1}{|x|} dx dy$ explodes close to $\mathcal{Z}$. The corresponding Laplace-Beltrami operator is
\begin{equation}
\Delta = \partial_x^2 + x^2 \partial_y^2 - \frac{1}{x} \partial_x.
\end{equation}
This is a particular instance of almost-Riemannian structure (ARS). It is not hard to show that $\Delta$, with domain $C^\infty_c(\R^2\setminus\mathcal{Z})$, is essentially self-adjoint.

In \cite{BL-LaplaceBeltrami}, it is proved that the Laplace-Beltrami operator for $2$-dimensional, compact, orientable almost-Riemannian strctures (ARS), defined on the complement of the singular region, is essentially self-adjoint. The confinement of quantum particles on these structures is surprising, and in sharp contrast with the behaviour of classical ones which, following geodesics, almost always cross the singular region. It was thus conjectured that the Laplace-Beltrami is essentially self-adjoint for all ARS, of any dimension. Unfortunately, since the techniques used in \cite{BL-LaplaceBeltrami} are based on normal forms for ARS, which are not available in higher dimension, different tools are required to attack the general case.

Motivated by this problem, we investigate the essential self-adjointness of $H$ on non-complete Riemannian structures, with a particular emphasis on the connection with the underlying geometry. Our setting allows to treat in an unified manner many classes of non-complete structures, including, most importantly, those whose metric completion is \emph{not} a smooth Riemannian manifold (such as ARS), or not even a topological manifold (such as cones). In this general setting, we are able to apply and extend some techniques inspired by \cite{Nenciu2008,DeVerdiere2009a}, based on Agmon-type estimates and Hardy inequality, to yield sufficient conditions for self-adjointness. We remark that very recently, in \cite{NN-stoc}, the aforementioned techniques have been combined with the so-called Lioville property to prove sufficient conditions for stochastic (and quantum) confinement of drift-diffusion operators on domains of $\R^n$. 
An interesting perspective would then be to obtain geometric criteria for stochastic confinement on non-complete Riemannian manifolds, by combining these methods with the ones in this paper.

Since we are interested in conditions for purely geometrical confinement, the main thrust of the paper is the case $V \equiv 0$. Nevertheless, for completeness, we included the external potential in our main statement, even though this leads to some technicalities. The main novelty of our approach is the identification of an intrinsic function -- depending only on $(M,g)$ and the measure $\omega$ -- which we call the \emph{effective potential}:
\begin{equation}
\Veff = \left(\frac{\Delta_\omega \delta}{2}\right)^2 + \left(\frac{\Delta_\omega \delta}{2}\right)^\prime,
\end{equation}
where $\delta$ denotes the distance from the metric boundary, and the prime denotes the normal derivative. Under appropriate conditions on $\Veff$ -- typically, a sufficiently fast blow-up at the metric boundary -- one can infer the essential self-adjointness of $H$ even in absence of any external potential (see Section~\ref{s:main}).

We observe that the explosion of the measure $\omega$ close to the metric boundary (as it happens for the Grushin metric), is \emph{not} a necessary condition for essential self-adjointness of $\Delta_\omega$. Indeed, the formula for $\Veff$ shows that not only the explosion of $\omega$, but also of its first and second derivatives, plays a role in the confinement. In particular, one can attain quantum completeness in presence of measures that vanish sufficiently fast close to the metric boundary. This is the topic of Section~\ref{s:degorder}, in the framework of quantum completeness induced by singular or degenerate measures.

Another application of our main result, this time in presence of an external potential $\pot$, is the generalization of the Kalf-Walter-Schmincke-Simon Theorem for strongly singular potentials to the Riemannian setting for any dimension of the singularity. This is studied in Section~\ref{s:singular}, and extends the results of \cite{Brusentsev}, obtained in the Euclidean setting, and of \cite{DonnelyGarofalo97}, for point-like singularities on Riemannian manifolds. See also the recent work \cite{Nenciu-Deficiency}, where a particular emphasis is put on the study of deficiency indices in the Euclidean setting.

Recall that, if $\omega = \vol_g$ is the Riemannian measure, then $\Delta_\omega \delta$ is proportional to the mean curvature of the level sets of the distance from the metric boundary $\delta$. Hence, the very existence of the above formula for $\Veff$ sheds new light on the relation between curvature and essential self-adjointness. In particular, via Riccati comparison techniques, this connection leads to the first, to our knowledge, curvature-based criteria for quantum completeness (see Section~\ref{s:curvature}).

Finally, and most important, in Section~\ref{s:arg} we prove that our machinery can be applied to the almost-Riemannian setting. Then, under mild assumptions on the underlying geometry, we settle the almost-Riemannian part of the Boscain-Laurent conjecture, proving that the Laplace-Beltrami operator is essentially self-adjoint for regular ARS. We then discuss the non-regular case, describing the limitation of our techniques and exhibiting examples of ARS where we are not able to infer the essential self-adjointness of the Laplace-Beltrami.

In the remainder of the section we provide a panoramic view of the main results.

\subsection{Assumption on the metric structure}

In order to describe precisely the behavior of $H$ near the ``escape points'' of $M$, we need an assumption on the metric structure $(M,d)$ induced by the Riemannian metric $g$. For this purpose, we let $(\hat M, \hat d)$ be the metric completion of $(M,d)$ and $\partial\hat M:=\hat M\setminus M$ be the \emph{metric boundary}. The \emph{distance from the metric boundary} $\delta:M\to [0,+\infty)$ is then 
\begin{equation}
	\delta(p) := \inf\left\{ \hat d(p,q)\mid q\in \partial \hat M \right\}.
\end{equation}
We assume the following.
\begin{itemize}
\item[$(\star)$] There exists $\varepsilon >0$ such that $\delta$ is $C^2$ on $M_\varepsilon := \{ 0<\delta \leq \varepsilon\}$.
\end{itemize}
Under this assumption, as shown in Lemma~\ref{l:dist}, there exists a $C^1$-diffeomorphism $M_\varepsilon \simeq (0,\varepsilon] \times X_\varepsilon$, where $X_\varepsilon=\{\delta=\varepsilon\}$ is a $C^2$ embedded hypersurface, such that $\delta(t,x)=t$.

Assumption $(\star)$ is verified when $M=N \setminus \mathcal{Z}$, where $N$ is a smooth manifold, $\mathcal{Z} \subset N$ is a $C^2$ submanifold of arbitrary dimension, and $g$, $\omega$ are possibly singular on $\mathcal{Z}$. As already mentioned, $(\star)$ holds in more general situations, in which the metric completion $\hat M$ need not be a Riemannian manifold (e.g.\ to ARS), or even a topological manifold (e.g.\ to cones).

\subsection{Effective potential and main result}

Here and thereafter, for any function $f :M \to \R$, the symbol $f^\prime$ represents the normal derivative with respect to the metric boundary, that is the derivative in the direction $\nabla\delta$:
\begin{equation}
f' := df (\nabla\delta) = g(\nabla\delta,\nabla f).
\end{equation}

We start by introducing the main object of interest of the paper, which allows to characterize the effect of the metric boundary on the self-adjointness of $H$ taking into account the interaction of the Riemannian structure with the measure. 
\begin{definition}
The \emph{effective potential} $\Veff : M_\varepsilon \to \R$ is the continuous function\footnote{The fact that $\Veff$ is well defined and continuous is proven in Lemma~\ref{l:dist} and Proposition~\ref{l:pot}.},
\begin{equation}\label{eq:potential}
\Veff:=\left(\frac{\Delta_\omega \delta}{2}\right)^2 + \left(\frac{\Delta_\omega \delta}{2}\right)^\prime.
\end{equation}
\end{definition}

The main result of the paper is the following criterion for essential self-adjointness of $H$. Standard choices for the function $\nu$ appearing in its statement are, e.g., the distance $\delta$ from the metric boundary, or the Riemannian distance $d(p,\cdot)$ from a fixed point $p\in M$.

\begin{thm-intro}[Main quantum completeness criterion]\label{t:main-intro}
Let $(M,g)$ be a Riemannian manifold satisfying $(\star)$ for $\varepsilon >0$. Let $\pot \in L^2_\loc(M)$. Assume that there exist $\kappa \geq 0$ and a Lipschitz function $\nu : M \to \R$ such that, close to the metric boundary, 
	\begin{equation}\label{eq:hypo-1-intro}
\Veff + \pot  \geq \frac{3}{4\delta^2}-\frac{\kappa}{\delta} -\nu^2, \qquad  \text{for $\delta \leq \varepsilon$}.
\end{equation}
Moreover, assume that there exist $\varepsilon' < \varepsilon$, such that,
\begin{equation}\label{eq:hypo-2-intro}
\pot  \geq - \nu^2, \qquad \text{for $\delta > \varepsilon'$}.
\end{equation}	
Then, $H=-\Delta_\omega+\pot$ with domain $C^\infty_c(M)$ is essentially self-adjoint in $L^2(M)$.

Finally, if $\hat{M}$ is compact, the unique self-adjoint extension of $H$ has compact resolvent. Therefore, its spectrum is discrete and consists of eigenvalues with finite multiplicity.
\end{thm-intro}
The very existence of the intrinsic formula \eqref{eq:potential} for the effective potential $\Veff$, providing a direct link between geometry and self-adjointness properties, is one of the most interesting results of this paper. Some remarks about $\Veff$ are in order.
\begin{rmk} By the generalized Bochner formula \cite[Eqs. 14.28, 14.46]{VillaniOldandNew}, we have
\begin{equation}\label{eq:potential2-intro}
\Veff=\frac{1}{4}\left( (\Delta_\omega \delta)^2 - 2 \|\Hess( \delta)\|_{\mathrm{HS}}^2 - 2 \Ric_\omega(\nabla\delta, \nabla\delta)\right),
\end{equation}
where, if $\omega = e^{-f} \vol_g$, then $\Ric_\omega:= \mathrm{Ric} + \Hess(f)$ is the Bakry-Emery Ricci tensor and $\| \cdot \|_{\mathrm{HS}}$ denotes the Hilbert-Schmidt norm. If $\omega = \vol_g$, the Bakry-Emery tensor is the standard Ricci curvature and $\Delta_{\vol_g} = \Delta$ is the Laplace-Beltrami operator. In this case, $\Veff$ is a function of the mean curvature $m=\Delta \delta$ of the level sets of $\delta$.
\end{rmk}

Since, in our view, the main interest of the paper is the case $V \equiv 0$, we point out the following immediate corollary of Theorem~\ref{t:main-intro}.

\begin{cor-intro}\label{c:main-nopot-intro}
	Let $(M,g)$ be a Riemannian manifold satisfying $(\star)$ for $\varepsilon >0$. Assume that there exist $\kappa \geq 0$ such that, 
	\begin{equation}
\Veff  \geq \frac{3}{4\delta^2}-\frac{\kappa}{\delta}, \qquad  \text{for $\delta \leq \varepsilon$}.
\end{equation}
Then, $\Delta_\omega$ with domain $C^\infty_c(M)$ is essentially self-adjoint in $L^2(M)$.
\end{cor-intro}

\subsection{Measure confinement}

The condition of Corollary~\ref{c:main-nopot-intro} reflects on the measure $\omega$ in a natural way, as discussed in Section~\ref{s:degorder}. Moreover this condition is sharp for measures with power behavior near the metric boundary, as shown in the following. Here, we identify $M_\varepsilon \simeq (0,\varepsilon] \times X_\varepsilon$, and denote points of $M$ as $p=(t,x)$, with $x \in X_\varepsilon$.

\begin{thm-intro}[Pure measure confinement]\label{t:degorder-intro}
Assume that the Riemannian manifold $(M,g)$ satisfies $(\star)$ for $\varepsilon >0$. Moreover, let $\omega$ be a smooth measure such that there exists $a\in\R$ and a reference measure $\mu$ on $X_\varepsilon$ for which
	\begin{equation}\label{eq:porwer-ord-intro}
		d\omega(t,x) = t^a\,dt\,d\mu(x),\qquad (t,x)\in(0,\varepsilon] \times X_\varepsilon.
	\end{equation}	
	Then, $\Delta_\omega$ with domain $C^\infty_c(M)$ is essentially self-adjoint in $L^2(M)$ if $a\ge 3$ or $a\le -1$.
\end{thm-intro}

The preceding result can be directly applied, choosing $\omega=\text{vol}_g$, to \emph{conic or anti-conic-type structures}. These are Riemannian structures that satisfy $(\star)$ for some $\varepsilon>0$ and such that their metric, under the identification $M_\varepsilon \simeq (0,\varepsilon] \times X_\varepsilon$, can be written as
	\begin{equation}\label{eq:coni}
		g|_{M_\varepsilon} = dt \otimes dt + t^{2\alpha}\,h, \qquad \alpha\in\R, 
	\end{equation}
where $h$ is some Riemannian metric on $X_\varepsilon$.

The above structures are \emph{cones} when $\alpha=1$ (see, e.g., \cite{MR530173}), \emph{metric horns} when $\alpha > 1$ (see \cite{Horns}) and \emph{anti-cones} when $\alpha< 0$ (see \cite{anticonic}).
For $n=2$ and $M = \R\times \mathbb \mathbb{S}^1$, the corresponding embedding in $\R^3$ for $\alpha\ge 1$ or $\alpha = 0$ are shown in Figure~\ref{fig:coni}.
For $-\alpha \in \mathbb N$ these structures are almost-Riemannian, see Section~\ref{s:arg}.

The measure of these structures is of the form \eqref{eq:porwer-ord-intro}, with $a= (n-1)\alpha$, hence we have the following generalization of a result in \cite{anticonic}.
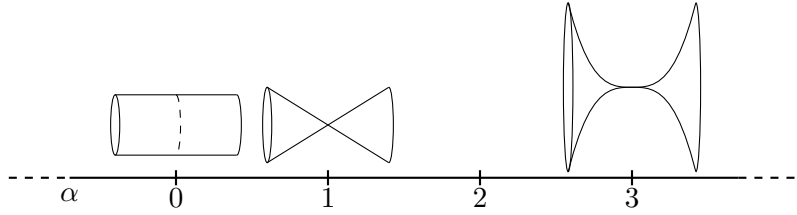
\begin{figure}
\begin{tikzpicture}
\draw[thick] (-1.4,0) -> (7.4,0);
\draw[thick, dashed] (-2.2,0) -> (-1.4,0);
\draw[thick, dashed] (7.4,0) -> (8.2,0);
\foreach \x in {3,2,1,0}
	\draw[thick] (2*\x,-.1) -> (2*\x,.1);

\node[below=.1] at (0,0) {$0$};
\node[below=.1] at (2,0) {$1$};
\node[below=.1] at (4,0) {$2$};	
\node[below=.1] at (6,0) {$3$};
\node[below=.1] at (-1.4,0) {$\alpha$};

\begin{scope}[shift = {(6,1.2)}, scale = .7, xscale=.6]
	\draw (-2,0) circle[x radius=.15, y radius = 1.6] ;
	\draw (2,1.6) arc[start angle = 90, end angle = -90, x radius=.15, y radius = 1.6] ;

	\draw plot[domain = -2:2] (\x,.2*\x^3); %.2*2^3 = 1.6
	\draw plot[domain = -2:2] (\x,-.2*\x^3);
\end{scope}

\begin{scope}[shift = {(2,.7)}, scale = .5, xscale = .8]
	\draw (-2,1) -> (2,-1);
	\draw (-2,-1) -> (2,1);

	\draw (-2,0) circle[x radius=.15, y radius = 1] ;
	\draw (2,1) arc[start angle = 90, end angle = -90, x radius=.15, y radius = 1] ;
\end{scope}

\begin{scope}[shift = {(0,.7)}, scale = .4, xscale = 1]
	\draw (-2,1) -> (2,1);
	\draw (-2,-1) -> (2,-1);

	\draw (-2,0) circle[x radius=.15, y radius = 1] ;
	\draw (2,1) arc[start angle = 90, end angle = -90, x radius=.15, y radius = 1] ;
	\draw[dashed] (0,1) arc[start angle = 90, end angle = -90, x radius=.15, y radius = 1] ;
\end{scope}
\end{tikzpicture}
	\caption{Depiction of the embeddings in $\R^3$ of the 2-dimensional structures on $\R\times \mathbb{T}^1$ with metric $g = dt^2+t^\alpha d\theta^2$, $\alpha\ge0$.} 
	\label{fig:coni}
\end{figure}

\begin{cor-intro}\label{c:coni}
	Consider a conic or anti-conic-type structure as in \eqref{eq:coni}.
	Then, $\Delta = \Delta_{\vol_g}$ is essentially self-adjoint in $L^2(M)$ if $\alpha\ge\tfrac3{n-1}$ or $\alpha\le-\tfrac{1}{n-1}$.
\end{cor-intro}

\begin{rmk}
The bounds of Theorem~\ref{t:degorder-intro} and Corollary~\ref{c:coni} are sharp. Indeed, the Laplace-Beltrami operator $-\Delta$ on $M = (0,+\infty) \times \mathbb{S}^1$ given by the global metric
\begin{equation}
g = dt \otimes dt + t^{2\alpha} d\theta \otimes d\theta,
\end{equation}
is essentially self-adjoint if and only if $\alpha \in (-\infty,-1]\cup [3,\infty)$. The proof of the ``only if'' part of this statement relies on the explicit knowledge of the symmetric solutions of $(-\Delta^* - \lambda) u = 0$ for this metric, and can be found, for example, in \cite{anticonic}.
\end{rmk}

\subsection{Strongly singular potentials}

A well known and classical result by Kalf-Walter-Schmincke-Simon \cite{Simon-SSingular} (see also \cite[Thm. X.30]{MR0493420}) states that, if $\pot = \pot_1 +\pot_2$ with $\pot_2 \in L^\infty(\R^n)$ and $\pot_1 \in L^2_\loc(\R^n \setminus \{0\})$ obeying
\begin{equation}
\pot_1 (z) \geq -\frac{n(n-4)}{4|z|^2},
\end{equation}
then $-\Delta + \pot$ is essentially self-adjoint on $C^\infty_c(\R^n  \setminus \{0\})$. The above theorem, in particular, implies that, starting from dimension $n \geq 4$, points are ``invisible'' from the point of view of a free quantum particle living in $\R^n$, i.e., with $V \equiv 0$.

This result has been generalized to the case of potentials singular along affine hypersurfaces of $\R^n$ in \cite{Maeda}, and for singularities along well-separated submanifold of $\R^n$ in \cite[Thm. 6.2]{Brusentsev}. In the Riemannian setting, to our best knowledge, the only result so far is \cite{DonnelyGarofalo97}, by Donnelly and Garofalo, for point-like singularities. See also \cite[Thm. 3]{MilatovicTruc}, where the authors obtain similar results for general differential operators on Hermitian vector bundles under assumptions implying $V\ge -c$ (that is, not strongly singular).

The method of effective potentials developed in this paper allows to obtain a generalization of the Kalf-Walter-Schmincke-Simon Theorem for potentials singular along arbitrary dimension submanifolds of complete Riemannian manifolds, proved in Section~\ref{s:singular}. We stress that, in the case of points -- i.e.\ dimension $0$ singularities -- condition \eqref{eq:conditionsingular-intro} is strictly weaker than the one in \cite[Thm. 2.5]{DonnelyGarofalo97}, allowing a stronger singularity of the potential.

\begin{thm-intro}[Kalf-Walter-Schmincke-Simon for Riemannian submanifolds]\label{t:quattro-autori-intro}
Let $(N,g)$ be a $n$-dimensional, complete Riemannian manifold. Let $\mathcal{Z}_i \subset N$, with $i \in I$, be a finite collection of embedded, compact $C^2$ submanifolds of dimension $k_i$ and denote by $d(\cdot,\mathcal{Z}_i)$ the Riemannian distance from $\mathcal{Z}_i$. Let $\pot \in L^2_\loc(N \setminus \mathcal{Z}_i)$ be a \emph{strongly singular potential}. That is, there exists $\varepsilon>0$ and a non-negative Lipschitz function $\nu : N \to \R$, such that,
\begin{enumerate}
	\item[(i)]  for all $i\in I$ and $p \in N$ such that $0< d(p,\mathcal{Z}_i)\leq \varepsilon$, we have
	\begin{equation}\label{eq:conditionsingular-intro}
		\pot(p) \geq -\frac{(n-k_i)(n-k_i-4)}{4 d(p,\mathcal{Z}_i)^2} - \frac{\kappa}{d(p,\mathcal{Z}_i)}- \nu(p)^2, \qquad \kappa\geq 0;
	\end{equation}
	\item[(ii)] for all $p \in N$ such that $d(p,\mathcal{Z}_i) \geq \varepsilon$ for all $i \in I$, we have
	\begin{equation}
		\pot(p) \geq - \nu(p)^2.
	\end{equation}
\end{enumerate}
Then, the operator $H = -\Delta + \pot$ with domain $C^\infty_c(M)$ is essentially self-adjoint in $L^2(M)$, where $M = N \setminus \bigcup_{i} \mathcal{Z}_i$, or any one of its connected components.
\end{thm-intro}
As a consequence of Theorem~\ref{t:quattro-autori-intro}, any submanifold of codimension $n-k \geq 4$ is ``invisible'' from the point of view of free quantum particles living on $N$ i.e., with $V \equiv 0$. This result is also sharp, in fact one can show that, if $n-k < 4$, the Laplace-Beltrami $H=-\Delta$ with domain $C^\infty_c(M)$ is not essentially-self adjoint.

\begin{rmk}
	Theorem~\ref{t:quattro-autori-intro} can be easily generalized to accommodate a countable number of singularities, under the assumption
	\begin{equation}
		\inf_{i\neq j} d(\mathcal Z_i,\mathcal Z_j) >0.
	\end{equation}
	Moreover, the compactness of the singularities can be removed, provided that the non-complete manifolds $N\setminus \mathcal Z_i$ satisfy $(\star)$ for each $i\in I$ and some fixed $\varepsilon>0$.
\end{rmk}

\subsection{Curvature-based criteria for self-adjointness}

In this section, we fix $\omega = \vol_g$, and investigate how the curvature of $(M,g)$ is related with the essential self-adjointness of the Laplace-Beltrami operator $\Delta = \Delta_{\vol_g}$. A crucial observation is that sectional curvature is not the only actor. This can be easily observed by considering, e.g., conic and anti-conic-type structures given by \eqref{eq:coni}. In this case, for all planes $\sigma$ containing $\nabla \delta$,
\begin{equation}\label{eq:sectionalsymmetry-intro}
	\mathrm{Sec}(\sigma) = -\frac{\alpha(\alpha-1)}{\delta^2}, \qquad \delta \leq \varepsilon,
\end{equation}
and Corollary~\ref{c:coni} implies the existence of non-self-adjoint and self-adjoint structures with exactly the same sectional curvature (e.g., take $n=2$ and $\alpha=2$ and $\alpha =-1$, respectively).

It turns out that the essential self-adjointness property of $\Delta$ is influenced also by the principal curvatures of the $C^2$ level sets $X_t=\{\delta=t\}$, $0<t\leq \varepsilon$, that is the eigenvalues of the \emph{second fundamental form}\footnote{Recall that the second fundamental form (or shape operator) of an hypersurface is well defined up to a sign, depending on the choice of the unit normal vector. In our case, the normal vector to $X_t$ is $-\nabla \delta$.} $H(t)$ of $X_t$, describing its extrinsic curvature:
\begin{equation}
H(t) := \mathrm{Hess}(\delta)|_{X_t}.
\end{equation}
Here, the $(2,0)$ symmetric tensor $\mathrm{Hess}(\delta)$ is the Riemannian Hessian. Straightforward computations show that, in the conic and anti-conic-type of structures, we have
\begin{equation}
	H(t)= \frac{\alpha}{t} g.
\end{equation}
This breaks the symmetry observed for the sectional curvatures in \eqref{eq:sectionalsymmetry-intro}, allowing to control the essential self-adjointness (e.g., as already mentioned, for $n=2$, the case $\alpha=-1$ is essentially self-adjoint, the case $\alpha =2$ is not).

In Section~\ref{s:curvature}, Theorems~\ref{t:curvosc} and \ref{t:curvoscsuper}, we prove two criteria for essential self-adjointness of the Laplace-Beltrami operator, under bounds on the sectional curvature near the metric boundary and the principal curvatures of $X_\varepsilon$. % and controls of the type
%\begin{equation}
%	-\frac{c_1}{\delta^r}\le\mathrm{Sec}(\sigma)\le -\frac{c_2}{\delta^r}, \qquad\delta\le\varepsilon,
%\end{equation} 
%for $r\ge 2$ and constants $c_1\ge c_2>0$ satisfying specific relations. 
In particular, we allow for wild oscillations of the sectional curvature. For simplicity, we hereby present a unified version of these results, without explicit values of the constants.

\begin{thm-intro}\label{t:curv-intro}
Let $(M,g)$ be a Riemannian manifold satisfying $(\star)$ for $\varepsilon>0$.
Assume that there exist $c_1 \geq c_2 \geq 0$ and $r \geq 2$ such that, for all planes $\sigma$ containing $\nabla \delta$, one has
\begin{equation}\label{eq:bound-curv-sec-intro}
-\frac{c_1}{\delta^r}\leq  \mathrm{Sec}(\sigma) \leq -\frac{c_2}{\delta^r}, \qquad \delta \leq \varepsilon.
\end{equation}
Then, there exist a region $\Sigma(n,r) \subset \R^2$, and a constant $h_\varepsilon^*(c_2,r)>0$ such that, if $(c_1,c_2) \in \Sigma(n,r)$ and if the principal curvatures of the hypersurface $X_\varepsilon=\{\delta = \varepsilon\}$ satisfy
\begin{equation}\label{eq:bound-curv-princ-intro}
	H(\varepsilon) < h_\varepsilon^*(c_2,r),
\end{equation}
then $\Delta$ with domain $C^\infty_c(M)$ is essentially self-adjoint in $L^2(M)$.
\end{thm-intro}

In \eqref{eq:bound-curv-princ-intro}, the notation $H(\varepsilon) < \alpha$, for $\alpha \in \R$, is understood in the sense of quadratic forms, that is for all $q \in X_\varepsilon$, we have
\begin{equation}
H(\varepsilon)(X,X) < \alpha g(X,X), \qquad \forall X \in T_q X_\varepsilon.
\end{equation}
For the explicit values of the constants $h^*_\varepsilon(c,r)$ and region $\Sigma(n,r)$ see Theorems~\ref{t:curvosc} and \ref{t:curvoscsuper}. Here, we only observe that if we take $c_1=c_2=c$ in \eqref{eq:bound-curv-sec-intro}, then $(c_1,c_2) \in \Sigma(n,r)$ with $r>2$ if $c>0$, and $(c_1,c_2) \in \Sigma(n,2)$ if $c \geq n/(n-1)^2$.

\subsection{Almost-Riemannian geometry}

As already mentioned, the motivation of this work comes from a conjecture on the essential self-adjointness of the Laplace-Beltrami  operator for almost-Riemannian structures (ARS). These structures have been introduced in \cite{ABS-Gauss-Bonnet}, and represent a large class of non-complete Riemannian structures. Roughly speaking, an ARS on a smooth manifold $N$ consist in a metric $g$ that is singular on an embedded smooth hypersurface $\mathcal Z\subset N$ and smooth on the complement $M=N\setminus \mathcal Z$. For the precise definition see Section~\ref{s:arg}. 

To introduce the results it suffices to observe that for any $q\in \mathcal Z$ there exists a neighborhood $U$ and a local generating family of smooth vector fields $X_1,\ldots, X_n$, orthonormal on $U\setminus\mathcal Z$, which are not linearly independent on $\mathcal Z$. The bracket-generating assumption,
\begin{equation}
	\text{Lie}_q(X_1,\ldots,X_n) = T_q N, \qquad\forall q\in N,
\end{equation}
implies that Riemannian geodesics can cross the singular region. In particular, sufficiently close points on opposite sides of $\mathcal{Z}$ can be joined by smooth trajectories minimizing the length: the Riemannian manifold $(M,g|_{M})$ is not geodesically complete, and hence the classical dynamics is not confined to $M$. 
Surprisingly, recent investigations have shown that the quantum dynamics is quite different. 
\begin{thm-intro}[Boscain, Laurent \cite{BL-LaplaceBeltrami}]\label{t:UgoCam}
Let $N$ be a 2-dimensional ARS on a compact orientable manifold, with smooth singular set $\mathcal{Z} \simeq \mathbb{S}^1$. Assume that, for every $q \in \mathcal{Z}$ and local generating family $\{X_1,X_2\}$, we have
\begin{equation}\label{eq:hypUgo}
\spn\{ X_1,X_2,[X_1,X_2]\}_q = T_q N, \qquad (\text{bracket-generating of step $2$}).
\end{equation}
Then, the Laplace-Beltrami operator $\Delta = \Delta_{\vol_g}$, with domain $C^\infty_c(N \setminus \mathcal{Z})$ is essentially self-adjoint in $L^2(N \setminus \mathcal{Z})$ and its unique self-adjoint extension has compact resolvent.
\end{thm-intro}

In the closing remarks of \cite{BL-LaplaceBeltrami} it has been conjectured that the above result holds true for any sub-Riemannian structure which is rank-varying or non-equiregular on an hypersurface. This is a large class of structures strictly containing the almost-Riemannian ones, to which we will restrict henceforth. We observe that the proof of the above result given in \cite{BL-LaplaceBeltrami} consists in a fine analysis which relies on the normal forms of local generating families of $2$-dimensional almost-Riemannian structures, which is available under the condition \eqref{eq:hypUgo}, but not for higher steps. Moreover, although normal forms for ARS are known also in dimension $n=3$, \cite{BCGM-3ARSNormal}, their complexity increases quickly with the number of degrees of freedom. Hence, it is unlikely for the technique of \cite{BL-LaplaceBeltrami} to yield general results.

On this topic, our main result is the following extension of Theorem~\ref{t:UgoCam}.

\begin{thm-intro}[Quantum completeness of regular ARS]\label{t:regularARS-intro}
Consider a regular almost-Rieman\-nian structure on a smooth manifold $N$ with compact singular region $\mathcal{Z}$. Then, the Laplace-Beltrami operator $\Delta$ with domain $C^\infty_c(M)$ is essentially self-adjoint in $L^2(M)$, where $M = N \setminus \mathcal{Z}$ or one of its connected components. Moreover, when $M$ is relatively compact, the unique self-adjoint extension of $\Delta$ has compact resolvent.
\end{thm-intro}

Regular almost-Riemannian structures (see Definition~\ref{d:regularARS}), are structures where the singular set $\mathcal Z$ is an embedded hypersurface without tangency points, that is, such that $\spn\{X_1,\ldots,X_n\} \pitchfork T_q \mathcal{Z}$ for all $q \in \mathcal{Z}$. Moreover, it is required that, locally $\det(X_1,\ldots,X_n) = \pm\psi^k$ for some $k \in \mathbb{N}$, where $\psi$ is a local submersion defining $\mathcal{Z}$. The latter condition implies that the Riemannian structure on $M=N\setminus \mathcal Z$ satisfies $(\star)$, allowing us to apply Theorem~\ref{t:main-intro}. We also remark that, even in dimension $2$, our result is stronger than Theorem~\ref{t:UgoCam}, as it allows for non-compact, non-orientable and, most importantly, higher step structures.

\subsubsection{Open problems} The conjecture of \cite{BL-LaplaceBeltrami} remains open for non-regular ARS. Notwithstanding, once a local generating family is given explicitly, it is easy to compute $\Veff$. In this way, one can apply the general Theorem~\ref{t:main-intro} to many specific examples of non-regular ARS, yielding the essential self-adjointness of their Laplace-Beltrami operator. On the other hand, Section~\ref{s:non-regular} contains examples of non-regular ARS where, even in dimension $n=2$, we are not able to infer whether $\Delta$ is essentially self-adjoint or not.

We mention that, after the publication of this paper, the techniques developed in this paper have been extended to sub-Laplacians, see \cite{Vale}.

\subsection{Notations and conventions}

In this paper, all manifolds are considered without boundary unless otherwise stated. On the smooth Riemannian manifold $(M,g)$, we denote with $|\cdot |$ the Riemannian norm, without risk of confusion. As usual, $C^\infty_c(M)$ denotes the space of smooth functions with compact support. We denote with $L^2(M)$ the complex Hilbert space of (equivalence classes of) functions $u : M \to \mathbb{C}$, with scalar product
\begin{equation}
\langle u, v \rangle  = \int_M u \bar{v}\, d \omega, \qquad u,v \in L^2(M),
\end{equation}
where the bar denotes complex conjugation. The corresponding norm is denoted by the symbol $\|u\|^2 = \langle u,u \rangle$. Similarly, $L^2(TM)$ is the complex Hilbert space of sections of the complexified tangent bundle $X : M \to TM^\mathbb{C}$, with scalar product
\begin{equation}
\langle X, Y \rangle  = \int_M g(X,Y)\, d \omega, \qquad  X,Y \in L^2(TM),
\end{equation}
where in the above formula, with an abuse of notation, $g$ denotes the Hermitian product on the fibers of $TM^\mathbb{C}$ induced by the Riemannian structure.

Following \cite[Ch. 4]{Gregorio}, we denote by $W^1(M)$ the Sobolev space of functions in $L^2(M)$ with distributional gradient $\nabla u \in L^2(TM)$. This is a Hilbert space with scalar product
\begin{equation}\label{eq:W1}
\langle u, v \rangle_{W^1}= \langle \nabla u, \nabla v \rangle + \langle u, v \rangle.
\end{equation}

We denote by $L^2_\loc(M)$ and $W^1_\loc(M)$ the space of functions $u: M \to \mathbb{C}$ such that, for any relatively compact set $\Omega \Subset M$, their restriction to $\Omega$ belongs to $L^2(\Omega)$ and $W^1(\Omega)$, respectively. Similarly, $L^2_\comp(M)$ and $W^1_\comp(M)$ denote the spaces of functions in $L^2(M)$ and $W^1(M)$, respectively, with compact support. We recall Green's identity:
\begin{equation}
\langle \nabla u, \nabla v \rangle = \langle u, -\Delta_\omega v \rangle, \qquad \forall u,v \in C^\infty_c(M).
\end{equation}

Finally, the symmetric bilinear form associated with $H$ is
\begin{equation}\label{eq:dirichletform}
	\mathcal{E}(u,v)=\int_M \bigg( g(\nabla u, \nabla v) + V u \bar{v} \bigg) \,d\omega, \qquad u,v \in C^\infty_c(M).
\end{equation}
We use the same symbol to denote the above integral, eventually equal to $+\infty$, for all functions $u,v \in W^1_\loc(M)$. We also let, for brevity, $\mathcal{E}(u) = \mathcal{E}(u,u)$. 
 % Introduction
\section{Structure of the metric boundary}

In this section we collect some structural properties of the metric boundary (Lemma~\ref{l:dist}) and provide a simple formula for the computation of $\Veff$ (Proposition~\ref{l:pot}). The results of Lemma~\ref{l:dist} are standard if $\hat{M}$ is itself a Riemannian manifold, but some care is needed to deal with the presence of a general metric boundary, and the issue of low regularity.

Recall that the $(2,0)$ tensor $\mathrm{Hess}(\delta)$ denotes the Riemannian Hessian of $\delta$. The $(1,1)$ tensor $H$ obtained by ``raising an index'' is defined by $g(H X,Y) = \mathrm{Hess}(\delta)(X,Y)$ for any pair of tangent vectors $X,Y$. Finally, $R^\nabla$ is the $(3,1)$ curvature tensor
\begin{equation}
R^\nabla(X,Y)Z = \nabla_X \nabla_Y Z - \nabla_Y \nabla_X Z - \nabla_{[X,Y]} Z.
\end{equation}

\begin{figure}
\includegraphics[scale=1]{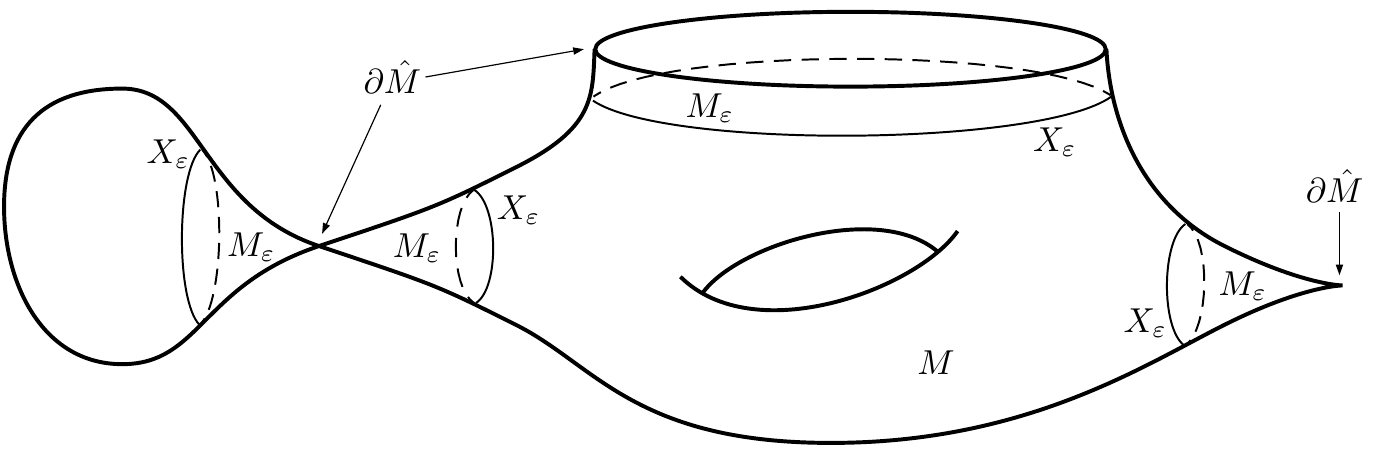}
\caption{Structure of the metric boundary.}
\end{figure}

\begin{lemma}[Properties of the metric boundary]\label{l:dist}
Assume that $(\star)$ holds, that is, there exists $\varepsilon >0$ such that the distance from the metric boundary $\delta : M \to \R$ is $C^2$ on $M_\varepsilon=\{ 0<\delta \leq \varepsilon\}$. Then, on $M_\varepsilon$ we have the following:
\begin{itemize}
\item The distance from the metric boundary satisfies the Eikonal equation:
\begin{equation*}
|\nabla \delta| = 1;
\end{equation*}
\item The integral curves of $\nabla \delta$ are geodesics, and therefore smooth;
\item Let $X_\varepsilon:=\{\delta = \varepsilon\}$. The map $\phi : (0,\varepsilon] \times X_\varepsilon \to M_\varepsilon$, defined by the flow of $\nabla \delta$,
\begin{equation*}
\phi(t,x) := e^{(t-\varepsilon) \nabla \delta}(x),
\end{equation*}
is a $C^1$-diffeomorphism such that $\delta(\phi(t,x)) = t$;
\item $H$ is smooth along the integral curves of $\nabla \delta$;
\item For any integral curve $\gamma(t)$ of $\nabla\delta$, $H$ satisfies the Riccati equation:
\begin{equation*}
\nabla_{\dot\gamma} H + H^2 + R =0,
\end{equation*}
where $R$ is the $(1,1)$ tensor defined by $R X = R^\nabla(X,\nabla\delta)\nabla\delta$, computed along $\gamma(t)$.
\item For any smooth measure $\omega$, the Laplacian $\Delta_\omega \delta$ and all its derivatives in the direction $\nabla \delta$ are continuous.
\end{itemize}
\end{lemma}
\begin{rmk}
The tensor $R$ encodes the sectional curvatures of the planes containing $\nabla\delta$. In fact, for any unit vector $X$ orthogonal to $\nabla\delta$, we have $g(R X,X) =\mathrm{Sec}(\sigma)$, where $\sigma$ is the plane generated by $\nabla\delta$ and $X$.
\end{rmk}
\begin{proof}
Let $p,q \in M$. By the triangle inequality for $\hat{d}$, we have
\begin{equation}
\delta(p) \leq \hat{d}(p,q) + \delta(q).
\end{equation}
Since $\hat{d}(p,q) = d(p,q)$, we obtain $|\delta(p) - \delta(q)| \leq d(p,q)$, that is $\delta$ is $1$-Lipschitz. As a consequence, $\delta$ is differentiable almost everywhere, with $|\nabla \delta | \leq 1$. We now restrict to $M_\varepsilon$ where, by hypothesis, $\nabla\delta$ is $C^1$.

Observe that $(M,d)$ is a length space, with length functional $\ell$, and so is $(\hat{M},\hat{d})$, with the length functional
\begin{equation}
\hat{\ell}(\gamma) = \sup \sum_{i=1}^N \hat{d}(\gamma(t_{i-1}),\gamma(t_i)),
\end{equation}
where the sup is taken over all partitions $0=t_0\leq t_1 \leq \ldots \leq t_N = 1$ and $N \in \mathbb{N}$. Recall that length functionals are continuous as a function of the endpoints of the path \cite[Prop. 2.3.4]{BBI}. Let $\gamma:[0,1] \to \hat{M}$ be a rectifiable curve, such that $\gamma(t) \in M$ for all $t>0$, i.e.\ only the initial point can belong to the metric boundary. Up to reparametrization, we can assume that $\gamma$ is Lipschitz, so that it is differentiable a.e. on $(0,1]$, where its speed is given by $|\dot\gamma(t)|$. In this case,
\begin{equation}
\hat{\ell}(\gamma) = \lim_{s \to 0^+} \hat{\ell}(\gamma|_{[s,1]}) =  \lim_{s \to 0^+} \ell(\gamma|_{[s,1]}) = \int_0^1 |\dot\gamma(t)| dt = \ell(\gamma).
\end{equation}
In particular, for such curves we can measure the length as the usual Lebesgue integral of the speed $|\dot\gamma(t)|$ using the Riemannian structure of $M$. Now recall that, for $p \in M_\varepsilon$,
\begin{align}
\delta(p) & = \inf\{\hat{d}(q,p)\mid q \in \partial\hat{M} \} \\
& = \inf \{\hat{\ell}(\gamma) \mid \gamma(0) \in \partial\hat{M} ,\; \gamma(1) =p\} \\
& = \inf \{\ell(\gamma) \mid \gamma(0) \in \partial\hat{M} ,\; \gamma(1) = p,\; \gamma(t) \in M \text{ for all } t >0\}.
\end{align}
Consider a sequence of Lipschitz curves $\gamma_n:[0,1] \to \hat{M}$ such that $\gamma_n(0) \in \partial\hat{M} $, $\gamma(1) = p$, $\gamma(t) \in M$ for all $t>0$, and
\begin{equation}
\lim_{n \to +\infty} \ell(\gamma_n) = \delta(p).
\end{equation}
Since $\ell$ is invariant by reparametrization, we assume that $\gamma_n$ is parametrized by constant speed $|\dot{\gamma}_n| = \ell(\gamma_n)$ and, since $p \in M_\varepsilon$, we can assume that $\gamma_n(t) \in M_\varepsilon$ for all $t>0$. Since $\delta(\gamma_n(t)) \to 0$ for $t \to 0^+$, we obtain
\begin{equation}
\delta(p) \leq \int_0^1 |g(\nabla\delta,\dot\gamma_n)| dt \leq \int_0^1|\nabla\delta| |\dot\gamma_n| dt = \ell(\gamma_n)\int_0^1|\nabla\delta|dt,
\end{equation}
where we used Cauchy-Schwarz inequality. Integrating $1-|\nabla\delta| \geq 0$ on $\gamma_n$, we obtain
\begin{equation}
0  \leq \int_0^1 \left(1- |\nabla\delta|\right) dt  \leq \frac{\ell(\gamma_n) - \delta(p) }{\ell(\gamma_n)}  \longrightarrow 0, \qquad n \to +\infty.
\end{equation}
This proves that $|\nabla \delta| \equiv 1$ on $M_\varepsilon$.

It is well-known that the integral lines of the gradient of $C^2$ functions satisfying the Eikonal equation are Riemannian geodesics (see \cite[Ch.\ 5, Sec.\ 2]{Petersen}). In particular, the curve $\gamma(t)$ such that $\gamma(0) = p$ and $\dot{\gamma} = -\nabla\delta$ is a unit-speed geodesic such that $\delta(\gamma(t)) = \delta(p) -t$. Using Cauchy-Schwarz inequality, one can show that this is the unique unit-speed curve with this property.

Now consider the set $X_\varepsilon = \{\delta = \varepsilon\}$. Since $\delta$ is $C^2$ with no critical points, $X_\varepsilon\subset M_\varepsilon$ is a $C^2$ embedded hypersurface. Then, we define the $C^1$ map:
\begin{equation}
\phi :(0,\epsilon] \times X_\varepsilon \to M_\varepsilon, \qquad \phi(t,x) = e^{(t -\varepsilon )\nabla\delta}(x),
\end{equation}
where $s\mapsto e^{s V}(x)$ is the integral curve of $V$ starting at $x \in X_\varepsilon$. Since $|\nabla\delta| = 1$ on $M_\varepsilon$, the flow is well defined on $(0,\varepsilon]$. This map is indeed a $C^1$-diffeomorphism, and $\delta(\phi(t,x)) = t$.

The fact that $H(t) = \mathrm{Hess}(\delta)|_{\gamma(t)}$ satisfies the Riccati equation is usually proved assuming that $\delta$ is smooth (see, e.g. \cite[Prop. 7]{Petersen}). When $\delta \in C^2$, then $H$ satisfies the Riccati equation in the distributional sense. We omit the details since they would require the introduction of distributional covariant derivatives, which is out of the scope of this paper (see, e.g., \cite[Ch. 1]{marsden}). Then, one obtains that $H(t)$ is actually smooth via a bootstrap argument, exploiting the fact that the term $R=R(t)$ has the same regularity of $\nabla \delta|_{\gamma(t)}$. The same argument shows that all derivatives $\nabla^i_{\nabla \delta} H$ are continuous on $M_\varepsilon$.

The last statement follows from the formula $\Delta \delta = \tr H$ for the Laplace-Beltrami operator, and the fact that, if $\omega = e^{h} \vol_g$, it holds $\Delta_\omega = \Delta + g(\nabla h,\nabla \cdot)$.
\end{proof}

\begin{prop}[Formula for the effective potential]\label{l:pot}
Through the identification $M_\varepsilon \simeq (0,\varepsilon] \times X_\varepsilon$ of Lemma~\ref{l:dist} we have
\begin{equation}
d\omega(t,x) = e^{2 \vartheta(t,x)}\, dt\, d\mu(x),
\end{equation}
where $d\mu$ is a fixed $C^1$ measure on $X_\varepsilon$. The function $\vartheta$ is smooth in $t\in (0,\varepsilon]$ and is continuous in $x \in X_\varepsilon$, together with its derivatives w.r.t.\ $t$. Moreover,
\begin{equation}
\Veff = (\partial_t \vartheta)^2 + \partial_t^2 \vartheta.
\end{equation}
In particular, the effective potential $\Veff$ is continuous. 
\end{prop}
\begin{rmk}
By choosing a different reference measure $d\tilde{\mu}$ on $X_\varepsilon$, we have that $\tilde{\vartheta}(t,x) = \vartheta(t,x) + g(x)$, so that the value of $\Veff(t,x)$ does not depend on this choice.
\end{rmk}
\begin{proof}

First observe that, if $\delta$ is smooth on $M_\varepsilon$, then the map $\phi : (0,\varepsilon] \times X_\varepsilon \to M_\varepsilon$ is a smooth diffeomorphism and $d\mu$ can be chosen to be smooth. With the identification $M_\varepsilon \simeq (0,\varepsilon] \times X_\varepsilon$, we have $\nabla\delta = \partial_t$. Then, by definition of $\dive_\omega$, we obtain
\begin{align}
(\Delta_\omega \delta)  \omega = \dive_\omega(\partial_t) \omega  & = \mathcal{L}_{\partial_t} \omega  = (2 \partial_t \vartheta) e^{2\vartheta} dt\, d\mu + e^{2\vartheta} \mathcal{L}_{\partial_t} (dt\, d\mu) = (2\partial_t \vartheta) \omega,
\end{align}
where we used the fact that $\mathcal{L}_{\partial_t} (d\mu) =0$. Moreover,
\begin{equation}\label{eq:derseconda}
(\Delta_\omega\delta)^\prime = 2 g(\partial_t, \nabla(\partial_t \vartheta )) =  2 \partial_t^2 \vartheta.
\end{equation}
The statement then follows from the definition of the effective potential \eqref{eq:potential}.

In the general case, $\delta$ is only $C^2$, hence $\vartheta : (0,\varepsilon] \times X_\varepsilon \to \R$ is only continuous. Nevertheless, we claim that $2 \partial_t\vartheta = \Delta_\omega \delta$, for any fixed $x \in X_\varepsilon$. By the last item of Lemma~\ref{l:dist} and \eqref{eq:derseconda}, this will conclude the proof of the statement. 

In order to prove the claim, let $\chi \in C^2_c(X_\varepsilon)$ and $\varphi \in C^\infty_c((0,\varepsilon))$. Then,
\begin{align}
\int_{X_{\varepsilon}} \left( \int_{0}^\varepsilon (e^{2\vartheta} \Delta_\omega \delta) \varphi(t)\, dt\right) \chi(x)\, d\mu(x) & = \int_{(0,\varepsilon)\times X_{\varepsilon}} (\Delta_\omega \delta) \chi(x) \varphi(t)   \, d\omega   \\
&= -\int_{(0,\varepsilon)\times X_{\varepsilon}} \chi(x)\partial_t \varphi (t)\, d\omega \\
& = -\int_{X_{\varepsilon}} \left( \int_{0}^\varepsilon  e^{2\vartheta}\partial_t \varphi(t)\, dt\right) \chi(x)\, d\mu(x),
\end{align}
where we used Fubini's Theorem, Green's identity, and the fact that, with the identification $M_\varepsilon \simeq (0,\varepsilon] \times X_\varepsilon$, we have $\nabla\delta = \partial_t$. By the arbitrariness of $\chi$, we have, 
\begin{equation}
	\int_{0}^\varepsilon (e^{2\vartheta} \Delta_\omega \delta) \varphi \, dt = -\int_{0}^\varepsilon e^{2\vartheta} \partial_t \varphi\, dt, \qquad \forall \varphi\in C^\infty_c((0,\varepsilon)).
\end{equation}
Since $e^{2\vartheta} \Delta_\omega \delta$ is continuous, $\partial_t e^{2\vartheta} = e^{2\vartheta} \Delta_{\omega}\delta$ in the strong sense. In particular, by the chain rule for distributional derivatives, $2\partial_t \vartheta =  \Delta_\omega\delta$, completing the proof of the claim.
\end{proof}

\section{Main self-adjointness criterion}\label{s:main}

In this section we prove Theorem~\ref{t:main-intro}, which we restate here for the reader's convenience.

\begin{theorem}[Main quantum completeness criterion]\label{t:main}
Let $(M,g)$ be a Riemannian manifold satisfying $(\star)$ for $\varepsilon >0$. Let $\pot \in L^2_\loc(M)$. Assume that there exist $\kappa \geq 0$ and a Lipschitz function $\nu : M \to \R$ such that, close to the metric boundary, 
	\begin{equation}\label{eq:hypo-1}
\Veff + \pot  \geq \frac{3}{4\delta^2}-\frac{\kappa}{\delta} -\nu^2, \qquad  \text{for $\delta \leq \varepsilon$}.
\end{equation}
Moreover, assume that there exist $\varepsilon' < \varepsilon$, such that,
\begin{equation}\label{eq:hypo-2}
\pot  \geq - \nu^2, \qquad \text{for $\delta > \varepsilon'$}.
\end{equation}	
Then, $H=-\Delta_\omega+\pot$ with domain $C^\infty_c(M)$ is essentially self-adjoint in $L^2(M)$.

Finally, if $\hat{M}$ is compact, then the unique self-adjoint extension of $H$ has compact resolvent. Therefore, its spectrum is discrete and consists of eigenvalues with finite multiplicity.
\end{theorem}

\begin{rmk}
It is well-known that the $3/4$ factor in \eqref{eq:hypo-1} is optimal and cannot be replaced with a smaller constant. (See, e.g., \cite[Thm. X.10]{MR0493420} for the one-dimensional case.) However, as proven in \cite{Nenciu2008} in the case of bounded domains in $\mathbb R^n$, the whole right hand side of \eqref{eq:hypo-1} can be replaced by functional expressions of $\delta$ that satisfies some precise conditions. For clarity, and since it is sufficient for the forthcoming applications, we limit ourselves to an expression of the form \eqref{eq:hypo-1}. Notwithstanding, we see no obstacles in applying the refined techniques of \cite{Nenciu2008} in our geometrical setting to obtain sharper functional conditions.
\end{rmk}

An important ingredient in the proof is the inclusion $D(H^*)\subset W^1_\loc(M)$. For a general $V \in L^2_\loc(M)$, this \emph{a-priori} regularity is not guaranteed. As proven in \cite[Thm. 2.3]{BMS-manifolds}, this inclusion holds whenever the potential $V \in L^2_\loc(M)$ can be decomposed as $V = V_+ + V_-$, where $V_+ \geq 0$ and $V_- \leq 0$ are such that, for any compact $K \subset M$ there are positive constants $a_K <1$ and $C_K$ such that,
\begin{equation}\label{eq:ShubinASSA}
\left(\int_K |V_-|^2 |u|^2 \, d\omega\right)^{1/2} \leq a_K \|\Delta_\omega u \| + C_K \|u\|, \qquad \forall u \in C^\infty_c(M).
\end{equation}
This is true, for example, when $V_- \in L^p_\loc(M)$ with $p > n/2$ for $n \geq 4$ and $p = 2$ for $n \leq 3$ or is in the local Stummel class, see \cite[Remark 2.2]{BMS-manifolds} and also \cite{DonnelyGarofalo97}. In particular, this is certainly true when $V \in L^{\infty}_\loc(M)$.

\begin{lemma}\label{l:shubin}
	Under the assumptions of Theorem~\ref{t:main}, $D(H^*)\subset W^1_\loc(M)$.
\end{lemma}

\begin{proof}
Observe that in the simpler case $V\in L^\infty_\loc(M)$, with no other assumptions, the result is a consequence of standard elliptic regularity theory. In fact, in this case, $u \in D(H^*)$ implies $\Delta_\omega u \in L^2_\loc(M)$, in the sense of distributions. Then, the result follows from \cite[Thm. 6.9]{Gregorio}, and taking in account the claim of \cite[p. 144]{Gregorio}. On the other hand, in the general case $V \in L^2_\loc(M)$, assumptions \eqref{eq:hypo-1} and \eqref{eq:hypo-2} imply \eqref{eq:ShubinASSA} with $a_K=0$. This guarantees $\dom( \op^*) \subset W^1_\loc(M)$, by \cite[Thm. 2.3]{BMS-manifolds}. 
\end{proof}

%Henceforth we let $\op = -\Delta_\omega+\pot$. 

\begin{proof}[Proof of Theorem~\ref{t:main}]
We first prove the statement in the case $\nu \equiv 0$, in particular $V \geq 0$ for $\delta > \varepsilon'$. In this case, as shown in Proposition~\ref{prop:weak-hardy}, the operator $\op$ is semibounded. Thus, by a well-known criterion, $\op$ is essentially self-adjoint if and only if there exists $E <0$ such that the only solution of $\op^*\psi = E\psi$ is $\psi \equiv 0$ (see \cite[Thm. X.I and Corollary]{MR0493420}). This is guaranteed by the Agmon-type estimate of Proposition~\ref{p:Agmon}.

In order to complete the proof, notice that, for any $\lambda \geq 1$, the operator $-\Delta_\omega + V'$, with $V':= V + \lambda \nu^2$ falls in the previous case, hence it is essentially self-adjoint. Then, we conclude by Proposition~\ref{p:toglipotenziale}.

The compactness of the resolvent when $\hat{M}$ is compact is the result of Proposition~\ref{prop:compact-resolvent}.
\end{proof}

\begin{rmk}
Assumption \eqref{eq:hypo-2} can be relaxed by requiring that, for some $a \in [0,1)$ it holds $-a\Delta_\omega + V \geq -\nu^2$. However, in this case, the inclusion $D(H^*)\subset W^1_\loc(M)$ is not guaranteed by the arguments of Lemma~\ref{l:shubin} and must be enforced. Then, the proof of Theorem~\ref{t:main} is mostly unchanged, with minor modifications in step 2 in the proof of Proposition~\ref{p:Agmon}, and a straightforward extension of Proposition~\ref{p:toglipotenziale}.
\end{rmk}

\begin{prop}\label{prop:weak-hardy}
Let $(M,g)$ be a Riemannian manifold satisfying $(\star)$ for some $\varepsilon >0$. Let $\pot \in L^2_{\loc}(M)$. Assume that there exist $\kappa \geq 0$ and $\varepsilon' < \varepsilon$ such that,
	\begin{align}\label{eq:assumptionprop}
\Veff + \pot & \geq \frac{3}{4\delta^2}-\frac{\kappa}{\delta}, & \text{for $\delta \leq \varepsilon$}, \\
\pot & \geq 0, &  \text{for $\delta > \varepsilon'$}.
	\end{align}
Then, there exist $\eta \leq 1/\kappa$ and $c \in \R$ such that
	\begin{equation}\label{eq:weak-hardy}
\mathcal{E}(u)   \geq	\int_{M_{\eta}} \left(\frac{1}{\delta^2} - \frac{\kappa}{\delta} \right)|u|^2\,d\omega  +c \|u\|^2, \qquad \forall u\in W^1_\comp(M).
	\end{equation}
In particular, the operator $H=-\Delta_\omega+\pot$ is semibounded on $C^\infty_c(M)$.
\end{prop}
\begin{proof} First we prove \eqref{eq:weak-hardy} for $u \in W^{1}_{\comp}(M_{\varepsilon})$, and with $\eta = \varepsilon$, possibly not satisfying $\eta \leq 1/\kappa$. Then, we extend it for $u \in W^1_\comp(M)$, choosing $\eta\le 1/\kappa$. 

\emph{Step 1.} Let $u \in W^{1}_{\comp}(M_{\varepsilon})$. By Lemma~\ref{l:dist}, we identify $M_\varepsilon \simeq (0,\varepsilon] \times X_\varepsilon$ in such a way that $\delta(t,x) = t$ for $(t,x) \in (0,\varepsilon] \times X_\varepsilon$. By Proposition~\ref{l:pot}, fixing a reference measure $d\mu$ on $X_\varepsilon$, we have
\begin{equation}
	d\omega = e^{2 \vartheta(t,x)} dt\, d\mu(x), \qquad \text{on } M_\varepsilon,
\end{equation}
for some function $\vartheta : M_\varepsilon \to \R$ smooth in $t$ and continuous in $x$, together with its derivatives w.r.t.\ $t$. Consider the unitary transformation $T :L^2(M_\varepsilon,d\omega)\to L^2(M_\varepsilon,dt\,d\mu )$ defined by $Tu = e^\vartheta u$. Letting $v=T u$, and integrating by parts yields
	\begin{equation}
			\mathcal{E}(u) \geq \int_{M_\varepsilon}\left( |\partial_t u|^2 + \pot |u|^2 \right) \,d\omega 
			=  \int_{M_\varepsilon} \bigg(|\partial_t v|^2 + \bigg(\underbrace{(\partial_t \vartheta)^2 + \partial^2_t \vartheta}_{= \Veff} +\pot\bigg) |v|^2\bigg) \, dt\,d\mu,
	\end{equation}
where the expression for $\Veff$ is in Proposition~\ref{l:pot}. Recall the 1D Hardy inequality:
\begin{equation}\label{eq:Hardy1D}
\int_0^\varepsilon |f'(s)|^2\, ds  \geq \frac{1}{4}\int_0^\varepsilon \frac{|f(s)|^2}{s^2}\,ds, \qquad \forall f \in W^{1}_\comp((0,\varepsilon)).
\end{equation}
Since $u\in W^1_\comp(M_\varepsilon)$ and $\vartheta$ is smooth in $t$, for a.e.\ $x \in X_\varepsilon$, the function $t \mapsto v(t,x)$ is in $W^{1}_\comp((0,\varepsilon))$ (see \cite[Thm. 4.21]{EGmeas}). Then, by using \eqref{eq:assumptionprop}, Fubini's Theorem and \eqref{eq:Hardy1D}, we obtain \eqref{eq:weak-hardy} for functions $u \in W^{1}_{\comp}(M_{\varepsilon})$ with $\eta = \varepsilon$ and $c=0$.
	
\emph{Step 2.}	Let $u \in W^1_\comp(M)$, and let $\chi_1,\chi_2$ be smooth functions on $[0,+\infty)$ such that
\begin{itemize}
\item $0 \leq \chi_i \leq 1$ for $i=1,2$;
\item $\chi_1 \equiv 1$ on $[0,\varepsilon']$ and $\chi_1 \equiv 0$ on $[\varepsilon,+\infty)$;
\item $\chi_2 \equiv 0$ on $[0,\varepsilon']$ and $\chi_2 \equiv 1$ on $[\varepsilon,+\infty)$;
\item $\chi_1^2+\chi_2^2 =1 $.
\end{itemize}
Consider the functions $\phi_i : M \to \R$ defined by $\phi_i:= \chi_i \circ \delta$. We have $\phi_1 \equiv 1$ on $M_{\varepsilon'}$, $M_{\varepsilon'} \subset \mathrm{supp}(\phi_1) \subseteq M_\varepsilon$, moreover $0\leq \phi_1 \leq 1$, and $\phi_1^2 + \phi_2^2 = 1$. Notice that $\phi_2 \equiv 1$ and $\phi_1 \equiv 0 $ on $M \setminus M_\varepsilon$, and so $\nabla \phi_i \equiv 0$ there. Moreover, since $|\nabla\delta| \leq 1$,
\begin{equation}
c_1=  \sup_{M} \sum_{i=1}^2|\nabla \phi_i|^2 \leq \sup_{[0,\varepsilon]} \sum_{i=1}^2 |\chi_i'|^2< +\infty.
\end{equation}
Since $\supp (\phi_2 u) \subseteq M \setminus M_{\varepsilon'}$, and recalling that $V \geq 0$ on $M \setminus M_{\varepsilon'}$, we have $\mathcal{E}(\phi_2 u) \geq 0$. By \eqref{eq:keyidentity} of Lemma~\ref{l:trick}, we obtain the following IMS-type formula:
\begin{align}
\mathcal{E}(u) = \sum_{i=1}^2 \mathcal{E}(\phi_i u) -\sum_{i=1}^2  \int_M |\nabla \phi_i|^2 |u|^2 d\omega 
 \geq \mathcal{E}(\phi_1 u) -  c_1 \|u\|^2.
\end{align}
In particular, applying the previously proven statement to $\phi_1 u \in W^1_\comp(M_\varepsilon)$, we get
\begin{align}
\mathcal{E}(u)  &\geq \int_{M_\varepsilon}\left(\frac{1}{\delta^2}-\frac{\kappa}{\delta}\right)|\phi_1 u|^2 d\omega - c_1\|u\|^2. 
\end{align}
Letting $\eta = \min\{\varepsilon',1/\kappa\}$, we have
\begin{align}
	\mathcal E(u)& \geq \int_{M_{\eta}} \left(\frac{1}{\delta^2}-\frac{\kappa}{\delta}\right)|u|^2 d\omega- \int_{M_{\varepsilon}\setminus M_{\eta}} \left\lvert\frac{1}{\delta^2}-\frac{\kappa}{\delta}\right\rvert |\phi_1 u|^2 d\omega -c_1\|u\|^2 \\
& \geq \int_{M_{\eta}} \left(\frac{1}{\delta^2}-\frac{\kappa}{\delta}\right)|u|^2 d\omega - \left(c_1 + \sup_{\eta \leq \delta \leq \varepsilon}\left\lvert \frac{1}{\delta^2} - \frac{\kappa}{\delta}\right\rvert \right) \|u\|^2,
\end{align}
which concludes the proof.
\end{proof}

\begin{prop}[Agmon-type estimate]\label{p:Agmon}
Assume that there exist $\kappa \geq 0$, $\eta \leq 1/\kappa$ and $c \in \R$ such that,
	\begin{equation}\label{eq:agmon}
\mathcal{E}(u) \geq  \int_{M_\eta} \left(\frac {1} {\delta^2} - \frac \kappa \delta\right)|u|^2 d\omega + c \|u\|^2 , \qquad \forall u \in W^{1}_\comp(M).
	\end{equation}
Then, for all $E < c$, the only solution of $\op^* \psi = E \psi$ is $\psi \equiv 0$.
\end{prop}
Notice that the requirement $\eta \leq 1/\kappa$ ensures the non-negativity of the integrand in \eqref{eq:agmon}. The proof of the above follows the ideas of \cite{Nenciu2008,DeVerdiere2009a}, via the following. 

\begin{lemma}\label{l:trick}
Let $f$ be a real-valued Lipschitz function. Let $u \in W^1_\loc(M)$, and assume that $f$ or $u$ have compact support $K \subset M$. Then, we have
\begin{equation}\label{eq:keyidentity}
\mathcal{E}(f u,fu) = \re \mathcal{E}(u,f^2 u ) + \langle u , |\nabla f|^2 u \rangle.
\end{equation}
Moreover, under the assumptions of Proposition~\ref{prop:weak-hardy}, if $\psi \in \dom(H^*)$ satisfies $H^* \psi = E \psi$, and $f$ is a Lipschitz function with compact support, we have
\begin{equation}\label{eq:keyidentity2}
\mathcal{E}(f \psi,f\psi) = E \| f \psi \|^2 + \langle \psi , |\nabla f|^2 \psi \rangle.
\end{equation}
\end{lemma}
\begin{proof}
Observe that $fu \in W^1_\comp(M)$. By using the fact that $f$ is real-valued, a straightforward application of Leibniz rule yields
\begin{align}
\langle  \nabla u, \nabla (f^2 u)\rangle & = \langle  f\nabla u, \nabla (f u)\rangle  + \langle \nabla u, f u \nabla f  \rangle \\
& = \langle \nabla (fu), \nabla (fu)\rangle - \langle u \nabla f,\nabla (fu) \rangle+ \langle \nabla u, f u \nabla f  \rangle \\
& = \langle \nabla (fu), \nabla (fu)\rangle - \langle u \nabla f,u\nabla f \rangle- \langle u \nabla f, f\nabla u \rangle +\langle f \nabla u,  u \nabla f  \rangle  \\
& = \langle \nabla (fu), \nabla (fu)\rangle - \langle u, |\nabla f|^2 u\rangle + 2 i \im \langle f \nabla u,  u \nabla f  \rangle.
\end{align}
Thus, by definition of $\mathcal{E}$, we have
\begin{align}
\re \mathcal{E}(u,f^2 u) & =  \langle \nabla (fu), \nabla (fu)\rangle + \langle V u, f^2u\rangle  - \langle u, |\nabla f|^2 u \rangle \\
& = \mathcal{E}(fu, fu) - \langle u, |\nabla f|^2 u \rangle,
\end{align}
completing the proof of \eqref{eq:keyidentity}.

To prove \eqref{eq:keyidentity2}, recall that $W^{-1}_\loc(M)$ is the dual of $W^1_\comp(M)$. We denote the duality with the symbol $(u,v)$ where $u \in W^{-1}_\loc(M) $ and $v \in W^1_\comp(M)$. By Lemma \ref{l:shubin}, $\dom(H^*) \subset W^1_\loc(M)$, then $-\Delta_\omega u \in W^{-1}_\loc(M)$, in the sense of distributions. Decompose $V = V_+ + V_-$ in its positive and negative parts. By \eqref{eq:assumptionprop}, $V_-  \in L^\infty_\loc(M)$, and so $V_- u \in W^{-1}_\loc(M)$. Thus, 
\begin{equation}
V_+ u = H^*u + \Delta_\omega u -V_- u\in W^{-1}_\loc(M).
\end{equation}
By applying \cite[Lemma 8.4]{BMS-manifolds} to $V_+$ and $-V_-$, respectively, we have\footnote{Observe that if $v \in W^{-1}_\loc(M) \cap L^1_\loc(M)$ and $u \in W^1_\comp(M)$, then it can happen that $vu$ is not in $L^1_\comp(M)$, and thus the integral $\langle v, u \rangle = \int \bar{v}u \,d \omega$ can fail to be well defined, even though $(v, u)$ is well defined by the duality, in particular $(v,u) := \lim_n \langle v,u_n\rangle$ for some sequence $u_n \to u$ in the $W^1$ topology. The content of \cite[Lemma 8.4]{BMS-manifolds} is that, if $v=Au$, with $A \geq 0$, and $Au \in W^{-1}_\loc(M)$, then $(Au,u) = \lim_n \langle Au,u_n\rangle = \langle Au,u\rangle$.}
\begin{equation}
(V u,u) = \int_{\supp(w)} V \overline{u} u\, d \omega = \langle V u, u\rangle.
\end{equation}
Thus, since $(Vfu,fu) = (Vu,f^2u)$, we finally obtain
\begin{align}
\mathcal{E}(u,f^2 u) & =  \langle \nabla u , \nabla (f^2 u) \rangle  + \langle V u, f^2 u \rangle \\
& = (-\Delta_\omega u, f^2 u) + (V u, f^2 u) \\
& = \langle H^* u, f^2 u \rangle .
\end{align}
Setting $u= \psi$, we obtain $\mathcal{E}(\psi,f^2 \psi) = E \|f \psi\|^2$, yielding the statement.
\end{proof}

\begin{proof}[Proof of Proposition~\ref{p:Agmon}]

Let $f: M \to \R$ be a bounded Lipschitz function with $\supp f \subset \overline{M \setminus M_\zeta}$, for some $\zeta >0$, and $\psi$ be a solution of $(H^* - E) \psi = 0$ for some $E <c $. We start by claiming that
\begin{equation}\label{eq:agmondfirststep}
(c-E) \|f \psi\|^2 \leq \langle \psi, |\nabla f |^2 \psi \rangle - \int_{M_\eta} \left(\frac{1}{\delta^2} - \frac{\kappa}{\delta}\right) |f \psi|^2 d\omega.
\end{equation}
If $f$ had compact support, then $f \psi \in W^1_\comp(M)$, and hence \eqref{eq:agmondfirststep} would follow directly from \eqref{eq:agmon} and \eqref{eq:keyidentity2}. To prove the general case, let $\theta : \R \to \R$ be the function defined by
\begin{equation}
\theta(s) = \begin{cases}
1 & s \leq 0, \\
1-s & 0\leq s \leq 1, \\
0 & s \geq 1.
\end{cases}
\end{equation}
Fix $q \in M$ and let $G_n : M \to \R$ defined by $G_n(p) = \theta(d_g(q,p) - n )$. Notice that $G_n$ is Lipschitz, with $|\nabla G_n | \leq 1$ and $\supp(G_n) \subseteq \bar{B}_q(n+1)$. Observe that
\begin{equation}\label{eq:compactnesssupport}
\supp G_n f \subseteq \overline{(M\setminus M_\zeta) \cap B_q(n+1)}.
\end{equation}
Even if $(M,d)$ is a non-complete metric space (and hence, its closed balls might fail to be compact), the set on the right hand side of \eqref{eq:compactnesssupport} is compact, being uniformly separated from the metric boundary. This can be proved with the same argument of \cite[Prop. 2.5.22]{BBI}. Hence, the support of $f_n:= G_n f$ is compact, and \eqref{eq:agmondfirststep} holds with $f_n$ in place of $f$.
The claim now follows by dominated convergence. Indeed, $f_n \to f$ point-wise as $n \to +\infty$ and $f_n \leq f$. Hence $\|f_n \psi \| \to \|f \psi \|$. Thus, since $\supp f_n \subset \overline{M \setminus M_\zeta}$, we have
\begin{equation}
\lim_{n \to +\infty} \int_{M_\eta} \left(\frac{1}{\delta^2} - \frac{\kappa}{\delta}\right) |f_n \psi|^2\, d\omega= \int_{M_\eta} \left(\frac{1}{\delta^2} - \frac{\kappa}{\delta}\right) |f \psi|^2\, d\omega.
\end{equation}
Finally, since $|\nabla f_n | \leq C$, and $\nabla f_n \to \nabla f$ a.e.\ we have $\langle \psi, |\nabla f_n |^2 \psi \rangle \to \langle \psi, |\nabla f |^2 \psi \rangle$, yielding the claim.

We now plug a particular choice of $f$ into \eqref{eq:agmondfirststep}. Set
\begin{equation}\label{eq:functin}
f(p) := \begin{cases}
F(\delta(p)) &  0 < \delta(p) \leq \eta, \\
1 & \delta(p) > \eta ,
%1 & \varepsilon < \delta(p) \leq R , \\
%R + 1 - D(p) & R < \delta(p) \leq R+1 ,\\
%0 & D(p) > R+1,
\end{cases}
\end{equation}
where $F$ is a Lipschitz function to be chosen later. Recall that $|\nabla\delta| \leq 1$ a.e. on $M$. In particular, on $M_\eta$, we have $|\nabla f| = |F'(\delta)| |\nabla\delta | \leq |F'(\delta)|$. Thus, by \eqref{eq:agmondfirststep}, we have
\begin{equation}\label{eq:agmondsecondstep}
(c-E) \|f \psi\|^2 \leq  \int_{M_\eta} \left[F'(\delta)^2 -\left(\frac{1}{\delta^2} - \frac{\kappa}{\delta}\right) F(\delta)^2 \right]|\psi|^2 d\omega .%+ \int_{R <D<R+1} |\psi|^2 d\omega.
\end{equation}
Let now $0<2\zeta <\eta$. We choose $F$ for $\tau \in [2\zeta,\eta]$ to be the solution of
\begin{equation}\label{eq:costruzioneFAgmon}
F'(\tau) = \sqrt{\frac{1}{\tau^2}- \frac{\kappa}{\tau}} F(\tau), \qquad \text{with } F(\eta) = 1,
\end{equation}
to be zero on $[0,\zeta]$, and linear on $[\zeta,2\zeta]$, see Fig.~\ref{f:function}. Observe that the assumption $\eta \leq 1/\kappa$ implies that the above equation is well defined. One can check that the global function defined by \eqref{eq:functin} is Lipschitz with support contained in $\overline{M\setminus M_\zeta}$. Moreover, explicit computations yield that $F' \leq K$ on $[\zeta,2\zeta]$, for some constant independent of $\zeta$. Indeed, if $\kappa = 0$, the claim is trivial. Assuming $\kappa >0$, the solution to \eqref{eq:costruzioneFAgmon}, on the interval $[2\zeta,\eta]$, is
\begin{equation}
F(\tau)=C(\kappa,\eta) \frac{1-\sqrt{1-\kappa  \tau}}{1+\sqrt{1-\kappa  \tau}}e^{2 \sqrt{1-\kappa  \tau}}, \qquad \tau \in[2\zeta,\eta],
\end{equation}
for a constant $C(\kappa,\eta)$ such that $F(\eta) = 1$. By construction of $F$ on $[\zeta,2\zeta]$, we obtain
\begin{equation}
F'(\tau) = \frac{F(2\zeta)}{\zeta}, \qquad \tau \in [\zeta,2\zeta].
\end{equation}
We have $F(2\zeta) = \tfrac{1}{2} C(\kappa,\eta) e^2 \kappa \zeta + o(\zeta)$, which yields the boundedness of $F'$ on $[\zeta,2\zeta]$ by a constant not depending on $\zeta$. Thus, by \eqref{eq:agmondsecondstep},
\begin{figure}
\includegraphics[width=.8\textwidth]{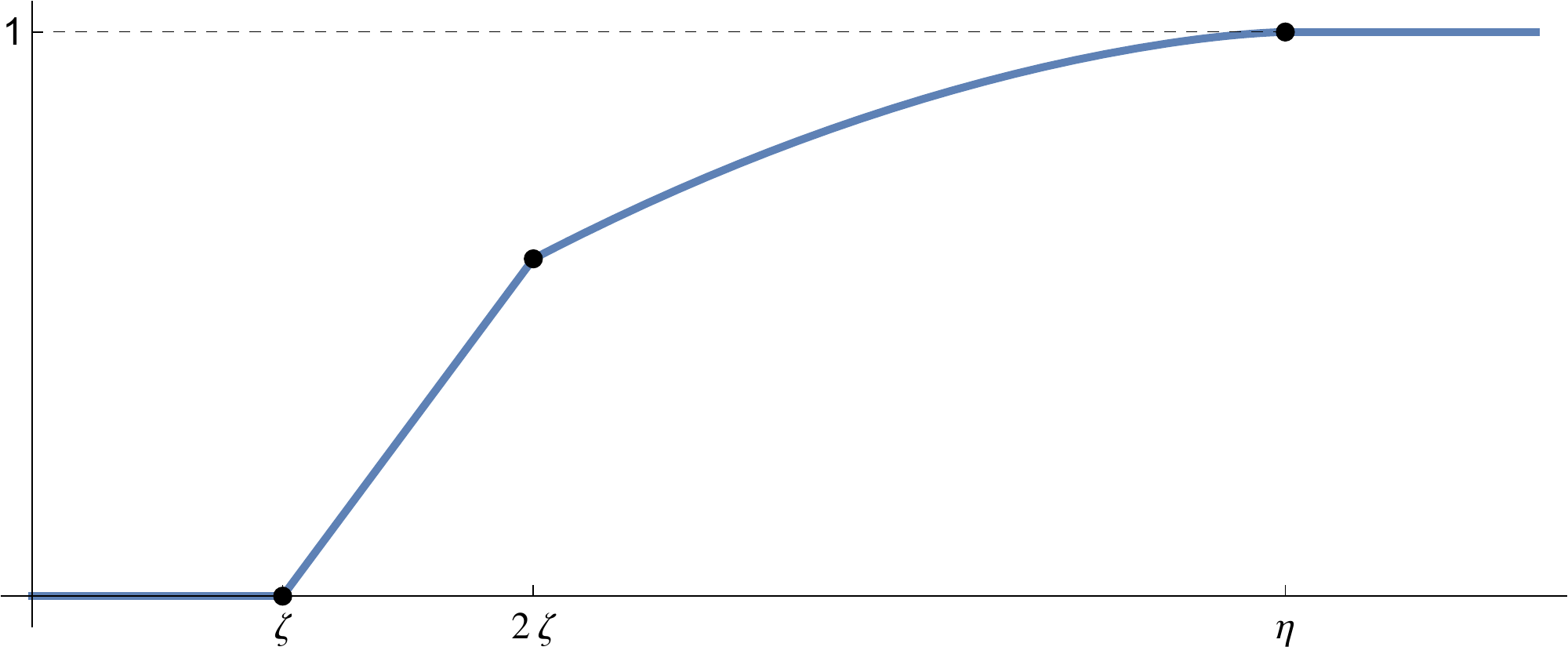}
\caption{Plot of the function $F(\tau)$. Compare with \cite[Fig. 4.1]{DeVerdiere2009a}.}\label{f:function}
\end{figure}
\begin{equation}\label{eq:laststep}
(c-E)\|f \psi \|^2 \leq K^2 \int_{ \zeta \leq \delta \leq 2\zeta} |\psi|^2 d \omega .
\end{equation}
If we let $\zeta \to 0$, then $f$ tends to an almost everywhere strictly positive function. Recalling that $E<c$, and taking the limit, equation \eqref{eq:laststep} implies $ \psi \equiv 0$.
\end{proof}

\begin{prop}\label{p:toglipotenziale}
Let $\nu : M \to \R$ be a non-negative Lipschitz function with Lipschitz constant $L>0$. Assume that $V \in L^2_\loc(M)$ satisfies $V \geq - \nu^2$. Then, if for some $\lambda \geq \max\{1+4L^2,2\}$, the operator $-\Delta_\omega + V + \lambda \nu^2$ with domain $C^\infty_c(M)$ is essentially self-adjoint, the same holds for $\op=-\Delta_\omega + \pot$.
\end{prop}
\begin{proof}
By our assumptions, $\op + \lambda \nu^2\geq 0$. Then, for $\mu \geq 1$ to be fixed later, consider the essentially self-adjoint operator $N := \op + \lambda \nu^2 + \mu \geq 1$, with $\dom(N) = C^\infty_c(M)$. Then, by \cite[Thm.\ X.37]{MR0493420}, it suffices to prove that, for some $C,D\geq 0$, and all $u \in C^\infty_c(M)$,
	\begin{gather}
		\|\op u\|^2 \le C\|Nu\|^2, \label{eq:reed-1-new} \\ 
		| \im \langle \op u, N u\rangle |  \leq D \langle N u,u\rangle. \label{eq:reed-2-new}
	\end{gather}
By \eqref{eq:keyidentity} of Lemma~\ref{l:trick}, letting $L$ be the Lipschitz constant of $\nu$, we have for all $u \in C^\infty_c(M)$,
\begin{align}
\re \mathcal{E}(u,\nu^2 u) & = \mathcal{E}(\nu u ,\nu u) - \langle u, |\nabla \nu|^2 u\rangle \\
& \geq -\|\nu^2 u \|^2 - L^2 \|u\|^2,
\end{align}
where we used the fact that $H = -\Delta_\omega + V \geq -\nu^2$. Hence, 
\begin{align}
\|N u \|^2 & = \| (\op + \lambda \nu^2 + \mu )u \|^2 \\
& =  \| (\op + \lambda \nu^2 )u \|^2  + \mu^2 \|u\|^2 + 2 \mu\re \langle (H + \lambda \nu^2 ) u , u \rangle \\
& \geq \| \op  u \|^2 + \lambda^2 \| \nu^2 u \|^2 + 2 \lambda \re \langle H u, \nu^2 u \rangle + \mu^2 \|u\|^2 \\
& = \| \op  u \|^2 + \lambda^2 \| \nu^2 u \|^2 + 2 \lambda \re \mathcal{E}(u,\nu^2 u) + \mu^2 \|u\|^2 \\
& \geq \| \op u \|^2 + \lambda(\lambda -2) \| \nu^2 u \|^2 + (\mu^2 - 2L^2\lambda)\|u\|^2  \geq \| \op u \|^2,
\end{align}
where, in the last inequality, we fixed $\mu \geq 1$ such that $\mu^2 \geq 2 L^2 \lambda$. This proves \eqref{eq:reed-1-new} with $C=1$. To prove \eqref{eq:reed-2-new}, observe that
\begin{equation}
\begin{aligned}
0  \leq \|\nabla u \pm i u \nabla \nu^2 \|^2  & = \|  \nabla u  \|^2 + 4\langle u, \nu^2 |\nabla \nu |^2 u \rangle \pm 2 \im \langle \nabla u \cdot \nabla \nu^2, u \rangle \\
& = \mathcal{E}(u,u)  - \langle u, Vu \rangle + 4\langle u,  \nu^2 |\nabla \nu |^2 u \rangle \pm 2 \im \langle \nabla u \cdot \nabla \nu^2, u \rangle \\
& \leq \mathcal{E}(u,u) + \| \nu u \|^2 +  4\langle u, \nu^2 |\nabla \nu |^2 u \rangle \pm  \im \mathcal{E}(u,\nu^2 u), \label{eq:inequality-im}
\end{aligned}
\end{equation}
where, in the last passage, we used the same computations as in the proof of the first part of Lemma~\ref{l:trick}. Recalling that $N= H+\lambda\nu^2+\mu$, we have
\begin{align}
\im \langle \op u, N u\rangle  =  \lambda \im \langle Hu, \nu^2 u \rangle = \lambda\im \mathcal{E}(u,\nu^2 u).
\end{align}
Hence, using \eqref{eq:inequality-im}, and the fact that $|\nabla \nu| \leq L$, we obtain
\begin{align}
\frac{1}{\lambda} | \im \langle \op u, N u\rangle | & \leq \mathcal{E}(u,u) + \| \nu u \|^2 +  4 \langle u,  \nu^2 |\nabla \nu |^2 u \rangle \\
& \leq  \langle N u,u\rangle - (\lambda -1 - 4L^2)\| \nu u\|^2 -  \mu \|u \|^2 \leq  \langle N u,u\rangle,
\end{align}
where we used the assumption on $\lambda$. Hence \eqref{eq:reed-2-new} holds with $D=\lambda$.
\end{proof}
 % Main result, its proof
\subsection{Compactness of the resolvent}

To prove the last part of Theorem~\ref{t:main}, it is sufficient to show that there exists $z \in \R$ such that the resolvent $(H^*-z)^{-1}$ on $L^2(M)$ is compact. In fact, by the first resolvent formula \cite[Thm.\ VIII.2]{MR0493419}, and since compact operators are an ideal of bounded ones, this implies the compactness of $(H^*-z)^{-1}$ for all $z$ in the resolvent set. Furthermore if $\hat{M}$ is compact, then $H$ is semibounded, that is
\begin{equation}
\langle H u, u \rangle \geq - \sup_{q \in M} \nu^2(q)\, \|u \|^2, \qquad \forall u \in \dom(\op).
\end{equation}
It is well known that the spectrum of bounded operators with compact resolvent consists of discrete eigenvalues with finite multiplicity \cite[Thm. XIII.64]{MR0493421}. Thus, the proof of Theorem~\ref{t:main} is concluded by the following proposition.

\begin{prop}\label{prop:compact-resolvent}
Let $\hat{M}$ be compact. Under the assumptions of Theorem~\ref{t:main} there exists $z \in \R$ such that the resolvent $(H^*-z)^{-1}$ on $L^2(M)$ is compact, where $H^*= \bar{H}$ is the unique self-adjoint extension of $H$.
\end{prop}

\begin{proof}
Under the assumptions of Theorem~\ref{t:main}, and thanks to the compactness of $\hat{M}$, we have $V \geq - \sup \nu^2 > - \infty$. Hence, the conclusion of Proposition \ref{prop:weak-hardy} holds. That is, there exists a constant $c \in \R$, $\kappa\geq 0$, and $0<\eta\leq 1/\kappa$ such that
\begin{equation}\label{eq:rappello}
\mathcal{E}(u) \geq \int_{M_\eta} \left(\frac{1}{\delta^2}-\frac{\kappa}{\delta}\right) |u|^2 \,d\omega + c \|u\|^2, \qquad \forall u \in W^1_\comp(M).
\end{equation}
In particular, Proposition \ref{p:Agmon} and the fact that $H^*$ is self-adjoint, yield that for all $z<c$, the resolvent $(H^*-z)^{-1}$ is well defined on $L^2(M)$, with $\|(H^*-z)^{-1}\| \leq 1/(c-z)$. 

In order to prove the compactness of $(H^*-z)^{-1}$, we need two regularity properties of functions $u \in D(H^*) \subset W^1_\loc(M)$, respectively close and far away from the metric boundary. Let $\chi_1,\chi_2$ be real valued Lipschitz functions on $[0,+\infty)$ such that
	\begin{itemize}
	\item $0 \leq \chi_i \leq 1$ for $i=1,2$;
	\item $\chi_1 \equiv 1$ on $[0,\eta/2]$ and $\chi_1 \equiv 0$ on $[\eta,+\infty)$;
	\item $\chi_2 \equiv 0$ on $[0,\eta/2]$ and $\chi_2 \equiv 1$ on $[\eta,+\infty)$;
	\item they interpolate linearly elsewhere.
	\end{itemize}
	Consider the Lipschitz functions $\phi_i:= \chi_i \circ \delta$. Notice that $\phi_1 + \phi_2 = 1$. 

Since $\hat{M}$ is compact, the support of $\phi_2$ is compact in $M$. Hence we are in the setting of Lemma~\ref{l:trick}, and we obtain
\begin{equation}
\mathcal{E}(\phi_2 u, \phi_2 u)  = \re \mathcal{E}(u, \phi_2^2 u) + \langle u,|\nabla \phi_2|^2 u \rangle = \re \langle H^*u, \phi^2_2 u \rangle + \langle u,|\nabla \phi_2|^2 u \rangle.
\end{equation}
In particular, letting $\psi = (H^*-z)u \in L^2(M)$, we obtain,
\begin{align}
\mathcal{E}(\phi_2 u, \phi_2 u) & = z \|\phi_2 u\|^2 + \re \langle \psi, \phi^2_2 u\rangle + \langle u,|\nabla \phi_2|^2 u \rangle \\
& \leq z \|u\|^2 + \|\psi\| \|u\| + 4\|u\|^2 / \eta^2, \label{eq:boundenergia}
\end{align}
where we used the fact that $|\nabla \phi_2| \leq |\chi_2^\prime| |\nabla \delta| \leq 2/\eta$. Notice also that, since $\phi_2 u \in W^1_\comp(M)$, and $\phi_2 \equiv 0$ on $M_{\eta/2}$, we have
\begin{equation}
\mathcal{E}(\phi_2 u,\phi_2 u) = \int_{M \setminus M_{\eta/2}} \left( | \nabla (\phi_2 u)|^2 + \pot |\phi_2 u|^2 \right) \,d\omega. \label{eq:energiaW1}
\end{equation}
As we already mentioned, our assumptions on the potential, and the compactness of $\hat{M}$, imply that $\pot \geq - K$ for some constant $K \geq 0$. Hence, \eqref{eq:boundenergia} and \eqref{eq:energiaW1} imply
\begin{equation}\label{eq:lontanodalbordo}
\int_{M \setminus M_{\eta/2}} | \nabla (\phi_2 u)|^2 \leq z \|u\|^2 + \|\psi\| \|u\| + 4\|u\|^2/\eta^2 + K \|u\|^2.
\end{equation}

We turn now to $\phi_1 u$. Let $u_k \in C^\infty_c(M)$ such that $\|H^* (u_k-u)\| +\|u_k - u\| \to 0$. This is possible since $H$ is essentially self-adjoint, hence $H^* = \bar{H}$. In particular $\mathcal{E}(u_k,u_k) = \langle H^* u_k,u_k\rangle \to \langle H^*u,u\rangle$. From \eqref{eq:rappello}, we have
\begin{align}
\int_{M_\eta} |u_k|^2 \, d\omega  & = \int_{M_{\eta}} \frac{\delta^2}{1-\delta\kappa} \frac{1-\delta\kappa}{\delta^2}|u_k|^2 \, d\omega \\
& \leq  \frac{\eta^2}{1-\eta \kappa} \int_{M_\eta}\left(\frac{1}{\delta^2} - \frac{\kappa}{\delta}\right) |u_k|^2 \, d \omega \\
& \leq  \frac{\eta^2}{1-\eta \kappa} \left( \mathcal{E}(u_k,u_k) - c \|u_k\|^2\right) = \frac{\eta^2}{1-\eta \kappa} \left( \langle H^* u_k,u_k\rangle - c \|u_k\|^2\right).
\end{align}
By taking the limit, and recalling that $(H^*-z)u = \psi$, we obtain,
\begin{equation}
\int_{M_\eta} |u|^2 \, d\omega \leq \frac{\eta^2}{1-\eta \kappa} \left(\langle H^* u,u\rangle - c\|u\|^2\right)= \frac{\eta^2}{1-\eta \kappa} \left((z-c) \|u\|^2 + \langle \psi,u\rangle\right).
\end{equation}
In particular, recalling that $\phi_1 \leq 1$, we obtain,
\begin{equation}\label{eq:vicinoalbordo}
\int_{M_{\eta}} |\phi_1 u|^2 d \omega \leq \int_{M_{\eta}} |u|^2 d \omega \leq \frac{\eta^2}{1-\eta \kappa} \left((z-c) \|u\|^2 + \|\psi\|\|u\|\right).
\end{equation}

We are now ready to prove that $(H^*-z)^{-1}$ is compact, with $z<c$. Let $\psi_n$ be a bounded sequence in $L^2(M)$, say $\|\psi_n\|\leq (c-z)$, and $u_n =(H^*-z)^{-1}  \psi_n \in D(H^*)$. By the boundedness of the resolvent, $\|u_n\| \leq 1$.  Let $u_n = u_{n,1} + u_{n,2}$, with $u_{n,i}:= \phi_{n,i} u_n$.

Equation \eqref{eq:lontanodalbordo} applied to $u=u_n$, $\psi=\psi_n$, implies
\begin{equation}
\|u_{n,2}\|^2_{W^1(M)} = \int_{M \setminus M_{\eta/2}} | \nabla u_{n,2}|^2\, d\omega + \|u_{n,2}\|^2  \leq c + 4/\eta^2 + K + 1.
\end{equation}
That is, $u_{n,2}$ is bounded in $W^1(M)$. Moreover, by construction, $u_{n,2} \in W^1_\comp(\Omega) \subset W^1_0(\Omega)$, where $W^1_0(\Omega)$ denotes the closure of $C^\infty_c(\Omega)$ in $W^1(\Omega)$, and $\Omega \Subset M$ is a relatively compact open subset. Then, by compact embedding of $W^1_0(\Omega)$ in $L^2(\Omega)$ \cite[Cor. 10.21]{Gregorio}, we have that $u_{n,2}$ converges,  up to extraction, in $L^2(\Omega)$, and thus in $L^2(M)$.

On the other hand, \eqref{eq:vicinoalbordo}, and taking in account the mentioned bounds (recall that $z<c$), imply that for some constant $C$ independent of $\eta$, we have
\begin{equation}
\|u_{n,1}\|^2 = \int_{M_{\eta}} |u_{n,1}|^2\, d \omega \leq \frac{\eta^2}{1-\eta \kappa} 2(c-z) \leq C \eta^2,
\end{equation}

Since $\eta$ in \eqref{eq:rappello} can be arbitrarily small, say $\tilde{\eta}_k^2 = 1/k$, we actually proved that for all $k \in \N$, there is a subsequence $n \mapsto \gamma_k(n)$ such that $u_{\gamma_k(n)} = \sum_{i=1}^2 u_{\gamma_k(n),i}$ with $\|u_{\gamma_k(n),1}\| \leq C/k$ and $u_{\gamma_k(n),2}$ convergent in $L^2(M)$. Exploiting this fact, we build a Cauchy subsequence of $u_n$, yielding the compactness of $(H^*-z)^{-1}$, and concluding the proof. 

\begin{figure}[t]
\includegraphics{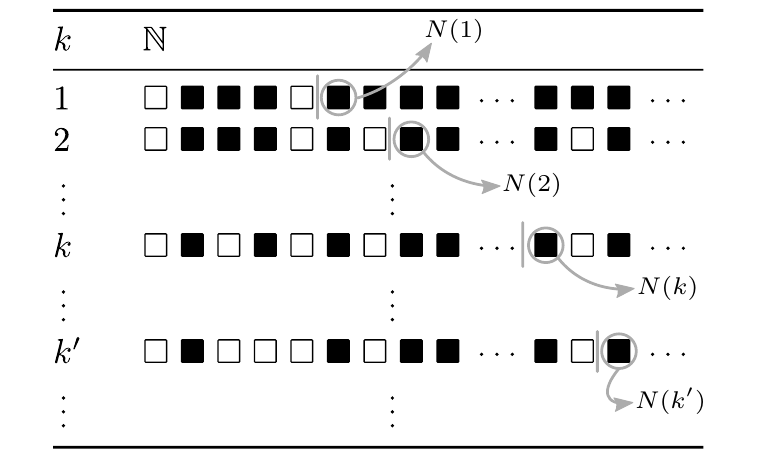}
\caption{Extraction of the sequence $n \mapsto \nu(n)$. Black and white squares denote respectively the available elements and the deleted ones.}\label{t:table}
\end{figure}

To this purpose, we build an infinite table as in Figure~\ref{t:table}. In the zeroth line, put all natural numbers, in order, from the left to the right, representing the original sequence. Recursively, each next line is a copy of the previous one, leaving an empty space corresponding to the elements that do not belong to $S_k:= \gamma_k(S_{k-1})$, with $S_0:= \N$. We obtain an infinite table where each line is a non-empty and infinite subsequence of the previous ones, and the $k$-th line represents the $\gamma_k$-th subsequence of the $k-1$-th one. 

Let $\mu: \N \times \N \to \N$ be the map
\begin{equation}
\mu(k,n) = \text{n-th element appearing in the $k$-th line}.
\end{equation}
For $k \in \N$, consider the cutoff functions $\phi_{k,i}$, with $i=1,2$, built as above with the choice $\tilde{\eta}^2_k = 1/k$. In particular, $\phi_{k,1}$ is supported in $M_{\tilde\eta_k}$, and $\phi_{k,2}$ is supported in $M_{\tilde\eta_k/2}$. Moreover, $\phi_{k,1} + \phi_{k,2} = 1$. Let $u_{\mu(k,n),i}:= u_{\mu(k,n)} \phi_{k,i}$. The localization close to the metric boundary, $u_{\mu(k,n),1}$, satisfies
\begin{equation}
\|u_{\mu(k,n),1} \| \leq C \tilde\eta_k = C/k.
\end{equation}
The localization away from the metric boundary, $u_{\mu(k,n),2}$, is Cauchy in $L^2(M)$. In particular, there exists $N(k)$ such that for all $n,m \geq N(k)$, the difference $\|u_{\mu(k,m),2} - u_{\mu(k,n),2} \| \leq 1/k$. Without loss of generality, assume that $k \mapsto N(k)$ is non-decreasing. Thus, let $\nu: \N \to \N$ the subsequence
\begin{equation}
\nu(n):= \mu(n,N(n)).
\end{equation}
We claim that $u_{\nu(n)}$ is a Cauchy sequence in $L^2(M)$. In fact, assume $k' \geq k$. Indeed, by construction of the table, $\nu(k)$ and $\nu(k')$ both appear in the $k$-th line of the aforementioned table, and $\nu(k') \geq \nu(k) \geq N(k)$. Hence, by definition of $N(k)$,
\begin{equation}
\|u_{\nu(k),2} - u_{\nu(k'),2}\| \leq 1/k.
\end{equation}
Moreover, again since $\nu(k)$ and $\nu(k')$ both appear in the $k$-th line of the table
\begin{equation}
\|u_{\nu(k),1} - u_{\nu(k'),1} \| \leq 2C/k. 
\end{equation}
In particular, since $u_{\nu(k)} = \sum_{i=1}^2 u_{\nu(k),1}$, we have
\begin{equation}
\| u_{\nu(k)} - u_{\nu(k')} \| \leq (2C+1)/k, \qquad \forall k'\geq k.
\end{equation}
This proves that the subsequence $u_{\nu(n)}$ is Cauchy in $L^2(M)$.
\end{proof}
 % compactness of resolvent
\section{Measure confinement}\label{s:degorder}

In this section, we prove essential self-adjointness results in presence of a singular or degenerate measure, and we discuss some examples where our techniques either do or do not apply. In particular, we set $V \equiv 0$, that is $H= - \Delta_\omega$. As usual, we work under the assumption $(\star)$, and we identify $M_{\varepsilon} \simeq (0,\varepsilon] \times X_\varepsilon$. Moreover, we fix a reference measure $d\mu(x)$ on $X_\varepsilon$, the choice of which is irrelevant. 

%The next results yield a general test for essential self-adjointness of $\Delta_\omega$, once an explicit expression for the measure is provided.

\begin{theorem}[Pure measure confinement I]\label{t:nonconst-ord}
	Let $(M,g)$ be a Riemannian manifold satisfying $(\star)$ for some $\varepsilon>0$.
	Let $\omega$ be a smooth measure of the form
	\begin{equation}\label{eq:degorder-nonconst}
		d\omega(t,x) = t^{a(x)} e^{2\phi(t,x)}\,dt\,d\mu(x),\qquad (t,x)\in(0,\varepsilon]\times X_\varepsilon,
	\end{equation}
where $d\mu(x)$ is a reference measure on $X_\varepsilon$. Assume that $a(x)\le -1$ or $a(x)\ge 3$ for any $x\in X_\varepsilon$ and that there exists $\kappa\ge 0$ such that
	\begin{equation}\label{eq:phi}
		a(x) \partial_t\phi + t\left((\partial_t \phi)^2 + \partial_t^2\phi\right) \ge -\kappa, \qquad \forall t\leq \varepsilon.
	\end{equation}
	Then, $\Delta_\omega$ with domain $C^\infty_c(M)$ is essentially self-adjoint on $L^2(M)$.
\end{theorem}

\begin{proof}
Recall that $\delta(t,x) = t$. Then, by Proposition~\ref{l:pot} we obtain, with $\vartheta(t,x) = \tfrac{a(x)}{2} \log t + \phi(t, x)$,
	\begin{align}
	\Veff  & = (\partial_t \vartheta)^2 + (\partial_t^2 \vartheta) \\
	& =\frac{a(x)^2-2a(x)}{4t^2} + \frac{a(x)}{t} \partial_t\phi + (\partial_t \phi)^2 + \partial_t^2\phi \\
	& \ge \frac{3}{4t^2} -\frac{\kappa}{t},
	\end{align}
	where we used \eqref{eq:phi}. The statement now follows from Theorem~\ref{t:main}. 
\end{proof}

Assumption \eqref{eq:phi} is verified whenever $X_\varepsilon$ is compact and $\phi(t,x)$ (which, in general, is smooth in $t$ and continuous in $x$ with all its derivatives w.r.t.\ $t$) can be extended to a function with the same regularity on the compact set $[0,\varepsilon]\times X_\varepsilon$. In particular, we obtain the following straightforward consequence, which is Theorem~\ref{t:degorder-intro} of the introduction.

\begin{theorem}[Pure measure confinement II]\label{t:degorder}
Assume that the Riemannian manifold $(M,g)$ satisfies $(\star)$ for $\varepsilon >0$. Moreover, let $\omega$ be a smooth measure such that there exists $a\in\R$ and a reference measure $\mu$ on $X_\varepsilon$ for which
	\begin{equation}\label{eq:porwer-ord}
		d\omega(t,x) = t^a\,dt\,d\mu(x),\qquad (t,x)\in(0,\varepsilon] \times X_\varepsilon.
	\end{equation}	
	Then, $\Delta_\omega$ with domain $C^\infty_c(M)$ is essentially self-adjoint in $L^2(M)$ if $a\ge 3$ or $a\le -1$.
\end{theorem}

The next example shows that, in general, assumption \eqref{eq:phi} must be checked carefully.

\begin{example}\label{ex:non-constant}
	Consider a measure $\omega$ given on $M_\varepsilon$ by the expression
	\begin{equation}
		d\omega(t,x) = t^m(t^\ell+f(x))^k\,dt\,d\mu(x), \qquad m,k\in\R, \, \ell \geq 0,
	\end{equation}
	for some reference measure $d\mu(x)$ on $X_\varepsilon$ and a smooth $f \geq 0$ attaining the value zero on a proper, non-empty subset of $X_\varepsilon$. This is of the form \eqref{eq:degorder-nonconst}, with
	\begin{equation}
		a(x) = \begin{cases}
			m+k\ell & f(x) = 0,\\
			m & f(x) \neq 0,
		\end{cases}
		\qquad
		\phi(t,x) = \begin{cases}
			0 & f(x) = 0, \\
			\frac k2\log\left( t^\ell + f(x)\right) & f(x) \neq 0.
		\end{cases}
	\end{equation}
In order to check assumption~\eqref{eq:phi}, let $R(t,x) := a(x) \partial_t\phi + t\left((\partial_t \phi)^2 + \partial_t^2\phi\right)$. We have,
	\begin{equation}
R(t,x) = \begin{cases} 
0 & f(x) = 0, \\
\frac{k\ell t^{\ell-1} \left((k\ell+2 m-2) t^\ell+2 (\ell+m-1) f (x)\right)}{4 \left(t^\ell+f (x)\right)^2} & f(x) \neq 0.
\end{cases}
	\end{equation}
To check assumption \eqref{eq:phi}, we consider two particular cases.
\begin{enumerate}
\item[1.] $m\ge 3$ and $k\geq 0$. In this case, $a(x)\ge 3$. Then, one can check that $R(t,x)\ge0$ for all $x \in X_\varepsilon$. Thus, by Theorem~\ref{t:nonconst-ord}, the operator $\Delta_\omega$ is essentially self-adjoint.
\item[2.] $m\le-1$ and $k<0$. In this case, $a(x)\le-1$, and the applicability of Theorem~\ref{t:nonconst-ord} depends in a crucial way on the relation between $m$ and $\ell$:
	\begin{itemize}
	\item If $\ell \le 1-m$, then $R(t,x)\ge0$, and assumption~\eqref{eq:phi} is satisfied. In particular, by Theorem~\ref{t:nonconst-ord}, the operator $\Delta_\omega$ is essentially self-adjoint.
	\item If $\ell > 1-m$, then along any sequence $(t_i,x_i)$ such that $t_i=1/i$ and $f(x_i) = 1/i^\ell$, we have that $R(t_i,x_i)\to -\infty$. Hence, we cannot apply Theorem~\ref{t:nonconst-ord}.
	\end{itemize}
\end{enumerate}	
\end{example}
\section{Applications to strongly singular potentials}\label{s:singular}

In this section we prove Theorem~\ref{t:quattro-autori-intro}, regarding the essential self-adjointness of a Schr\"o\-din\-ger operator $H=-\Delta +V$, where $\Delta=\Delta_{\vol_g}$ is the Laplace-Beltrami operator, and whose potential is singular along submanifolds of arbitrary dimension. We restate it here for the reader's convenience.

\begin{theorem}[Kalf-Walter-Schmincke-Simon for Riemannian submanifolds]
Let $(N,g)$ be a $n$-dimensional, complete Riemannian manifold. Let $\mathcal{Z}_i \subset N$, with $i \in I$, be a finite collection of embedded, compact $C^2$ submanifolds of dimension $k_i$ and denote by $d(\cdot,\mathcal{Z}_i)$ the Riemannian distance from $\mathcal{Z}_i$. Let $\pot \in L^2_\loc(N \setminus \mathcal{Z}_i)$ be a \emph{strongly singular potential}. That is, there exists $\varepsilon>0$ and a non-negative Lipschitz function $\nu : N \to \R$, such that,
\begin{enumerate}
	\item[(i)]  for all $i\in I$ and $p \in N$ such that $0< d(p,\mathcal{Z}_i)\leq \varepsilon$, we have
	\begin{equation}\label{eq:conditionsingular}
		\pot(p) \geq -\frac{(n-k_i)(n-k_i-4)}{4 d(p,\mathcal{Z}_i)^2} - \frac{\kappa}{d(p,\mathcal{Z}_i)}- \nu(p)^2, \qquad \kappa\geq 0;
	\end{equation}
	\item[(ii)] for all $p \in N$ such that $d(p,\mathcal{Z}_i) \geq \varepsilon$ for all $i \in I$, we have
	\begin{equation}
		\pot(p) \geq - \nu(p)^2,
	\end{equation}
\end{enumerate}
Then, the operator $H = -\Delta + \pot$ with domain $C^\infty_c(M)$ is essentially self-adjoint in $L^2(M)$, where $M = N \setminus \bigcup_{i} \mathcal{Z}_i$, or any one of its connected components.
\end{theorem}

\begin{proof}[Proof of Theorem~\ref{t:quattro-autori-intro}]
Since $N$ is complete, $M = N \setminus \bigcup_{i} \mathcal{Z}_i$ is a non-complete smooth Riemannian manifold whose metric boundary is $\bigcup_i \mathcal{Z}_i$, and
\begin{equation}
\delta(p) = \min_{i\in I} d(p,\mathcal{Z}_i).
\end{equation}
Since each $\mathcal{Z}_i$ is a $C^2$ compact submanifold, there exists $\varepsilon>0$ such that $\delta$ is $C^2$ on each $U_i=\{0<d(p,\mathcal{Z}_i)< \varepsilon\}$, see \cite{Footesmooth}. Hence, hypothesis $(\star)$ is satisfied.

We use Fermi coordinates from the submanifold $\mathcal{Z}_i$, see \cite{Gray}, which are the generalization in higher codimension of Riemannian normal coordinates from a point. In particular, for each $q \in \mathcal{Z}_i$ there is a coordinate neighborhood $\mathcal{O} \simeq \R^{k_i}_x \times \R^{n-k_i}_y$ of $q$ such that $\mathcal{Z}_i \cap \mathcal{O} \simeq \{(x,y)\mid y = 0\}$ and $d((x,y),\mathcal{Z}_i) = |y|$.

Taking polar coordinates $(t,\theta)$ on the $\R^{n-k_i}_y$ part of the Fermi coordinates, $\delta(x,t,\theta) = t$, and the Riemannian measure reads $\vol_g = t^{n-k_i-1} | b(x,t,\theta) \, dx\, dt \, d\theta |$, with $b(x,0,\theta) \neq 0$. Up to taking a smaller $\mathcal{O}$, we assume that $b \geq C > 0$ and the derivatives of $b$ are bounded.
Using the definition \eqref{eq:potential}, and taking in account that $\nabla \delta = \partial_t$, we obtain
\begin{align}
\Veff 
& = \left(\dive(\partial_t)/2\right)^2 + \partial_t \left(\dive(\partial_t)/2\right) \\
& = \frac{(n-k_i-1) (n-k_i-3)}{4 t^2} -\frac{t (\partial_t b)^2+2 b \left((k_i-n+1) \partial_t b-t \partial_t^2 b\right)}{4 t b^2} \\
& \geq \frac{(n-k_i-1) (n-k_i-3)}{4 \delta^2} - \frac{\kappa}{\delta}, \qquad \text{on } \mathcal{O} \setminus \mathcal{Z}_i,
\end{align}
for some constant $\kappa \geq 0$. By compactness of the $\mathcal{Z}_i$'s, and up to modifying the constant $\kappa$, the above estimate holds on $M_\varepsilon = \{ 0<\delta \leq \varepsilon\}$. We conclude by applying Theorem~\ref{t:main}, and using the assumptions on $\pot$.
\end{proof}
\section{Curvature and self-adjointness}\label{s:curvature}

In this section, $\omega=\vol_g$, and $\Delta = \Delta_\omega$ is the Laplace-Beltrami operator. Our aim is to prove two criteria for the essential self-adjointness of $\Delta$, Theorems~\ref{t:curvosc} and \ref{t:curvoscsuper}, which imply Theorem~\ref{t:curv-intro}, presented in the introduction. As discussed there, the blow-up of the sectional curvature at the metric boundary alone is not a sufficient condition. 

For fixed $0<t\leq \varepsilon$, recall that $X_t = \{ \delta = t\}$ is a $C^2$ hypersurface. The $(2,0)$ symmetric tensor $\mathrm{Hess}(\delta)$, the Riemannian Hessian, describes the extrinsic curvature of the level sets $X_t$ in $M$. More precisely, for any fixed $t$, its restriction to $X_t$,
\begin{equation}
H(t) := \Hess(\delta)|_{X_t},
\end{equation}
is the \emph{second fundamental form}
of $X_t$. The eigenvalues of $H(t)$ on $T_p X_t$ are the \emph{principal curvatures} of $X_t$ at $p$, and are denoted by $h_i(t)$, $i=1,\ldots,n-1$. 
Finally, for any $p \in M$, the sectional curvature of a plane $\sigma \subset T_p M$ is denoted by $\mathrm{Sec}(\sigma)$.
\begin{rmk}
When $n=2$, the sectional curvature reduces to the Gauss curvature $\kappa$ of the surface $M$. Moreover $X_t$ is a curve, and $h_1(t)$ is its signed geodesic curvature, where the sign is computed with respect to the direction $-\nabla\delta$.
\end{rmk}

We consider a general setting in which the sectional curvature blows up with a power law. In particular, there exist admissible bands of oscillation, see Figure~\ref{f:admissibleregion}, whose size increases with the dimension $n$.

\begin{figure}
\includegraphics[width=.4\textwidth]{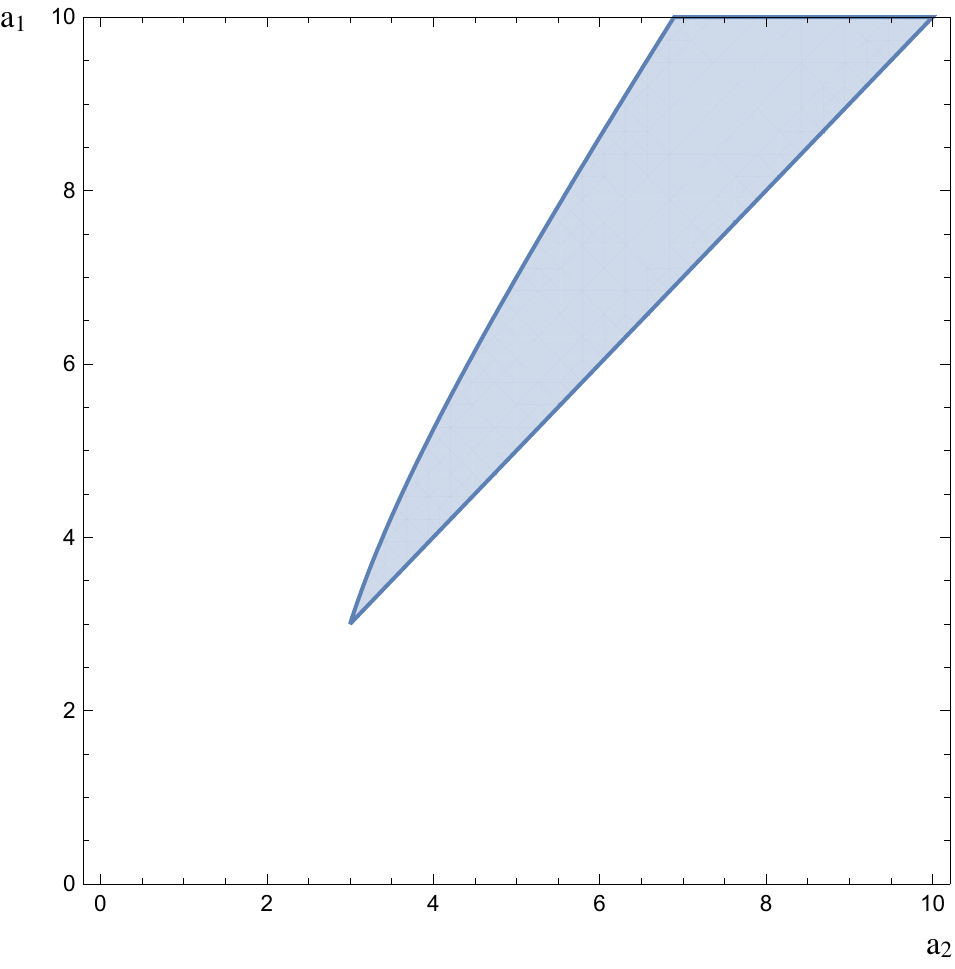} \hspace{.1\textwidth}
\includegraphics[width=.4\textwidth]{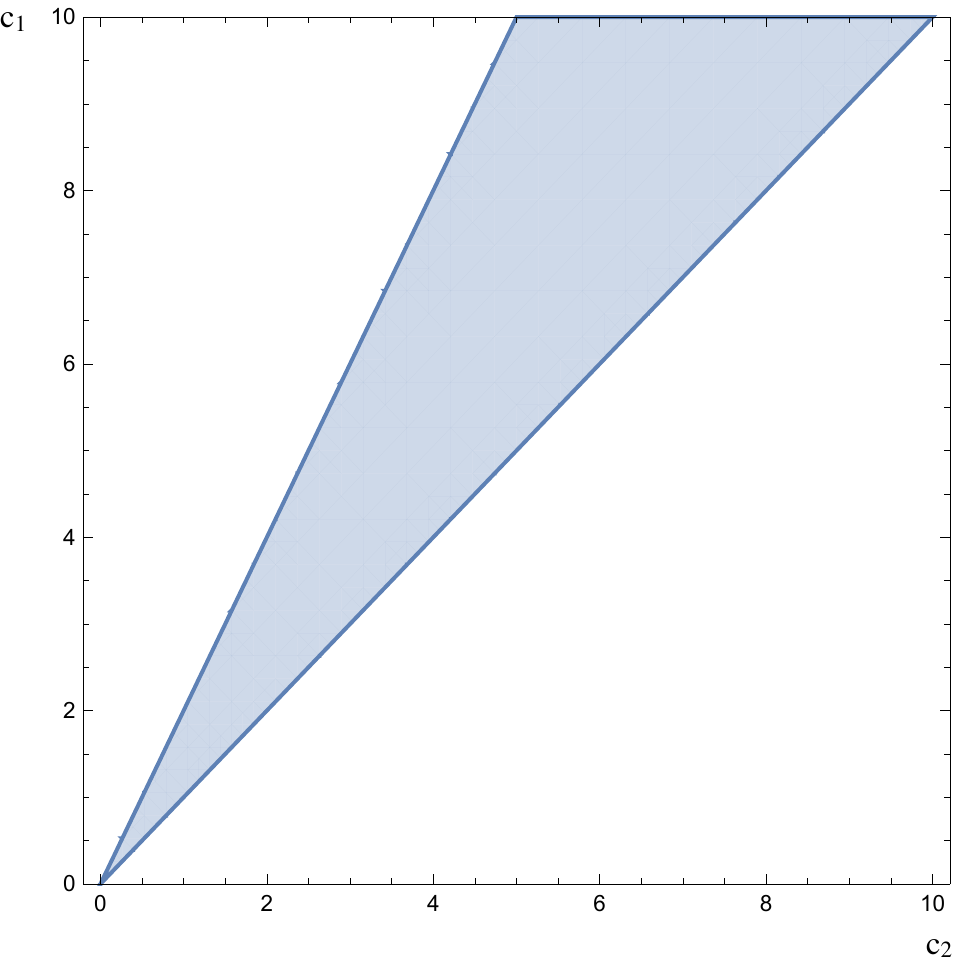}
\caption{Admissible region for the parameters, for $n=2$, in case of quadratic (left), and super-quadratic (right) curvature explosion.}\label{f:admissibleregion}
\end{figure}

\begin{theorem}[Quadratic curvature explosion]\label{t:curvosc}
Let $(M,g)$ be a Riemannian manifold satisfying $(\star)$ for $\varepsilon>0$. Assume that there exist $a_1\geq a_2 > 1$ such that, for all planes $\sigma$ containing the vector $\nabla \delta$, we have
\begin{equation}\label{eq:boundcurv}
-\frac{a_1^2-1}{4\delta^2} \leq \mathrm{Sec}(\sigma) \leq -\frac{a_2^2-1}{4\delta^2}, \qquad \delta \leq \varepsilon.
\end{equation}
Moreover, assume that the principal curvature of the hypersurface $X_\varepsilon=\{\delta = \varepsilon\}$ satisfies\footnote{In \eqref{eq:boundprincurv}, and similarly \eqref{eq:boundprincurvsuper}, the inequality $H(\varepsilon)<\alpha$, for $\alpha \in \R$, is understood in the sense of quadratic forms, that is, for all $q \in X_\varepsilon$ and $X \in T_q X_{\varepsilon}$, we have $H(\varepsilon)(X,X) < \alpha g(X,X)$.}
\begin{equation}\label{eq:boundprincurv}
H(\varepsilon) < \frac{1+a_2}{2\varepsilon}.
\end{equation}
Then, the operator $\Delta$ with domain $C^\infty_c(M)$ is essentially self-adjoint in $L^2(M)$ if
\begin{equation}
\frac{n-1}{16}\left[2(a_2^2-1) -(1-a_1)^2 +(n-2)(1-a_2)^2\right] \geq \frac{3}{4}.
\end{equation}
\end{theorem}

As soon as the rate of explosion of the sectional curvature is more than quadratic, we get a simpler self-adjointness criterion.
\begin{theorem}[Super-quadratic curvature explosion]\label{t:curvoscsuper}
Let $(M,g)$ be a Riemannian manifold satisfying $(\star)$ for $\varepsilon>0$. Let $r>2$ and assume that there exist $c_1\geq c_2>0$ such that for all planes $\sigma$ containing the vector $\nabla \delta$, we have
\begin{equation}\label{eq:boundcurvsuper}
-\frac{c_1}{\delta^r} \leq \mathrm{Sec}(\sigma) \leq -\frac{c_2}{\delta^r}, \qquad \delta \leq \varepsilon.
\end{equation}
Then, there exists a constant $h_\varepsilon^*(c_2,r)>0$ such that, if the principal curvature of the hypersurface $X_\varepsilon=\{\delta = \varepsilon\}$ satisfies
\begin{equation}\label{eq:boundprincurvsuper}
H(\varepsilon) \leq h_\varepsilon^*(c_2,r),
\end{equation}
the operator $\Delta$ with domain $C^\infty_c(M)$ is essentially self-adjoint in $L^2(M)$ whenever 
\begin{equation}
0<c_2 \leq c_1 < n c_2 .
\end{equation}
\end{theorem}
\begin{rmk}\label{r:critconstant}
For completeness, the explicit value of the constant $h_\varepsilon^*(c,r)$, expressed in terms of the modified Bessel functions $K_\nu(z)$, is
\begin{equation}
h_\varepsilon^*(c,r)=\frac{1}{2\varepsilon} - \frac{\sqrt{c}}{\varepsilon^{r/2}}\frac{K'_{1/(r-2)}(2\sqrt{c} \varepsilon^{1-r/2}/(r-2))}{K_{1/(r-2)}(2\sqrt{c} \varepsilon^{1-r/2}/(r-2))}.
\end{equation}
Notice that, by Lemma~\ref{l:boundingsol2}, the map $c \mapsto h_\varepsilon^*(c,r)$ is monotone increasing. 
\end{rmk}

\begin{rmk}
By the known asymptotics for the Bessel function \cite[Eqs.\ 9.7.8 and 9.7.10]{AbSteg}, one can check that the condition \eqref{eq:boundprincurvsuper} tends to the corresponding one \eqref{eq:boundprincurv} for $r \to 2^+$. However, we stress that Theorem~\ref{t:curvosc} is \emph{not} a limit case of Theorem~\ref{t:curvoscsuper}. Indeed, the proof of these results is based on a control on the asymptotic behavior of the effective potential, which does not pass to the limit.
\end{rmk}

\subsection{Proofs of curvature-based criteria}

Fix $x \in X_\varepsilon$, and let $\gamma: (0,\varepsilon] \to M$ be the geodesic such that $\gamma(\varepsilon) = x$, and $\dot\gamma(t) = \nabla\delta(\gamma(t))$, for which $\delta(\gamma(t)) = t$. Let $V_1,\ldots,V_{n-1},\nabla\delta$ be an orthonormal, parallel transported frame along $\gamma(t)$. {With a slight abuse of notation, we denote with
\begin{equation}
H(t) = \Hess(\delta)(V_i,V_j)|_{\gamma(t)}, \qquad i,j=1,\ldots,n-1,
\end{equation}
the $(n-1 )\times (n-1)$ symmetric matrix representing $\Hess(\delta)|_{X_t}$ along $\gamma$. By Lemma~\ref{l:dist}, $H$ is a solution of the matrix Riccati equation,
\begin{equation}
H^\prime + H^2 + R = 0, \qquad R(t) = g(R^\nabla(V_i,\nabla \delta)\nabla\delta,V_j)|_{\gamma(t)}.
\end{equation}
We will use the following ``backwards'' version of the classical Riccati comparison theorem, which follows directly from the analogous ``forward'' statement in \cite{Royden}.

\begin{lemma}[Riccati comparison]\label{l:riccaticomp}
Assume that $R_1(t) \leq R(t) \leq R_2(t)$ for some families $R_i(t)$ of symmetric matrices. Let $H_1(t)$ and $H_2(t)$ be solutions of
\begin{equation}\label{eq:riccatieq}
H^\prime_i + H^2_i + R_i = 0, \qquad i=1,2,
\end{equation}
both defined on a common maximal interval of the form $(t_*,\varepsilon]$, with initial conditions satisfying $H_1(\varepsilon) \leq H(\varepsilon) \leq H_2(\varepsilon)$. Then,
\begin{equation}
H_1(t) \leq H(t) \leq H_2(t), \qquad \forall t \in (t_*,\varepsilon].
\end{equation}
The statement remains true when all inequalities concerning $H$ are strict.
\end{lemma}

\begin{figure}
\includegraphics[width=\textwidth]{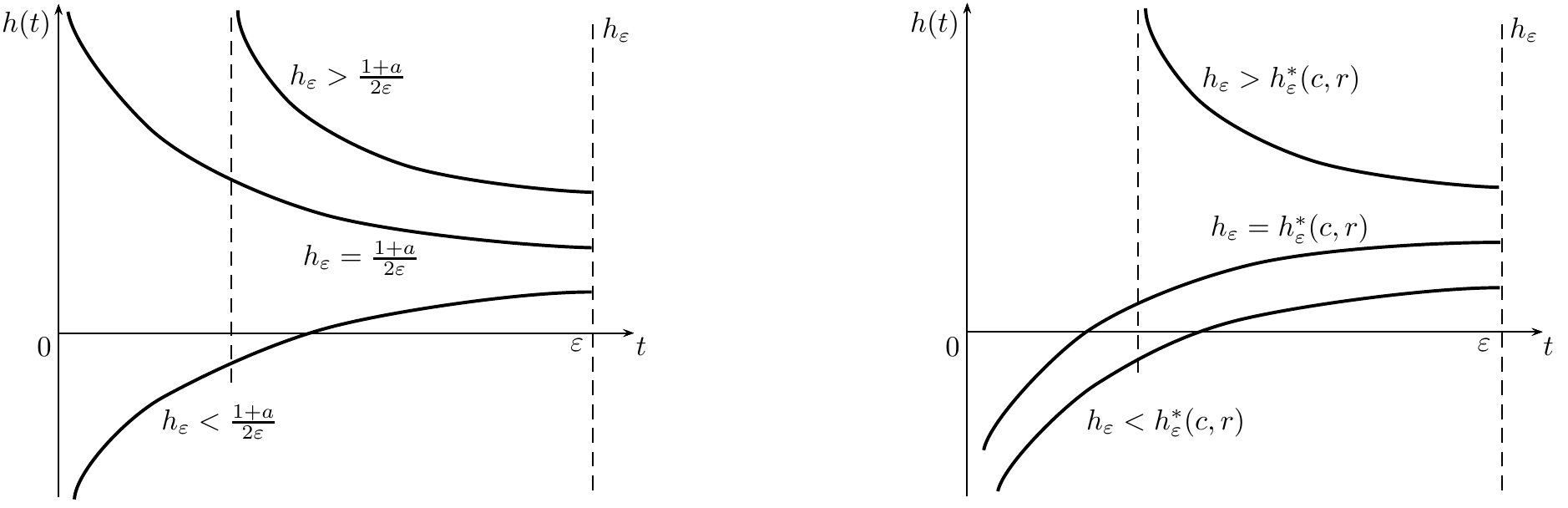}
\caption{Solutions of the Riccati equation in the quadratic (left) and super-quadratic case (right). In both cases, the blow-up time depends on the position of the initial datum with respect to a critical threshold.}\label{f:Riccatisolutions}
\end{figure}
\begin{lemma}[Exact solutions: quadratic case]\label{l:boundingsol}
Let $m \in \R$, $a>1$, and consider the backwards Riccati Cauchy problem:
\begin{equation}\label{eq:lemmariccati}
h^\prime+h^2 -\frac{a^2-1}{4t^2} = 0, \qquad h(\varepsilon) = \frac{1+a}{2\varepsilon} + \frac{m}{2\varepsilon}.
\end{equation}
Its unique solution is defined on a maximal interval $(t_*,\varepsilon]$, with $t_* \geq 0$. The blow-up time is $t_* =0$ if and only if $m \leq 0$ and, in this case, we have the asymptotic behavior
\begin{equation}
h(t) \sim \frac{1}{2t}\times\begin{cases}
1+a & m=0, \\
1-a & m<0,
\end{cases} \qquad \text{for } t \to 0^+.
\end{equation}
In particular, the solution blows-up at $+\infty$ for $m=0$ and at $-\infty$ for $m < 0$.
\end{lemma}
\begin{proof}
One can check that the  unique solution of the backwards Riccati equation is
\begin{equation}
h(t) = \frac{1}{2t}\frac{(2a+m)(a+1)(t/\varepsilon)^a +m(a-1)}{(2a+m)(t/\varepsilon)^a -m }.
\end{equation}
We observe that $h(t)$ has a simple pole at $t=0$, and thus $0 \leq t_* < \varepsilon$, with
\begin{equation}
\left(\frac{t_*}{\varepsilon}\right)^a = \begin{cases}
\frac{m}{2a+m} & \text{if } m>0, \\
0 &  \text{if } m \leq 0.
\end{cases}
\end{equation}
All the other statements follow from straightforward computations.
\end{proof}

For the next statement, we use the modified Bessel functions $I_\nu(z)$ and $K_\nu(z)$, which are real and positive for $\nu > -1$ and $z>0$, see \cite[Sec. 9.6]{AbSteg}.

\begin{lemma}[Exact solutions: super-quadratic case]\label{l:boundingsol2}
Let $c>0$ and $r>2$. Consider the backwards Riccati Cauchy problem:
\begin{equation}\label{eq:lemmariccatisuper}
h^\prime+h^2 -\frac{c}{t^r} = 0, \qquad h(\varepsilon) = h_\varepsilon.
\end{equation}
Its unique solution is defined on a maximal interval $(t_*,\varepsilon]$, with $t_* \geq 0$. The blow-up time is $t_*=0$ if and only if
\begin{equation}
h_\varepsilon \leq h_\varepsilon^*(c,r):=\frac{1}{2\varepsilon} - \frac{\sqrt{c}}{t^{r/2}}\frac{K'_\nu(\tau(\varepsilon))}{K_\nu(\tau(\varepsilon))}, \quad \text{where} \quad \tau(t) := \frac{2\nu \sqrt{c}}{t^{1/2\nu}}, \quad \nu := \frac{1}{r-2},
\end{equation}
in which case, we have the asymptotic behavior
\begin{equation}
h(t) \sim - \frac{\sqrt{c}}{t^{r/2}} \qquad \text{for } t \to 0^+.
\end{equation}
Moreover, the map $c \mapsto h_\varepsilon^*(c,r)$ is monotone increasing.
\end{lemma}
\begin{proof}
By replacing in the Riccati equation \eqref{eq:lemmariccatisuper} the ansatz
\begin{equation}\label{eq:ansatz}
h(t) = \frac{1}{2t}- \frac{\sqrt{c}}{t^{r/2}}\frac{w'(\tau(t))}{w(\tau(t))}, 
\end{equation}
we obtain that $h(t)$ is a solution if $z \mapsto w(z)$ satisfies the modified Bessel equation,
\begin{equation}
z^2 w''(z) + z w'(z) - ( z^2 + \nu^2 ) w(z) = 0.
\end{equation}
The modified Bessel functions $I_{\nu}$ and $K_\nu$ are a basis of solutions of the above. In the following, recall that $I_\nu$ and $K_\nu$ are real and positive for $\nu>-1$ and $z>0$ \cite[Sec. 9.6]{AbSteg}. Thus, the general solution of the Riccati equation \eqref{eq:lemmariccatisuper} is given by \eqref{eq:ansatz}, with
\begin{equation}
w = a I_\nu + b K_{\nu}, \qquad a,b \in \R.
\end{equation}
Consider first the case $a=0$, corresponding to the solution with initial datum $h_\varepsilon = h_\varepsilon^*(c,r)$,
\begin{equation}
h(t) = \frac{1}{2t} - \frac{\sqrt{c}}{t^{r/2}}\frac{K_\nu'(\tau(t))}{K_\nu(\tau(t))}.
\end{equation}
Since $K_\nu(z) >0$ for $\nu > -1$ and $z >0$, we have $t_* = 0$.

We proceed by assuming that $a \neq 0$,  and thus $h_\varepsilon \neq h_\varepsilon^*(c,r)$. In particular, since \eqref{eq:ansatz} is invariant under linear rescaling of $w$, we fix $a =-1$. Routine computations show that \eqref{eq:ansatz} is the unique solution corresponding to initial condition $h_\varepsilon$ if
\begin{equation}
w = - I_\nu + b K_\nu, \qquad \text{with} \qquad b = \frac{I_\nu(\tau(\varepsilon))}{K_\nu(\tau(\varepsilon))} \times \frac{h_\varepsilon-\tilde{h}_\varepsilon^*(c,r)}{h_\varepsilon-h_\varepsilon^*(c,r)},
\end{equation}
where we set
\begin{equation}
\tilde{h}_\varepsilon^*(c,r)=\frac{1}{2\varepsilon} - \frac{\sqrt{c}}{\varepsilon^{r/2}}\frac{I'_\nu(\tau(\varepsilon))}{I_\nu(\tau(\varepsilon))}. % \qquad h_\varepsilon^*(c,r)=\frac{1}{2\varepsilon} - \frac{\sqrt{c}}{\varepsilon^{r/2}}\frac{K'_\nu(\tau(\varepsilon))}{K_\nu(\tau(\varepsilon))}.
\end{equation}
Such a solution has a blow-up time $0<t_* <\varepsilon$ if and only if $t_*<\varepsilon$ is solution of
\begin{equation}\label{eq:tstar}
w(\tau(t_*)) = 0 \qquad \Leftrightarrow \qquad  b = \frac{I_\nu(\tau(t_*))}{K_\nu(\tau(t_*))}.
\end{equation}
We claim that the above hold if and only if  $h_\varepsilon > h_\varepsilon^*(c,r)$. In fact, using the relations \cite[Eqs.\ 9.6.26]{AbSteg}, and the fact that $I_\nu(z),K_\nu(z)>0$ for $\nu >-1$ and $z>0$, we deduce that
\begin{equation}\label{eq:inequalitiesbessel}
I_\nu'(z)>0, \qquad \text{and} \qquad K_\nu'(z) <0, \qquad \text{for } \nu>0,\, z>0.
\end{equation}
As a consequence the map $z \mapsto I_\nu(z)/K_\nu(z)$ is monotone increasing. By \cite[Eqs.\ 9.6.7 and 9.6.9]{AbSteg}, we have $\lim_{z \to 0^+} I_\nu(z)/ K_\nu(z) = 0$. Moreover, the function $t \mapsto \tau(t)$ is monotone decreasing. Then \eqref{eq:tstar} will have a solution $0<t_* < \varepsilon$ if and only if
\begin{equation}
b > \frac{I_\nu(\tau(\varepsilon))}{K_\nu(\tau(\varepsilon))}.
\end{equation}
Replacing the explicit expression for $b$, the above condition is equivalent to
\begin{equation}
\frac{h_\varepsilon^*(c,r)-\tilde{h}_\varepsilon^*(c,r)}{h_\varepsilon-h_\varepsilon^*(c,r)}>0.
\end{equation}
By \eqref{eq:inequalitiesbessel}, the numerator of the l.h.s.\ of the above is strictly positive, hence we have blow-up time $0<t_*< \varepsilon$ if $h_\varepsilon> h_\varepsilon^*(c,r)$, as claimed. On the other hand, if $h_\varepsilon \leq h_\varepsilon^*(c,r)$, equation \eqref{eq:tstar} has no solution for $0<t_* < \varepsilon$, and from \eqref{eq:ansatz} we see that the solution $h(t)$ blows-up at $t_* = 0$.

Finally, for any choice of the parameters $a,b$ in $w=a I_\nu + b K_\nu$, we have the asymptotic behavior \cite[Eq.s 9.7.1 and 9.7.3]{AbSteg}:
\begin{equation}
\frac{w'(z)}{w(z)} = 1 +  O(z^{-1}), \qquad z \to +\infty,
\end{equation}
concluding the first part of the proof.

Finally, we prove that $c\mapsto h^*_\varepsilon(c,r)$ is monotone increasing, for any fixed $r>2$ and $\varepsilon>0$. This is implied by the following property of modified Bessel functions:
\begin{equation}\label{eq:propertyBessel}
f(z) = -z \frac{K_\nu'(z)}{K_\nu(z)}\qquad \text{is monotone increasing for all $z>0$ and $\nu \in \R$}.
\end{equation}
Observe that $f(z)$ is well defined as $K_\nu(z) = K_{-\nu}(z)>0$ for all $z>0$ and $\nu \in \R$. In order to prove property \eqref{eq:propertyBessel}, we compute:
\begin{equation}
f'(z) = \frac{z^2 K_\nu'(z)^2 - (z^2+\nu^2) K_\nu(z)^2}{z K_\nu(z)^2},
\end{equation}
where we used the modified Bessel equation to cancel the second derivatives of $K_\nu(z)$.
As we already observed, $K_\nu'(z) <0$ for all $\nu \in \R$ and $z>0$. The fact that $f'(z)>0$ for $z>0$ then follows from the subtle inequality
\begin{equation}
z K_\nu'(z) /K_\nu(z) < - \sqrt{z^2+\nu^2}, \qquad \text{for all $z>0$ and $\nu \in \R$},
\end{equation}
which is proved in \cite[Eq. 2.2]{MR2439440} using a Tur\'an type inequality and a clever trick.
\end{proof}

\begin{proof}[Proof of Theorem~\ref{t:curvosc}]
Let $h_{1,\varepsilon}$ and $h_{2,\varepsilon}$ be, respectively, the smallest and largest eigenvalue of $H(\varepsilon)$. Thanks to the assumption on the sectional curvature \eqref{eq:boundcurv}, we can apply the Riccati comparison result of Lemma~\ref{l:riccaticomp} with 
\begin{gather}
-\frac{a_1^2-1}{4t^2}\,\mathbbold{1} =:R_1(t)  \leq R(t) \leq R_2(t) := -\frac{a_2^2-1}{4t^2} \mathbbold{1}, \\
h_{1,\varepsilon}\, \mathbbold{1} =: H_1(\varepsilon) \leq H(\varepsilon) \leq H_2(\varepsilon):= h_{2,\varepsilon} \,\mathbbold{1},
\end{gather}
which yields $H_1(t) \leq H(t) \leq H_2(t)$, where $H_i(t) = h_i(t)\, \mathbbold{1}$, and $h_i(t)$ is the solution of
\begin{equation}
h^\prime_i + h_i^2 -\frac{a_i^2-1}{4t^2} = 0, \qquad h_i(\varepsilon) = h_{i,\varepsilon} = \frac{1+a_i}{2\varepsilon} + \frac{m_i}{2\varepsilon},
\end{equation}
where the $m_i$'s are defined by the last equality. By the assumption on $H(\varepsilon)$, we have
\begin{equation}
h_{1,\varepsilon} \leq h_{2,\varepsilon} < \frac{1+a_2}{2\varepsilon} \leq \frac{1+a_1}{2\varepsilon}.
\end{equation}
In particular $m_1 \leq  m_2 < 0$. Then, by Lemma~\ref{l:boundingsol}, both solutions $h_i(t)$ are defined on $(0,\varepsilon]$ and have asymptotic behavior 
\begin{equation}\label{eq:asympth}
h_i(t) \sim \frac{1-a_i}{2t} , \qquad \text{for } t \to 0^+.
\end{equation}
For the effective potential along the given geodesic, using Riccati equation, we obtain
\begin{equation}\label{eq:effpotproof}
\Veff = \left(\frac{\tr H}{2}\right)^2 + \left(\frac{\tr H}{2}\right)^\prime = \frac{1}{4}\left[(\tr H)^2 - 2\tr(H^2) - 2 \tr(R) \right].
\end{equation}
The ``curvature component'' of \eqref{eq:effpotproof} is bounded thanks to our curvature assumptions:
\begin{equation}
-2\tr(R) \geq (n-1)\frac{a_2^2-1}{2t^2}.
\end{equation}
By \eqref{eq:asympth}, $h_i(t) \to -\infty$ for $t \to 0^+$ and $i=1,2$. In particular, possibly taking a smaller $\varepsilon$, we have that $h_1(t) \mathbbold{1} \leq H(t) \leq h_2(t) \mathbbold{1} < 0$ on $(0,\varepsilon]$. Denote with $\lambda_j$, for $j=1,\ldots,n-1$ the eigenvalues of $H$. Indeed, we have, for any value of $t \in (0,\varepsilon]$ the inequalities
\begin{equation}\label{eq:stimagood}
h_1 \leq \lambda_j \leq h_2 <0 \qquad \implies \qquad |h_2| \leq |\lambda_j| \leq |h_1|.
\end{equation}
Then, for the ``Hessian component'' of the effective potential \eqref{eq:effpotproof}, we get
\begin{equation}\label{eq:stimagood2}
\begin{aligned}
(\tr H)^2 - 2\tr(H^2) & = \bigg(\sum_{j=1}^{n-1} \lambda_j\bigg)^2 - 2 \sum_{j=1}^{n-1} \lambda_j^2 \\
& = - \sum_{j=1}^{n-1} \lambda_j^2 + 2 \sum_{1\leq i<j \leq n-1}  \lambda_i \lambda_j \\
& =  - \sum_{j=1}^{n-1} \lambda_j^2 + 2\sum_{1\leq i<j\leq n-1}  |\lambda_i | |\lambda_j  | \\
& \geq - (n-1) h_1^2 + (n-1)(n-2)h_2^2 .
\end{aligned}
\end{equation}
Thus, up to taking an possibly smaller $\varepsilon$, there exists $\kappa \geq 0$ such that
\begin{align}
\Veff & \geq \frac{n-1}{4}\left(\frac{a_2^2-1}{2t^2} -h_1^2 +(n-2)h_2^2 \right)  \\
 & \geq \frac{n-1}{16t^2}\left[2(a_2^2-1) -(1-a_1)^2 +(n-2)(1-a_2)^2 \right] - \frac{\kappa}{t}, \qquad \forall t \leq \varepsilon,
\end{align}
where, in the second line, we used the asymptotics of $h_i(t)$. Then, by Theorem~\ref{t:main}, $\Delta$ is essentially self-adjoint if $\Veff \geq \tfrac{3}{4t^2}- \frac{\kappa}{t}$, which yields the statement.
\end{proof}

\begin{proof}[Proof of Theorem~\ref{t:curvoscsuper}]
The proof follows the same comparison ideas of the one of Theorem~\ref{t:curvosc}. We apply the Riccati comparison result of Lemma~\ref{l:riccaticomp} with 
\begin{gather}
-\frac{c_1}{t^r}\,\mathbbold{1} =:R_1(t)  \leq R(t) \leq R_2(t) := -\frac{c_2}{t^r} \mathbbold{1}, \\
h_{1,\varepsilon} \, \mathbbold{1} =: H_1(\varepsilon) \,\mathbbold{1} \leq H(\varepsilon) \leq H_2(\varepsilon):= h_{2,\varepsilon}\,\mathbbold{1},
\end{gather}
which yields $H_1(t) \leq H(t) \leq H_2(t)$, where $H_i(t) = h_i(t)\, \mathbbold{1}$ and $h_i(t)$ is the solution of
\begin{equation}
h^\prime_i + h_i^2 -\frac{c}{t^r} = 0, \qquad h_i(\varepsilon) = h_{i,\varepsilon} = h^*_{\varepsilon}(c_i,r) + m_i,
\end{equation}
where the $m_i$'s are defined by the last equality. By the assumption on $H$, we have
\begin{equation}
h_{1,\varepsilon}\leq h_{2,\varepsilon} \leq h^*_{\varepsilon}(c_2,r) \leq h^*_{\varepsilon}(c_1,r).
\end{equation}
The last inequality follows since $c \mapsto h^*_\varepsilon(c,r)$ is monotone increasing by Lemma~\ref{l:boundingsol2}, and $c_1 \geq c_2$ by assumption.
In particular $m_1 \leq  m_2 \leq 0$. Then, by Lemma~\ref{l:boundingsol2}, both solutions $h_i(t)$ are defined on $(0,\varepsilon]$ and have asymptotic behavior 
\begin{equation}\label{eq:asympth-super}
h_i(t) \sim -\frac{\sqrt{c_i}}{t^{r/2}} , \qquad \text{for } t \to 0^+.
\end{equation}
For the effective potential along the given geodesic, using Riccati equation, we obtain
\begin{equation}\label{eq:effpotproof-super}
\Veff = \frac{1}{4}\left[(\tr H)^2 - 2\tr(H^2) - 2 \tr(R) \right].
\end{equation}
By \eqref{eq:asympth-super}, $h_i(t) \to -\infty$ for $t \to 0^+$ and $i=1,2$. Hence, up to taking a smaller $\varepsilon$, the same argument leading to the estimate \eqref{eq:stimagood2} holds. In particular, we obtain
\begin{equation}
\Veff \geq \frac{n-1}{4}\left(- \frac{2\tr(R)}{n-1} - h_1^2 + (n-2)h_2^2 \right).
\end{equation}
Up to taking a possibly smaller $\varepsilon$, there exists $\kappa \geq 0$ such that
\begin{equation}
\Veff \geq \frac{(n-1)}{4 t^r}\left( nc_2 - c_1 - \kappa t \right), \qquad \forall t \leq \varepsilon,
\end{equation}
where we used the asymptotics \eqref{eq:asympth-super} and the assumption on the curvature. Recall that $r>2$. Then, if $n c_2 > c_1$ we can apply Theorem~\ref{t:main}, yielding the statement.
\end{proof}
 % curvature results and their proof
\section{Almost-Riemannian geometry}\label{s:arg}

In this section we show that assumption $(\star)$ is verified for almost-Riemannian structures with no tangency points, we prove Theorem~\ref{t:regularARS-intro} for regular ARS, and then we discuss some examples of non-regular ARS and open problems.

\subsection{Preliminaries on almost-Riemannian structures}

Almost-Riemannian geometry has been introduced in \cite{ABS-Gauss-Bonnet} and describes a large class of singular Riemannian structures. Roughly speaking, an almost-Riemannian structure on a smooth $n$-dimensional manifold $N$ is locally given by a generating family of smooth vector fields $X_1,\ldots,X_n$. In the regular region where the rank of this family is maximal, it defines a Riemannian structure which however is singular on the set where some of them become linearly dependent.
\begin{definition}
Let $N$ be a smooth and connected manifold of dimension $n$. An almost-Riemannian structure (ARS) on $N$ is a triple $\mathcal{S}=(E,\xi,\cdot)$, where $\pi_E : E \to N$ is a vector bundle of rank $n$, and $\cdot$ is a smooth scalar product on the fibers of $E$. Finally, $\xi: E \to TN$ is a vector bundle morphism. That is, $\xi$ is a fiber-wise linear map such that, letting $\pi : TN \to N$ be the canonical projection,  the following diagram commutes:
\begin{equation*}
\begin{tikzcd}
E \arrow{rr}{\xi}\arrow[swap]{dr}{\pi_E}
& & TN \arrow{ld}{\pi} \\
& N 
\end{tikzcd}
\end{equation*}
Moreover, we assume the \emph{Lie bracket generating condition}, that is
\begin{equation}\label{eq:horm}
\mathrm{Lie}(\xi(\Gamma(E)))|_q = T_q N, \qquad \forall q \in N,
\end{equation}
where $\Gamma(E)$ denotes the $C^\infty(N)$-module of smooth sections of $E$, and $\mathrm{Lie}(\xi(\Gamma(E)))|_q$ denotes the smallest Lie algebra containing $\xi(\Gamma(E)) \subseteq \Gamma(TN)$, evaluated at $q$.
\end{definition}

Consider a set $\sigma_1,\ldots,\sigma_n$ of smooth local sections of $E$, defined on $\mathcal{O} \subseteq N$, and orthonormal with respect to the scalar product on $E$. The vector fields $X_i:= \xi \circ \sigma_i$ constitute a \emph{local generating family}. On $\mathcal{O}$, condition \eqref{eq:horm} reads
\begin{equation}\label{eq:horm2}
\mathrm{Lie}(X_1,\ldots,X_n)|_q = T_q N , \qquad \forall q \in \mathcal{O}.
\end{equation}

When possible, an efficient way to define an ARS is by giving a \emph{global generating family} of smooth vector fields $X_1,\ldots,X_n \in \Gamma(TN)$ satisfying \eqref{eq:horm2}. In fact, by setting $E = N \times \R^n$ and letting $\sigma_i(p) = (p,e_i)$, for $i=1,\ldots,n$, there exists a unique vector bundle morphism $\xi:E \to TN$ such that $X_i = \xi \circ \sigma_i$. Then, the ARS structure defined on $N$ by the global generating family is $\mathcal{S} = (E,\xi,\cdot)$, where $\cdot$ is the standard Euclidean product on the fibers of $E$.

The \emph{subspace of admissible directions} at $q \in N$ is $\distr_q := \xi(E_q) \subseteq T_q N$, where $E_q = \pi_E^{-1}(q)$.
The \emph{singular set} $\mathcal{Z} \subset N$ is the set of points where $\dim \distr_q < n$.

\begin{definition}
Assume that the singular set $\mathcal{Z}$ is a smooth embedded hypersurface. A point $q \in \mathcal{Z}$ is a \emph{tangency point} if $\distr_q \subseteq T_q\mathcal{Z}$.
\end{definition}
Tagency points have deep consequences on the local structure of the almost-Riemannian metric structure, and have been studied, in the $2$-dimensional case, in \cite{ARS-tangency,Sphere-tangency-2D}. If $\mathcal{Z}$ is a smooth, embedded submanifold, for all $q \in \mathcal{Z}$ there exists a non-zero $\lambda \in T_q^*N$, defined up to multiplication by a constant, such that $\lambda(T_q \mathcal{Z}) = 0$. Thus, $q$ is a tangency point if and only if $\lambda(\distr_q) = 0$.

\subsubsection{Almost-Riemannian metric structure}
For any $q \in N$ and $v \in \distr_q$, define the norm
\begin{equation}
|v|_{}^2:= \inf\{ v \cdot v  \mid u \in E_q, \quad \xi(u) = v\}.
\end{equation}
One can check that the above norm satisfies the parallelogram law, and hence it is defined by a scalar product on $\distr_q$, denoted with the symbol $g$. In particular, $g$ is a  smooth Riemannian metric on the regular region $M = N \setminus \mathcal{Z}$, but is singular on $\mathcal{Z}$ where $\distr_q \subset T_q M$ strictly. Notice that any local generating family $X_1,\ldots,X_n$ is orthonormal with respect to $g$ on the regular region. Despite the singularity of $g$, one can define a global metric structure on $N$ as we now explain.

Let $I$ be an interval. An absolutely continuous curve $\gamma : I \to N$ is \emph{admissible} if $\dot\gamma(t) \in \distr_{\gamma(t)}$ for a.e. $t \in I$. In this case, its \emph{length} is
\begin{equation}
\ell(\gamma) := \int_I |\dot\gamma(t)|_{}\,dt.
\end{equation}
Since $\ell$ is invariant under reparametrization of $\gamma$, when dealing with minimization of length we consider only intervals of the form $I=[0,T]$, for some fixed $T>0$. We define the \emph{almost-Riemannian distance} as
\begin{equation}
d_{\mathcal{S}}(p,q):=\inf\{\ell(\gamma)\mid \gamma \text{ is admissible},\quad \gamma(0) = p, \quad \gamma(1) = q\}.
\end{equation}
Under the bracket-generating condition \eqref{eq:horm}, the Chow-Rashevskii Theorem implies that $d_{\mathcal{S}}:N \times N \to \R$ is finite and continuous (see, e.g., \cite{nostrolibro}). Thus, $N$ is admissible-path connected and the metric space $(N,d_{\mathcal{S}})$ has the same topology of $N$. We say that the ARS is \emph{complete} if $(N,d_{\mathcal{S}})$ is complete as a metric space. Notice that, being it a locally compact length space, completeness is equivalent to the compactness of all closed balls, and implies the existence of admissible minimizing curves between any pair of points \cite[Thm. 2.5.28]{BBI}, possibly crossing the singular region $\mathcal{Z}$.

\subsubsection{Almost-Riemannian gradient}
The gradient of a smooth function $f$ is the smooth vector field $\nabla f \in \Gamma(\distr)$ such that
\begin{equation}
g(\nabla f,W) = W(f) = d f (W), \qquad \forall W \in \Gamma(\distr).
\end{equation}

Indeed, $\nabla f$ coincides with the Riemannian gradient on the complement of $\mathcal{Z}$. The gradient is smooth as a consequence of the next formula.
\begin{lemma}
If $X_1,\ldots,X_n$ is a local generating family for the ARS, then
\begin{equation}
\nabla f = \sum_{i=1}^n X_i(f) X_i, \qquad |\nabla f |_{}^2 = \sum_{i=1}^n X_i(f)^2.
\end{equation}
\end{lemma}
\begin{rmk}
The relevance of the above formula, and also of Lemma~\ref{l:speed} below, is that they hold also on the singular set $\mathcal{Z}$, where $X_1,\ldots,X_n$ are not independent.
\end{rmk}
\begin{proof}
Let $V,W \in \distr_q$, such that $V= \xi(v)$ and $W=\xi(w)$, with $w,v \in E_q$. Let $\Pi :E_q \to E_q$ be the orthogonal projection on $\ker \xi|_{E_q}$. In particular, $|V|_{} = | v - \Pi(v)|_{E_q}$ and $|W |_{} = |w - \Pi(w)|_{E_q}$, where $|\cdot|_{E_q}$ denotes the norm on $E_q$. By polarization, we obtain
\begin{equation}\label{eq:polarization}
g(V,W) = v \cdot w - \Pi(v) \cdot \Pi(w).
\end{equation}
We fix the representative $w^* :=w - \Pi(w)$ for $W=\xi(w^*)$ with the property $\Pi(w^*) = 0$. Moreover, fix $V = \sum_{i=1}^n v_i X_i$, with $v_i:=X_i(f)$. By \eqref{eq:polarization}, we have
\begin{equation}
g(V,W) = g(\sum_{i=1}^n v_i X_i, \sum_{j=1}^n w_j^* X_j) = \sum_{i=1}^n v_i w_i^* = \sum_{i=1}^n X_i(f) w_i^* = W(f).
\end{equation}
Since this holds for any $W \in \distr_q$, we obtain the statement.
\end{proof}

\subsubsection{Geodesics and Hamiltonian flow}
We recall basic notions on minimizing curves in almost-Riemannian geometry. This is a particular case of the length minimization problem on rank-varying sub-Riemannian structures, and we refer to \cite{nostrolibro,AS-GeometricControl} for further details.

A \emph{geodesic} is an admissible curve $\gamma :[0,T] \to N$ that locally minimizes the length between its endpoints. For what concerns necessary conditions for optimality, define the \emph{almost-Riemannian Hamiltonian} as the smooth function $H : T^*N \to \R$ such that
\begin{equation}
H(\lambda) := \frac{1}{2}\sum_{i=1}^n \langle \lambda, X_i \rangle^2, \qquad \lambda \in T^*N,
\end{equation}
where $X_1,\ldots,X_n$ is a local generating family for the ARS, and $\langle \lambda, \cdot \rangle $ denotes the action of covectors on vectors. If $\sigma$ denotes the canonical symplectic $2$-form on $T^*N$, the \emph{Hamiltonian vector field} $\vec{H}$ is defined by $\sigma(\cdot, \vec{H}) = dH$. Then, Hamilton's equations are
\begin{equation}\label{eq:Hamiltoneqs}
\dot{\lambda}(t) = \vec{H}(\lambda(t)).
\end{equation}
Solutions $\lambda : [0,T] \to T^*N$ of~\eqref{eq:Hamiltoneqs} are called \emph{normal extremals}, their projections $\gamma(t) := \pi(\lambda(t))$ on $N$ are locally minimizing curves, and are called \emph{normal geodesics}.
\begin{lemma}\label{l:speed}
Let $\lambda(t) =e^{t\vec{H}}(\lambda)$ be a normal extremal, and $\gamma(t) = \pi(\lambda(t))$ be the corresponding normal geodesic. Its tangent vector is given by
\begin{equation}
\dot{\gamma}(t) = \sum_{i=1}^n \langle\lambda(t),X_i(\gamma(t))\rangle X_i(\gamma(t)),
\end{equation}
and its speed is given by $| \dot\gamma(t) |_{} = \sqrt{2H(\lambda)}$. In particular $\ell(\gamma|_{[0,T]}) = \sqrt{2H(\lambda)} T$. 
\end{lemma}
\begin{proof}
In canonical coordinates $(p,x)$ in a neighborhood of $\lambda$, we denote $\lambda(t) = (p(t),x(t))$. In particular, by Hamilton's equations, we have
\begin{equation}
\dot{x}(t) = \frac{\partial H}{\partial p}(p(t),x(t)) = \sum_{i=1}^n (p^*(t) X_i(x(t))) X_i(x(t)),
\end{equation}
which yields the first formula. To prove the second statement, observe that
\begin{equation}
H(\lambda) = \max_{u \in \R^n} \sum_{i=1}^n \left(u_i \langle \lambda, X_i \rangle - \frac{1}{2}u_i^2\right).
\end{equation}
In particular, for any $v \in \R^n$ such that $\sum_{i=1}^n v_i X_i = \sum_{i=1}^n \langle \lambda, X_i \rangle X_i$, we have
\begin{equation}
H(\lambda)  \geq \sum_{i=1}^n \left(v_i \langle \lambda, X_i \rangle - \frac{1}{2}v_i^2 \right)= \sum_{i=1}^n \left(\langle \lambda,X_i\rangle^2 - \frac{1}{2} v_i^2 \right) = 2H(\lambda) - \frac{1}{2} \sum_{i=1}^n v_i^2.
\end{equation}
Hence, $2H(\lambda) \leq \sum_{i=1}^n v_i^2$, and we have equality if and only if $v_i = \langle \lambda,X_i\rangle$. This means that $2H(\lambda)$ realizes the almost-Riemannian squared norm of $\dot\gamma=\sum_{i=1}^n \langle \lambda, X_i \rangle X_i$.
\end{proof}

Let $q \in N$ and $p \in N \setminus \mathcal{Z}$. A standard argument employing the Lagrange multipliers rule shows that minimizing geodesics joining $p$ with $q$ must be normal geodesics. In particular, this is the case for any curve minimizing the length between $\mathcal{Z}$ and $N \setminus \mathcal{Z}$. When $p$ and $q$ are both in $\mathcal{Z}$, the presence of the so-called \emph{abnormal} geodesics must be taken in account. These are another class of minimizing curves, well known in sub-Riemannian geometry, that might not follow the Hamiltonian dynamic of~\eqref{eq:Hamiltoneqs}. Since we never deal with the distance between two points on $\mathcal{Z}$, abnormal geodesics do not play any role in what follows.

\begin{definition}
The \emph{exponential map} $\exp_{q} : D_q \to N$, with base $q \in N$ is
\begin{equation}
\exp_q (\lambda) := \pi \circ e^{\vec{H}}(\lambda), \qquad \lambda \in D_q,
\end{equation}
where $D_q \subseteq T_q^* N$ is the set of covectors such that the solution $t \mapsto e^{t \vec{H}}(\lambda)$ of~\eqref{eq:Hamiltoneqs} with initial datum $\lambda$ is well defined up to time $T=1$. 
\end{definition}
If $(N,d_{\mathcal{S}})$ is complete, by the Hopf-Rinow theorem for length spaces \cite[Thm. 2.5.28]{BBI}, normal geodesics can be prolonged on the interval $[0,+\infty)$, and $D_q = T_q^*N$ for all $q \in N$.

\begin{rmk}
In the Riemannian region $M = N \setminus \mathcal{Z}$, due to the canonical identification $TM \simeq T^*M$, the exponential map defined above is just the ``dual'' of the standard exponential one, and Hamilton's equations \eqref{eq:Hamiltoneqs} are equivalent to the Riemannian geodesic equations. However, the duality fails on $\mathcal{Z}$, and only the ``cotangent'' viewpoint survives.
\end{rmk}

\subsection{Almost-Riemannian metric structure versus metric completion}

Consider a complete al\-most-Rieman\-nian structure $\mathcal{S}$ on a smooth manifold $N$, with singular set $\mathcal{Z}$ consisting of a smooth embedded hypersurface. On $M = N \setminus \mathcal{Z}$, we consider the induced Riemannian metric structure, hereby denoted $(M,d_g)$ to avoid confusion. When $\mathcal{Z} \neq \emptyset$, the metric $d_g$ is different from the restriction of the almost-Riemannian one to $M$ and, as a consequence, the metric completion $(\hat{M},\hat{d}_g)$ is different from $(N,d_{\mathcal{S}})$.
\begin{example}
Consider the torus $\mathbb{S}^1_\theta \times \mathbb{S}^1_{\varphi}$, with the ARS given by the global generating family $\{\partial_{\theta}$, $\sin(\theta/2)^2 \partial_{\varphi}\}$. In this case $\mathcal{Z} = \{0\} \times \mathbb{S}^1$, and $\hat{M}$ is a closed cylinder $[0,2\pi]\times \mathbb{S}^1$. Let $p_\pm= (\pm \theta,\varphi)$, with $\theta \in (0,\pi/2]$.  Then $d_g(p_+,p_-) = 2\pi - 2\theta$, while the AR distance is $d_{\mathcal{S}}(p_+,p_-) = 2\theta$. On the other hand, for $\theta \in [\pi/2,\pi]$, the two distances coincide.
\end{example}
In order to apply Theorem~\ref{t:main} on the Riemannian region, and in particular to verify assumption $(\star)$, we exploit the relation between the almost-Riemannian metric structure $(N,d_{\mathcal{S}})$ and the metric completion $(\hat{M},\hat{d}_g)$.

Recall that points of $\hat{M}$ are represented by equivalence classes of Cauchy sequences of $(M,d_g)$ which, in particular, are also Cauchy sequences for $d_\mathcal{S}$. Then, consider the map $\pi : \hat{M} \to N$, which assigns to the Cauchy sequence $\{p_n\} \in (\hat{M},\hat{d}_g)$ its limit in $(N,d_\mathcal{S})$. Since $d_\mathcal{S} \leq d_g$ for points in $M$, the map $\pi$ is well defined and
\begin{equation}\label{eq:ineq1}
d_\mathcal{S}(\pi(q),\pi(p)) \leq \hat{d}_g(q,p), \qquad \forall q,p \in \hat{M}.
\end{equation}
In particular, $\pi$ is continuous. By identifying points of $M$ with constant sequences, we have $M \subset \hat{M}$, and the restriction $\pi|_M$ is the identity. Notice that, if $q,p$ belong to different connected components of $\hat{M}$ (which might occur even if we assumed that $N$ is connected), the inequality \eqref{eq:ineq1} is strict, as $\hat{d}_g(q,p) = +\infty$, while $d_{\mathcal{S}}$ is always finite.

Even though $d_\mathcal{S}$ and $\hat{d}_g$ do not agree on $M$, the distance from the metric boundary
\begin{equation}
\delta(p) := \inf\{\hat{d}_g(\tilde{q},p) \mid \tilde{q} \in \partial \hat M\}, \qquad p \in \hat{M},
\end{equation}
and the almost-Riemannian distance from $\mathcal{Z}$,
\begin{equation}
\delta_{\mathcal{S}}(p):= \inf\{d_{\mathcal{S}}(q,p) \mid q \in \mathcal{Z}\}, \qquad p \in N,
\end{equation}
do agree, as a consequence of the next Lemma. We stress that the following holds true even in presence of tangency points.

\begin{lemma}\label{l:relation} 
For any complete ARS, the following equality holds,
\begin{equation}\label{eq:eq2}
\delta(p) = \delta_{\mathcal{S}}(\pi(p)), \qquad \forall p \in \hat{M}.
\end{equation}
\end{lemma}
\begin{proof}
By completeness of $(N,d_{\mathcal{S}})$, we have $\pi(\partial\hat{M}) \subseteq \mathcal{Z}$, thus if $p \in \partial \hat{M}$ \eqref{eq:eq2} is verified as both sides are zero. Then, assume $p \in M$. Using \eqref{eq:ineq1}, we obtain the following inequality,
\begin{align}
\delta_{\mathcal{S}}(\pi(p)) & = \inf\{d_\mathcal{S}(q,\pi(p)) \mid q \in \mathcal{Z}\} \\
& \leq \inf\{d_\mathcal{S}(\pi(\tilde{q}),\pi(p)) \mid \tilde{q} \in \partial\hat{M}\} \leq \inf\{\hat{d}_g(\tilde{q},p)\mid \tilde{q} \in \partial\hat{M}\} = \delta(p).
\end{align}

To conclude the proof, we show that $\delta_{\mathcal{S}}(\pi(p)) \geq \delta(p)$. Let $\gamma : [0,1] \to N$ be an admissible curve such that $\gamma(0)\in \mathcal{Z}$ and $\gamma(1) = \pi(p)$. Without loss of generality, we assume that $\gamma(0)$ is the only point in the curve that belongs to $\mathcal{Z}$ (otherwise we can cut and reparametrize the curve, obtaining a new one with smaller length and verifying the assumptions). Let $q = \gamma(0)$. The sequence $q_n = \gamma(1/n)$ is Cauchy in $(M,d_g)$, hence it corresponds to a unique $\tilde{q} \in \partial \hat{M}$ and $q_n \to \tilde{q}$ as elements of $\hat{M}$. Thus,
\begin{equation}\label{eq:92}
\ell(\gamma) = \lim_{n \to +\infty} \ell(\gamma|_{[1/n,1]}) \geq \lim_{n \to +\infty} d_g(q_n,p) =\hat{d}_g(\tilde{q},p) \geq \delta(p),
\end{equation}
where, in the first inequality, we used that the curve $\gamma|_{[1/n,1]}$ belongs to a unique connected component of $M$. By taking the inf over all such $\gamma$, we obtain $d_{\mathcal{S}}(\pi(p)) \geq \delta(p)$.
\end{proof}

\subsection{Smoothness of the almost-Riemannian distance from the singular set}

From now on, we consider the restriction of $\delta_{\mathcal{S}}$ to $M$ and, for this reason, we omit the map $\pi$. Thanks to Lemma~\ref{l:relation}, in order to verify assumption $(\star)$, it is sufficient to study the regularity properties of the almost-Riemannian distance from $\mathcal{Z}$. As a byproduct of the proof of Lemma~\ref{l:arssmoothness}, we build a local frame useful for the computation of the effective potential in the almost-Riemannian setting, given in Lemma~\ref{l:localframe}.

\begin{lemma}\label{l:arssmoothness}
Let $\mathcal{S}$ be an ARS on an $n$-dimensional manifold $N$. Assume that the singular set $\mathcal{Z}$ is a smooth, embedded and compact hypersurface, with no tangency points. Then there exists $\varepsilon>0$ such that $\delta_{\mathcal{S}} : M_{\varepsilon} \to \R$ is smooth, where $M_\varepsilon = \{ 0<\delta_\mathcal{S}\leq \varepsilon\}$.
\end{lemma}

\begin{lemma}\label{l:localframe}
Under the same assumptions of the previous lemma, for any $q \in \mathcal{Z}$ there exist a neighborhood $\mathcal{O} \subseteq N$ of $q$ and coordinates $\mathcal{O} \simeq (-\varepsilon,\varepsilon) \times \R^{n-1}$ such that $\delta_{\mathcal{S}}(t,x) = |t|$, and a local generating family of the form
\begin{equation}
X_1 = \partial_t, \qquad X_i = \sum_{j=2}^n a_{ij}(t,x) \partial_{x_j}, \qquad i=2,\ldots,n,
\end{equation}
for some smooth functions $a_{ij}(t,x)$, such that $\det(a_{ij})(t,x) = 0$ if and only if $t=0$.
\end{lemma}

\begin{proof}[Proof of Lemma~\ref{l:arssmoothness}]
This is the almost-Riemannian version of the tubular neighborhood theorem for $\mathcal{Z}$. Let $A\mathcal{Z}$ be the annihilator bundle of the singular set. That is,
\begin{equation}
A\mathcal{Z} := \{(q,\lambda) \in T^*N \mid \lambda(T_q \mathcal{Z}) = 0 \}.
\end{equation}
This is a rank 1 vector bundle with base $\mathcal{Z}$, and the map $i_0  : \mathcal{Z} \to A\mathcal{Z}$ such that $i_0(q) = (q,0)$ is an embedding of $\mathcal{Z}$ onto the zero section of $A\mathcal{Z}$ (see Figure~\ref{f:fig3}).

Let $0 \neq \lambda \in A_q\mathcal{Z}$. Since $q$ is not a tangency point, $\lambda(\distr_q) \neq 0$, hence $H(\lambda) >0$. In particular, $\lambda \in A_q\mathcal{Z}$ is associated, using $H|_q : T_q^*N \to T_q N$, with a non-zero vector $v_q \in \distr_q$ transverse to $T_q \mathcal{Z}$, and $A\mathcal{Z}$ plays the role of the ``normal bundle'' usually employed for the construction of the tubular neighborhood.

Let $D \subseteq T^*N$ be the set of $(q,\lambda)$ such that $\exp_q(\lambda)$ is well defined. Indeed, $D$ is open and so is $D \cap A\mathcal{Z}$. Consider the map $E : A \mathcal{Z} \cap D \to N$, given by
\begin{equation}
E(q,\lambda) := \exp_q(\lambda) = \pi\circ e^{\vec{H}}(\lambda).
\end{equation}
Clearly, $i_0(\mathcal{Z}) \subset D$, and $E\circ i_0 = id_{\mathcal{Z}}$. Moreover, $E$ has full rank on $i_0(\mathcal{Z})$. In fact,
\begin{equation}
E(q+\delta q,0) = q + \delta q, \qquad E(q,\delta\lambda) = \delta v \neq 0,
\end{equation}
where we used the fact that, for $\delta\lambda \in A \mathcal{Z}$, $H(\delta\lambda) > 0$.

Since $\dim(A\mathcal{Z}) = \dim(N)$, and by the inverse function theorem, $E$ is a diffeomorphism on a neighborhood of $(q,0) \in A \mathcal{Z}$ which can be taken of the form
\begin{equation}
U_{\varrho}(q)= \{(q',\lambda') \mid d_\mathcal{S}(q,q') < \varrho,\; \sqrt{2H(\lambda')} < \varrho\}, \qquad \varrho>0.
\end{equation}
Here, we used the fact that $\mathcal{Z}$ is embedded, and that $2H$, restricted to the fibers of $A\mathcal{Z}$, is a well defined norm. For any $q \in \mathcal{Z}$, let 
\begin{equation}
\varepsilon(q) := \sup \{ \varrho>0 \mid E : U_{\varrho}(q) \to E(U_{\varrho}(q)) \text{ is a diffeomorphism} \} > 0.
\end{equation}
The function $\varepsilon : \mathcal{Z} \to \R_+$ is continuous, that is,
\begin{equation}\label{eq:ineqs}
|\varepsilon(q)-\varepsilon(q')| \leq d_{\mathcal{S}}(q,q'), \qquad \forall q,q' \in \mathcal{Z}.
\end{equation}
To prove it, assume without loss of generality that $\varepsilon(q) \geq \varepsilon(q')$. If $d_\mathcal{S}(q,q') \geq \varepsilon(q)$, then \eqref{eq:ineqs} clearly holds. On the other hand, if $d_{\mathcal{S}}(q,q') < \varepsilon(q)$, one can check using the triangle inequality for $d_{\mathcal{S}}$, that $U_{\varrho}(q') \subseteq U_{\varepsilon(q)}(q)$ for $\varrho =\varepsilon(q) - d_{\mathcal{S}}(q,q')$. Hence \eqref{eq:ineqs} holds.
\begin{figure}
\includegraphics[scale=0.88]{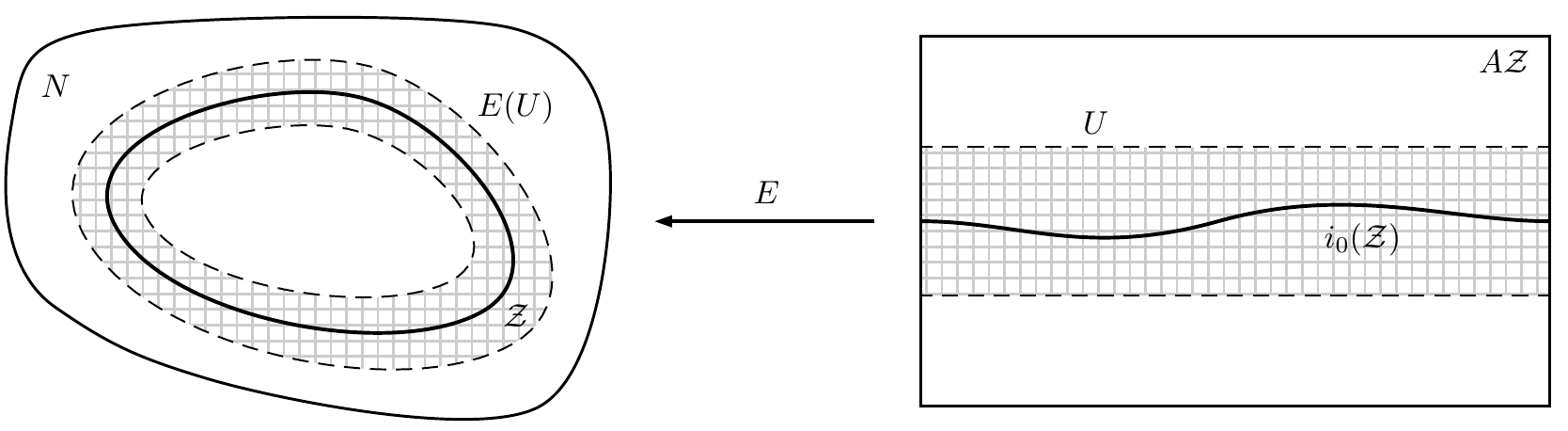}
\caption{Tubular neighborhood of $\mathcal{Z}$.}\label{f:fig3}
\end{figure}

Thanks to the compactness of $\mathcal{Z}$, we define the open neighborhood of $i_0(\mathcal{Z})$:
\begin{equation}
U:=\{(q,\lambda) \in A\mathcal{Z} \mid \sqrt{2H(\lambda)} < \varepsilon_0 \}, \qquad \varepsilon_0 := \min \{\varepsilon(q)/2 \mid q \in \mathcal{Z}\}>0.
\end{equation}
We claim that the restriction of $E$ to $U$ is injective. To prove it, let $(q_i,\lambda_i) \in U$, for $i=1,2$, with $p = E(q_i,\lambda_i)$. The normal geodesics $\gamma_i :[0,1] \to N$ defined by $\gamma_i(t) = E(q_i,t \lambda_i)$ have length $\ell(\gamma_i) = \sqrt{2H(\lambda_i)}$ by Lemma~\ref{l:speed}. Without loss of generality, we assume that $\varepsilon(q_1) \leq \varepsilon(q_2)$. By the triangle inequality,
\begin{equation}
d_{\mathcal{S}}(q_1,q_2) \leq d_{\mathcal{S}}(q_1, p) + d_{\mathcal{S}}(q_2,p) \leq \ell(\gamma_1) + \ell(\gamma_2) < 2\varepsilon_0 \leq \varepsilon(q_2).
\end{equation}
Hence, both $(q_i,\lambda_i) \in U_{\varepsilon(q_2)}(q_2)$. Since $E$ is injective on $U_{\varepsilon(q_2)}(q_2)$, then $(q_1,\lambda_1) = (q_2,\lambda_2)$, proving the claim. 

In particular, $E: U \to E(U)$ is a smooth diffeomorphism. Notice that, by construction, $E(U) \subseteq \{ \delta_{\mathcal{S}} < \varepsilon_0\}$. By compactness of $\mathcal{Z}$, and up to taking a smaller $\varepsilon_0$, we can assume that $E(U) \subseteq \{ \delta_{\mathcal{S}} < \varepsilon_0\} \subset K$, where $K$ is a compact set.

We will now prove that, in fact, $E(U) = \{ \delta_{\mathcal{S}} < \varepsilon_0\}$ and that, on $E(U)$, the almost-Riemannian distance from $\mathcal{Z}$ satisfies
\begin{equation}\label{eq:formuladistance}
\delta_{\mathcal{S}}(E(q,\lambda)) = \sqrt{2H(\lambda)}.
\end{equation}

To this purpose, let $p \in \{ \delta_{\mathcal{S}} < \varepsilon_0\} \subset K$. Since $K$ is compact, there exists at least one admissible curve $\gamma:[0,1] \to N$ minimizing the almost-Riemannian distance between $\mathcal{Z}$ and $p$. This must be a normal geodesic, that is $p=E(q,\lambda)$, with $q \in \mathcal{Z}$ and $\lambda \in T_{q}^*N$. Standard variation formulas show that, if there exists a direction $w \in T_{q} \mathcal{Z}$ with $\lambda_{q}(w) \neq 0$, then one can deform $\gamma$ in the direction of $w$, keeping its initial point in $\mathcal{Z}$, and decreasing its length. Since $\gamma$ is minimizing, this implies $\lambda(T_{q} \mathcal{Z}) = 0$, that is $(q,\lambda) \in A\mathcal{Z}$. Moreover, $\sqrt{2H(\lambda)} = \ell(\gamma) =  \delta_{\mathcal{S}}(p) < \varepsilon_0$. This implies that $(q,\lambda) \in U$, that is $p = E(q,\lambda) \in E(U)$, and $\delta_{\mathcal{S}}(E(q,\lambda)) = \sqrt{2H(\lambda)}$, as claimed.

Since $E$ maps the set $i_0(\mathcal{Z})=\{\lambda \in A\mathcal{Z} \mid 2H(\lambda) = 0\}$ onto $\mathcal{Z}$, \eqref{eq:formuladistance} together with the definition of $U$, imply that $\delta_{\mathcal{S}}$ is smooth on the set $\{0< \delta_{\mathcal{S}} \leq \varepsilon\}$, for all $\varepsilon < \varepsilon_0$.
\end{proof}

\begin{proof}[Proof of Lemma~\ref{l:localframe}]
In the proof of Lemma~\ref{l:arssmoothness}, we built a tubular neighborhood of $\mathcal{Z}$, that is a diffeomorphism $E: U \to E(U)$ from a neighborhood of the zero section
\begin{equation}
U  = \{ (q,\lambda) \in A \mathcal{Z} \mid \sqrt{2H(\lambda)} < \varepsilon \} \subset A \mathcal{Z},
\end{equation}
to a neighborhood $E(U)$ of $\mathcal{Z}$, such that $\delta_{\mathcal{S}}(E(q,\lambda)) = \sqrt{2H(\lambda)}$. In particular,
\begin{equation}
E(U)  = \{ p \in N  \mid \delta_{\mathcal{S}}(p) < \varepsilon \}.
\end{equation}
Let $W \subseteq \mathcal{Z}$ a coordinate neighborhood and $\eta: W \to A\mathcal{Z}$ be a smooth non-vanishing local section of $A\mathcal{Z}$, with $2H(\eta) = 1$. We identify $W \simeq \R^{n-1}$ with coordinates $x$. 

The map $(t,x) \mapsto E(x,t \eta(x))$ yields coordinates $(-\varepsilon, \varepsilon) \times \R^{n-1}$ on a neighborhood $\mathcal{O} \subseteq N$ of $W$. The curves $\tau \mapsto E(x,\tau\eta(x))= \exp_x(\tau\eta(x))$ are the unique normal geodesics with speed equal to $2H(\eta(x))=1$ that minimize the almost-Riemannian distance from $\mathcal{Z}$. Hence, in these coordinates, $|\partial_t|_{}=1$ and $\delta_{\mathcal{S}}(t,x) = \sqrt{2H(t\eta(x))} = |t|$. In particular, $\mathcal{O} \cap \mathcal{Z} = \{(0,x)\mid x \in \R^{n-1} \}$.

We claim that $\nabla t = \partial_t$. By Cauchy-Schwarz inequality, if $\nabla t$ is not parallel to $\partial_t$, then $1=|g(\partial_t,\nabla t)| < |\nabla t|_{}$ at some point $(t_0,x_0)$. Then, the unit-speed curve $\gamma(s) = e^{s \nabla t/|\nabla t|_{}}(t_0,x_0)$ satisfies
\begin{equation}
\delta_{\mathcal{S}}(\gamma(T)) - \delta_{\mathcal{S}}(\gamma(0)) = \int_0^T dt(\nabla t / |\nabla t|_{}) = \int_0^T |\nabla t |_{} >  T = \ell(\gamma|_{[0,T]}),
\end{equation}
leading to a contradiction, and implying the claim.

Let $X_1 := \partial_t$, $X_2,\ldots,X_n$ be a local generating family for the ARS on $\mathcal{O}$. Indeed, on the regular region $\mathcal{O} \setminus \mathcal{Z}$, they constitute a local orthonormal frame for $g$. In particular, for $i=2,\ldots,n$, we have $dt(X_i) = g(\nabla t, X_i) = g(X_1, X_i) = 0$. By continuity, this holds on the whole neighborhood $\mathcal{O}$, and $X_i = \sum_{j=2}^n a_{ij}(t,x) \partial_{x_j}$. Finally, by definition of $\mathcal{Z}$, we must have $\rank\{X_1,\ldots,X_n\}|_{(0,x)} < n$, that implies $\det(a_{ij})(0,x) =0$.
\end{proof}

\subsection{Essential self-adjointness for almost-Riemannian structures}

By Lemma~\ref{l:relation}, the distance from the metric boundary $\delta$ coincides with the almost-Rieman\-nian distance $\delta_\mathcal{S}$ from $\mathcal{Z}$. The latter, by Lemma~\ref{l:arssmoothness}, is smooth on a set of the form $M_{\varepsilon} = \{0<\delta_{\mathcal{S}} \leq \varepsilon \} \subset M$. Thus, hypothesis $(\star)$ is satisfied on any connected component of $(M,d_g)$, and we can exploit the self-adjointness criterion of Theorem~\ref{t:main}. We state this result as a separate Theorem for ARS.

\begin{theorem}[Quantum completeness criterion for ARS]\label{t:thars}
Let $\mathcal{S}$ be a complete ARS on an $n$-dimensional manifold $N$, equipped with a measure $\omega$, smooth on $N\setminus \mathcal{Z}$. Assume that the singular set $\mathcal{Z}$ is a smooth, embedded and compact hypersurface, with no tangency points. Assume that, for some $\varepsilon>0$, there exists a constant $\kappa\geq 0$ such that, letting $\delta = d_{\mathcal{S}}(\mathcal{Z}, \cdot\,)$, we have
\begin{equation}
\Veff = \left(\frac{\Delta_\omega\delta}{2}\right) + \left(\frac{\Delta_\omega\delta}{2}\right)^\prime \geq \frac{3}{4\delta^2} -\frac{\kappa}{\delta}, \qquad \text{for } 0<\delta \leq \varepsilon.
\end{equation}
Then $\Delta_\omega$ with domain $C^\infty_c(M)$ is essentially self-adjoint in $L^2(M)$, where $M = N \setminus \mathcal{Z}$, or any of its connected components. 

Moreover, when $M$ is relatively compact, the unique self-adjoint extension of $\Delta_\omega$ has compact resolvent. Therefore, its spectrum is discrete and consists of eigenvalues with finite multiplicity.
\end{theorem}

On the Riemannian region $N \setminus \mathcal{Z}$, it is natural to consider the Laplace-Beltrami operator $\Delta = \Delta_{\vol_g}$ with domain $C^\infty_c(N\setminus \mathcal{Z})$. The standing conjecture is that $\Delta$ is essentially self-adjoint \cite{BL-LaplaceBeltrami}, at least when $\mathcal{Z}$ is a compact embedded hypersurface with no tangency points. Hence in the following we fix $\omega = \vol_g$, that is $\Delta_{\omega} = \Delta$. Using Theorem~\ref{t:thars}, we prove this conjecture under a mild regularity assumption.

\subsection{Regular ARS}

Let $(E,\xi,\cdot)$ an ARS on a smooth manifold $N$. Let $X_1,\ldots,X_n\in \Gamma(TN)$ be a local generating family, defined on $\mathcal{O} \subset N$. The singular set $\mathcal{Z} \cap \mathcal{O}$ can be characterized as the zero locus of the smooth map $\det\xi|_\mathcal{O} : \mathcal{O} \to \R$:
\begin{equation}
\det\xi|_\mathcal{O}(q) = \det(X_1,\ldots,X_n)(q).
\end{equation}
This characterization does not depend on the choice of the local family.

\begin{definition}\label{d:regularARS}
We say that a complete almost-Riemannian structure $\mathcal{S}=(E,\xi,\cdot)$ on a smooth manifold $N$ is \emph{regular} if 
\begin{itemize}
\item[$(i)$] there exists $k \in \mathbb{N}$ such that, for all $q \in \mathcal{Z}$ there exists a neighborhood $\mathcal{O}$ of $q$ and a smooth submersion $\psi : \mathcal{O} \to \R$ such that $\det\xi|_\mathcal{O} = \pm \psi^k$; 
\item[$(ii)$] the singular set $\mathcal{Z}$, which is a smooth embedded hypersurface, contains no tangency points.
\end{itemize}
\end{definition}
\begin{rmk}
We stress that the regularity of an ARS is a local property of the morphism $\xi : E \to TN$, and does not depend on the choice of the local family $X_1,\ldots,X_n$ or the scalar product $\cdot$ on the fibers of $E$. In particular, for a different choice, $(i)$ is still satisfied, up to multiplying $\psi$ by a non-vanishing smooth function. 
\end{rmk}
\begin{rmk}
Property $(i)$ is equivalent to the condition that $\mathcal{Z}$ is a smooth, embedded hypersurface and, for $q \in \mathcal{Z}$, the order of $\det\xi$ at $q$ is equal to $k$ (that is, $\det\xi|_{\mathcal{O}} = O(|z|^k)$ for some, and thus any, set of local coordinates centered in $q$).
\end{rmk}
\begin{rmk}
The two conditions are independent. Consider an ARS structure on $\R^2$ given by $E= T\R^2$, where $\xi: T\R^2 \to T\R^2$ is defined by $\xi (\partial_x) = \partial_x$ and $\xi(\partial_y) = f(x,y)\partial_y$. Indeed, $\mathcal{Z}= \{ f(x,y) = 0\}$ and, using the standard Euclidean structure on $E$, we have $\det\xi = f(x,y)$. If $f(x,y) = y-x^2$, then the ARS satisfies $(i)$, but not $(ii)$, since the origin of $\R^2$ is a tangency point. On the other hand, the structure given by the choice $f(x,y) = x(x^2 + y^2)$ satisfies $(ii)$ but not $(i)$.
\end{rmk}

\begin{theorem}[Quantum completeness of regular ARS]\label{t:regularARS}
Consider a regular almost-Rieman\-nian structure on a smooth manifold $N$ with compact singular region $\mathcal{Z}$. Then, the Laplace-Beltrami operator $\Delta$ with domain $C^\infty_c(M)$ is essentially self-adjoint in $L^2(M)$, where $M = N \setminus \mathcal{Z}$ or one of its connected components. Moreover, when $M$ is relatively compact, the unique self-adjoint extension of $\Delta$ has compact resolvent.
\end{theorem}
\begin{proof}
By Lemma \ref{l:localframe}, for any point $q \in \mathcal{Z}$, there exists a neighborhood $\mathcal{O} \subseteq N$ and coordinates $\mathcal{O} \simeq (-\varepsilon,\varepsilon) \times \R^{n-1}$ such that the almost-Riemannian distance from $\mathcal{Z}$ is $\mathcal{\delta}_\mathcal{S}(t,x) = |t|$ and a local generating family of the form
\begin{equation}
X_1 = \partial_t, \qquad X_i = \sum_{j=2}^n a_{ij}(t,x) \partial_{x_j}, \qquad i=2,\ldots,n,
\end{equation}
for smooth functions $a_{ij}(t,x)$, such that $\mathcal{Z} \cap \mathcal{O} = \{ \det(a_{ij})(t,x) = 0 \} = \{(0,x)\mid x \in \R^{n-1}\}$.  Letting $a(t,x) = \det(a_{ij})(t,x)$, and thanks to the regularity assumption, we have
\begin{equation}
a(t,x) = \det(X_1,\ldots,X_n)(t,x) = \pm \psi(t,x)^k,
\end{equation}
where $\psi$ is a smooth submersion. In particular, since $\psi(0,x) =0$, we must have $\partial_t \psi(0,x) \neq 0$. Hence, $\psi(t,x) = t \phi(t,x)$, where $\phi(t,x)$ is some smooth never vanishing function. Then,
\begin{equation}
\vol_g = \sqrt{|g|} dt\, dx= \frac{dt\, dx}{|a(t,x) |} = \frac{dt\, dx}{|t|^k |\phi(t,x)|^k}.
\end{equation}
A straightforward computation using the definition of effective potential yields
\begin{align}
\Veff|_{\mathcal{O} \setminus \mathcal{Z}}  &=  \left(\frac{\Delta |t|}{2}\right)^2 + \partial_t\left(\frac{\Delta |t|}{2}\right) \\
& = \frac{k(k+2)}{4t^2} + \frac{k^2}{2|t|}\frac{\partial_t \phi(t,x)}{\phi(t,x)} + \frac{k(k+2)}{4}\frac{\partial_t \phi(t,x)^2}{\phi(t,x)^2} - \frac{k}{2}\frac{\partial_t^2 \phi(t,x)}{\phi(t,x)}.
\end{align}
Up to restricting to a smaller, compact subset $\mathcal{O}' \simeq [-\varepsilon',\varepsilon'] \times [-1,1]^{n-1}$,
we obtain that $\Veff|_{\mathcal{O'} \setminus \mathcal{Z}} \geq 3/4t^2 - \kappa'/|t|$, for some constant $\kappa'$. By compactness of $\mathcal{Z}$, up to choosing a sufficiently small $\eta$, the set $M_\eta = \{0<\delta \leq \eta \}$ can be covered with a finite number of these coordinate neighborhoods $\mathcal{O}'$, and we obtain the global estimate
\begin{equation}
\Veff \geq \frac{3}{4\delta^2} - \frac{\kappa}{\delta}, \qquad \text{for } 0< \delta\leq \eta.
\end{equation}
We conclude by Theorem~\ref{t:thars}.
\end{proof}

\subsection{Non-regular ARS} \label{s:non-regular}

The conjecture of \cite{BL-LaplaceBeltrami} remains open for general non-regular ARS. We discuss here cases in which $\mathcal{Z}$ is still a compact embedded hypersurface with no tangency points, but the ARS is not regular. 

An important object associated with the singularity of the structure at $\mathcal{Z}$ is the \emph{growth vector} \cite{FJ-book}, which we now define. Consider the sequence of subspaces $\distr_q=\distr_q^1 \subseteq \distr_q^2 \subseteq \ldots \subseteq T_q N$ given by
\begin{equation}
\distr^{i+1}_q : = [\Gamma(\distr^i),\Gamma(\distr)]_q, \qquad i \geq 1.
\end{equation}
Here, with the symbol $\distr^i \subset TN$, we denote the (rank-varying) smooth sub-bundle of $TN$ whose fibers are $\distr^{i+1}_q$, for all $i \geq 1$. By the Lie bracket generating assumption, for any $q \in N$, there exists $m(q)$ such that $\distr^{m(q)}_q = T_q N$.

\begin{definition}\label{d:growth}
Let $k_i(q) : = \dim \distr^i_q$. The \emph{growth vector} at $q$ is the finite sequence
\begin{equation}
\mathcal{G}_q = (k_1(q), \ldots,k_{m(q)}(q)).
\end{equation}
\end{definition}

At \emph{regular} (i.e.\ not singular) points, we have $\mathcal{G}_q = (n)$, that is $m(q) = 1$. The function $q \mapsto m(q)$ is upper semicontinuous with values in a discrete set. Hence, if $\mathcal{Z}$ is compact, then $m(q)$ is bounded. On the other hand, the functions $q \mapsto k_i(q)$ are lower semicontinuous with values in a discrete set. This implies that the set of points $q \in \mathcal{Z}$ such that $\mathcal{G}_q$ is locally constant is open and dense in $\mathcal{Z}$. 

Non-regular ARS can occur both with constant or non-constant growth vector on $\mathcal{Z}$, a feature which was believed to play a role in the problem of essential self-adjointness. For what concerns our theory, there are examples where we can apply Theorem~\ref{t:thars} and examples where we cannot.
\begin{example}[Non-regular ARS, non-constant growth vector]\label{ex:first}
Let $N= \R_t \times \mathbb{T}^{n-1}_x$, with $n \geq 2$. Here $x$ denotes  the coordinate on the torus $\mathbb{T}^{n-1}$ and $t \in \R$.  Let $\ell \in \mathbb{N}$. Consider the ARS given by the local generating family
\begin{equation}
X_0 = \partial_t, \qquad X_i =  t(t^{2\ell}+ f(x))\partial_{x_i}, \qquad i=1,\ldots,n-1,
\end{equation}
for some smooth function $f \geq 0$, attaining the value zero on a proper, non-empty subset. The almost-Riemannian distance from the singular set is $\delta_{\mathcal{S}}(t,x) = |t|$ and, by Lemma~\ref{l:relation}, coincides with the distance from the metric boundary of the Riemannian structure induced on $M = N \setminus \mathcal{Z}$. The restriction of the growth vector of this ARS structure to $\mathcal{Z}$ is 
\begin{equation}
\mathcal{G}_{(0,x)} = \begin{cases}
(1,\ldots,1,n) & f(x) =0, \\
(1,n)  & f(x) \neq 0,
\end{cases}
\end{equation}
where, in the first line, there are $2\ell + 1$ repeated ones. In particular, we have
\begin{equation}
\det\xi = \det(X_0,\ldots,X_{n-1}) = [t(t^{2\ell}+f(x))]^{n-1},
\end{equation}
and straightforward computations show that $\det\xi$ cannot be written as the power of a submersion, hence the ARS is not regular. For $\vol_g$, we obtain
\begin{equation}
\vol_g =  \sqrt{|g|}\, dt\, dx= \frac{dt\, dx}{[|t|(t^{2\ell}+f(x))]^{n-1}} .
\end{equation}
In particular, the order of $\vol_g$ is not constant close to $\mathcal{Z}$:
\begin{equation}
\vol_g = O(|t|^{a(x)})\, dt \, dx, \qquad a(x) = -\begin{cases}
(n-1)(2\ell +1)  & f(x)=0, \\
(n-1) &  f(x) \neq 0.
\end{cases}
\end{equation}
We are in the case of Example~\ref{ex:non-constant}, and
\begin{equation}
\Veff = \frac{a(x)(a(x)-2)}{4t^2} + R(t,x) \geq \frac{3}{4t^2} + R(t,x),
\end{equation}
where $R(t,x)$ is a remainder term of the form
\begin{equation}
R(t,x) = \begin{cases}
0 &f(x) =0, \\
\frac{\ell(n-1)t^{2\ell-2}[(\ell(n-1)+n)t^{2\ell} + (n-2\ell)f(x)]}{[t^{2\ell} + f(x)]^2} & f(x) \neq 0. 
\end{cases}
\end{equation}
The behavior of $\Veff$ depends in a crucial way on the parameters.
\begin{enumerate}
\item If $\ell \leq n/2$, then $R(t,x) \geq 0$. In particular $\Veff \geq 3/4t^2$, and the Laplace-Beltrami is essentially self-adjoint, for any value of $n$, thanks to Theorem~\ref{t:thars}.
\item If $\ell > n/2$, then along any sequence $(t_i,x_i)$ such that $t_i=1/i$ and $f(x_i) = 1/i^{2\ell}$, we have $t_i R(t_i,x_i) \to -\infty$. Hence we cannot apply Theorem~\ref{t:thars}. We do not know whether the Laplace-Beltrami operator in the Riemannian region is essentially self-adjoint or not, even in the simple case $n=2$ and $\ell=2$.
\end{enumerate}
\end{example}

In the particular case of $2$-dimensional ARS, the low dimension implies the following equivalence: the ARS is regular if and only if the growth vector $\mathcal{G}$ is constant on $\mathcal{Z}$. However, in general, the following examples show that these two conditions are independent.

\begin{example}[Non-regular ARS, constant growth vector]\label{ex:second}
Let $N= \R_t \times \mathbb{T}_x^{n-1} \times \mathbb{T}^{n-1}_y$, with $n \geq 2$. Let $\ell \in \mathbb{N}$, $\ell \geq1$. Consider the ARS given by the local generating family
\begin{equation}
X_0 = \partial_t, \qquad X_i =  t(t^{2\ell}+ f(x))\partial_{x_i}, \qquad Y_i = \partial_{y_i} + t \partial_{x_i}, \qquad i=1,\ldots,n-1,
\end{equation}
for some smooth function $f \geq 0$, attaining the value zero on a proper, non-empty subset. The restriction of the growth vector to $\mathcal{Z}$ is constant:
\begin{equation}
\mathcal{G}_{q} = (n,2n-1), \qquad \forall q \in \mathcal{Z}.
\end{equation}
With computations similar to Example~\ref{ex:first}, we obtain the same effective potential replacing $n-1 \mapsto 2(n-1)$. Thus we can apply Theorem~\ref{t:thars} if and only if $\ell \leq (2n-1)/2$.
\end{example}

\begin{example}[Regular ARS, non-constant growth vector]\label{ex:third}
Let $N= \R_t \times \mathbb{T}_{x,y,z}^{3}$. Consider the ARS given by the local generating family
\begin{equation}
X_1 = \partial_t, \qquad X_2 =  \partial_x, \qquad X_3 = \partial_y + x^2 \partial_z, \qquad X_4 = t^2 \partial_z.
\end{equation}
Observe that $\mathcal{Z} = \{0\} \times \mathbb{T}^3$. We have the following Lie brackets,
\begin{equation}
[X_2,X_3] = 2x \partial_z, \qquad [X_2,[X_2,X_3]]]  = 2\partial_z.
\end{equation}
Thus, the growth vector on $\mathcal{Z}$ is non-constant:
\begin{equation}
\mathcal{G}_{(0,x,y,z)} = \begin{cases}
(3,4) & x \neq 0, \\
(3,3,4) & x = 0.
\end{cases}
\end{equation}
On the other hand, the ARS structure is regular, since
\begin{equation}
\det(X_1,X_2,X_3,X_4) = t^2.
\end{equation}
In particular, by Theorem~\ref{t:regularARS}, the Laplace-Beltrami operator is essentially self-adjoint.
\end{example}
In light of the above examples, we do not expect the regularity of the growth vector on $\mathcal{Z}$ to play any role in the essential self-adjointness of $\Delta$, at least as far as the sufficient condition of Theorem~\ref{t:thars} is concerned. % almost riemannian

\subsection*{Acknowledgments}
\thanks{This research has  been supported by the European Research Council, ERC StG 2009 ``GeCoMethods'', contract n. 239748, and by the RIP program of the Institut Henri Poincar\'e, during which this project begun. The first and second authors were partially supported by the Grant ANR-15-CE40-0018 of the ANR, by the iCODE institute (research project of the Idex Paris-Saclay), and by the SMAI project ``BOUM''. The first author was also partly supported by the ANR grant NS-LBR. ANR-13-JS01-0003-01. This research benefited from the support of the ``FMJH Program Gaspard Monge in optimization and operation research'' and from the support to this program from EDF.}

\bibliographystyle{abbrv}
\bibliography{ARSelfAdjointness}

\end{document}